\documentclass[VANCOUVER,STIX1COL]{WileyNJD-v2}

\articletype{Article Type}%
\usepackage{graphicx} 
\usepackage{tikz}
\usetikzlibrary{arrows.meta}
\usetikzlibrary{decorations.text}
\usetikzlibrary{decorations.pathreplacing}
\usepackage{subcaption}
\usepackage{diagbox}
\usepackage{mathtools}
\usepackage{booktabs}
\usepackage{multirow}
\usepackage{comment}
\usepackage{cleveref}

\newcommand{\Rea}{\operatorname{Re}}
\newcommand{\Ima}{\operatorname{Im}}

\received{1 Jan 1900 }
\revised{ *** }
\accepted{ *** }

\begin{document}

\title{Weighted Relaxation for Multigrid Reduction in Time}

\author[1]{Masumi Sugiyama*}

\author[2]{Jacob B. Schroder}

\author[3]{Ben S. Southworth}

\author[4]{Stephanie Friedhoff}

\authormark{Sugiyama \textsc{et al.}}

\address[1]{\orgdiv{Dept. of Mathematics}, \orgname{University of Tennessee at Chattanooga}, \orgaddress{\state{Tennessee}, \country{USA}}}

\address[2]{\orgdiv{Dept. of Mathematics and Statistics}, \orgname{University of New Mexico}, \orgaddress{\state{New Mexico}, \country{USA}}}

\address[3]{\orgdiv{Theoretical Division}, \orgname{Los Alamos National Laboratory}, \orgaddress{\state{New Mexico}, \country{USA}}}

\address[4]{\orgdiv{Dept. of Mathematics}, \orgname{University of Wuppertal}, \orgaddress{\country{Germany}}}

\corres{*Masumi Sugiyama, Dept. of Mathematics, University of Tennessee at Chattanooga, Chattanooga, TN 37403, USA. \email{hgh889@mocs.utc.edu}}


\abstract[Summary]{Based on current trends in computer architectures, faster compute speeds must come from increased parallelism rather than increased clock speeds, which are currently stagnate. This situation has created the well-known bottleneck for sequential time-integration, where each individual time-value (i.e., time-step) is computed sequentially. One approach to alleviate this and achieve parallelism in time is with multigrid. In this work, we consider multigrid-reduction-in-time (MGRIT), a multilevel method applied to the time dimension that computes multiple time-steps in parallel. Like all multigrid methods, MGRIT relies on the complementary relationship between relaxation on a fine-grid and a correction from the coarse grid to solve the problem. All current MGRIT implementations are based on unweighted-Jacobi relaxation; here we introduce the concept of weighted relaxation to MGRIT. We derive new convergence bounds for weighted relaxation, and use this analysis to guide the selection of relaxation weights.
Numerical results then demonstrate that non-unitary relaxation weights consistently yield faster convergence rates and lower iteration counts for MGRIT when compared with unweighted relaxation. In most cases, weighted relaxation yields a 10\%--20\% saving in iterations.  For A-stable integration schemes, results also illustrate that under-relaxation can \emph{restore convergence} in some cases where unweighted relaxation is not convergent.
}

\keywords{parallel-in-time, multigrid, multigrid-reduction-in-time, weighted relaxation, polynomial relaxation}


\maketitle

\section{Introduction}
\label{sec:intro}

Based on current trends in computer architectures, faster compute speeds must come from increased parallelism rather than increased clock speeds, which are stagnate. This situation has created a bottleneck for sequential time-integration \cite{Ga2015, Fa2014, OnSc2019}, where each individual time-value (i.e., time-step) is computed sequentially. 
One approach to alleviate this is through parallelism in the time dimension, which goes back at least to Nievergelt \cite{Ni1964} in 1964.
For an introduction to parallel-in-time methods, see the review papers \cite{Ga2015, OnSc2019}, which give an overview of various approaches such as multiple shooting, waveform relaxation, domain decomposition, multigrid, and direct parallel-in-time methods.

In this work, we choose multigrid for parallelism in time for the same reasons that multigrid is often the method of choice for solving spatial problems \cite{TrOo2001,BrHeMc2000}, i.e., a well-designed multigrid solver is an optimal method.  In particular, we consider the multigrid-reduction-in-time (MGRIT) method \cite{Fa2014}, which has been applied in numerous settings, e.g., for nonlinear parabolic problems \cite{Fa2016}, compressible and incompressible Navier-Stokes \cite{FalKatz2014,ChrGaGuFaSc2019b}, elasticity \cite{HeNoRoSc2017}, power-grid systems \cite{LeFaWoTo2016,GuFaToWoSc2020}, eddy current \cite{FrHaKuSc2018,BoFrHaSc2020}, machine learning \cite{GuRuScCyGa2018,CyGuSc2019}, and more \cite{OnSc2019}.
However, we note that there exist other powerful multigrid-like parallel-in-time methods such as the popular parareal 
\cite{LiMaTu2001} and parallel full approximation scheme in space and time (PFASST) \cite{MiWi2008,EmMi2012,Mi2010} methods. Parareal can be viewed as a two-level multigrid reduction method that coarsens in time \cite{gander2007analysis}. PFASST can also be viewed as a multigrid 
method in time that utilizes a deferred correction strategy to compute multiple time-steps in parallel \cite{bolten2017multigrid}.   Unlike parareal, MGRIT is a full multilevel method applied to the time dimension, which allows for optimal scaling with respect to problem size.  In contrast, for the two-level case, the coarsest temporal grid typically grows with problem size, yielding a potentially fast, but non-optimal method.

Like all multigrid methods, MGRIT relies on the complementary relationship between relaxation on a fine-grid, typically unweighted (block) Jacobi, and a correction from the coarse grid to solve the problem. 
In this work, we extend the use of weighted relaxation in multigrid \cite{AdBrHuTu2003,BaFaKoYa2011,TrOo2001,BrHeMc2000} to MGRIT, and analyze and select effective relaxation weights. With an appropriate choice of weight, MGRIT with weighted relaxation consistently offers faster convergence when compared with standard (unweighted) MGRIT, at almost no additional computational work\footnote{Only one additional vector addition is performed.}. \Cref{sec:methods} introduces a framework for weighted relaxation in MGRIT, and derives a new convergence analysis for linear two-grid MGRIT with degree-1 weighted-Jacobi relaxation. The theory is then verified with simple numerical examples in \Cref{sec:verify}, and the utility of weighted relaxation is demonstrated on more complex problems in \Cref{sec:results}, including a 2D advection-diffusion problem and a 2D nonlinear eddy current problem. The new method consistently offers a 10--20\% savings in iterations over standard unweighted MGRIT, and in some cases, (particularly A-stable integration schemes) yields convergence several times faster. Additional experiments are provided in the Supplemental Materials \Cref{app2}, exploring the effects of level-dependent relaxation weights for multilevel solvers and degree-2 weighted-Jacobi.

\section{Multigrid-Reduction-in-Time (MGRIT) and weighted-Jacobi}
\label{sec:methods}
\subsection{Two-level MGRIT method}

This section derives the error-propagation operator for two-level linear MGRIT with weighted relaxation. Then, two-level convergence bounds are derived as a function of relaxation weight, providing insight on choosing the weight in practice. Although MGRIT uses full approximation storage (FAS) nonlinear multigrid cycling \cite{Brandt_1977} to solve nonlinear problems, the linear two-grid setting makes analysis more tractable (e.g.,  \cite{FriMac2015, Do2016, HeSoNoRoFaSc2018, So2019, FrSo2020}), and MGRIT behavior for linear problems is often indicative of MGRIT behavior for related nonlinear problems \cite{Do2016}. Thus, consider a linear system of ordinary differential equations (ODEs) with $N_x$ spatial degrees of freedom,
\begin{equation}\label{ODE}
\frac{d \mathbf{u}}{dt} = G\mathbf{u}(t) + \mathbf{g}(t), \hspace{10pt} \mathbf{u}(0)=\mathbf{g}_0, \hspace{10pt} t \in [0, T],
\end{equation}
where $\mathbf{u} \in \mathbb{R}^{N_x}$ and $G\in\mathbb{R}^{N_x\times N_x}$ is a linear operator in space. For simplicity, define a uniform temporal grid as $t_j = j \delta_t$, for $j=0,1,..,N_t-1$ where $N_t$ refers to the number of points in time, with constant spacing $\delta_t = T / (N_t - 1) > 0$. Let $\mathbf{u}_{j}$ be an approximation to $\mathbf{u}(t_j)$ for $j=1,2,..,N_t-1$ and $\mathbf{u}_0 = \mathbf{u}(0)$. Then, a general one-step time discretization for (\ref{ODE}) is defined as
\begin{equation}\label{time step eq}
\begin{aligned}
\mathbf{u}_0 &= \mathbf{g}_0,\\
\mathbf{u}_j &= \Phi \mathbf{u}_{j-1} + \mathbf{g}_j, \hspace{10pt} j=1,2,...,N_t-1 ,
\end{aligned}
\end{equation}
{where $\Phi$ is a one-step integration operator and $\mathbf{g}_j = \mathbf{g}(t_j)$.}
The solution to \eqref{time step eq} for all time points is equivalent to solving the system of equations
\begin{equation}\label{time step eq matrix}
\mathbf{Au} \coloneqq 
\begin{bmatrix} 
I &  &  &  & \\
-\Phi & I &  &  &\\
& \ddots &  \ddots&  & \\
& & -\Phi & I
\end{bmatrix}
\begin{bmatrix}
\mathbf{u_0}\\ \mathbf{u_1} \\ \vdots \\ \mathbf{u_{N_t-1}}
\end{bmatrix}=
\begin{bmatrix}
\mathbf{g_0}\\ \mathbf{g_1} \\ \vdots \\ \mathbf{g_{N_t-1}}
\end{bmatrix} = \mathbf{g}.
\end{equation}
While sequential time-stepping solves \eqref{time step eq matrix} directly with forward-substitution, MGRIT solves (\ref{time step eq matrix}) iteratively by combining a block Jacobi relaxation with error corrections computed on a coarse-grid. Let the coarse temporal grid be $T_i = i \delta_T$, for $i=0,1,...,N_{T-1}$ and $N_T = (N_t-1)/m+1$, which corresponds to a positive integer coarsening factor $m$ and constant spacing $\delta_T = m \delta_t$. (Without loss of generality, we assume that $N_t-1$ divides evenly by $m$ in this description.) The original grid of points $\{t_j\}$ is then partitioned into C-points given by the set of coarse grid points $\{T_i\}$, and F-points given by $\{t_i\} \setminus \{T_i\}$ (see Figure \ref{mgrit_grid}). These C-points then induce a new coarser time-grid, with equivalent time-propagation problem
\begin{equation}\label{time step eq coarse}
\begin{aligned}
\mathbf{u}_0 &= \mathbf{g}_0 \\
\mathbf{u}_{km} &= \Phi^m \mathbf{u}_{(k-1)m} + \tilde{\mathbf{g}}_{km}, \hspace{10pt} k=1,2,...,N_T-1,
\end{aligned}
\end{equation}
where $\tilde{\mathbf{g}}_{km} = \mathbf{g}_{km} + \Phi \mathbf{g}_{km-1} + \cdots + \Phi^{m-1} \mathbf{g}_{(k-1)m+1}$.
The solution to \eqref{time step eq coarse} is equivalent to solving the coarse system of equations
\begin{equation}\label{coarse time step eq matrix}
\mathbf{A_{\triangle} u_{\triangle}} \coloneqq  
\begin{bmatrix} 
I &  &  &  & \\
-\Phi^m & I &  &  &\\
& \ddots &  \ddots&  & \\
& & -\Phi^m & I
\end{bmatrix}
\begin{bmatrix}
\mathbf{u}_0\\ \mathbf{u}_m \\ \vdots \\ \mathbf{u}_{(N_T-1) m}
\end{bmatrix}=
\begin{bmatrix}
\mathbf{g}_0\\ \tilde{\mathbf{g}}_m \\ \vdots \\ \tilde{\mathbf{g}}_{(N_T-1) m}
\end{bmatrix} = \mathbf{g}_{\triangle},
\end{equation}
\normalsize
 where $\mathbf{A_{\triangle}}$ has $N_T$ block rows and block columns. Unfortunately, solving equation (\ref{coarse time step eq matrix}) is as expensive as solving equation (\ref{time step eq matrix}) because of the $\Phi^m$ operator.  Thus, $\Phi^m$ is usually replaced with a cheap approximation $\Phi_{\triangle}$, which in turn induces  a new operator on the coarse-grid, $\mathbf{B}_{\triangle} \approx \mathbf{A}_{\triangle}$.  The operator $\mathbf{B}_\triangle$ has the exact same structure as $\mathbf{A}_\triangle$, only the $\Phi^m$ has been replaced by $\Phi_\triangle$.

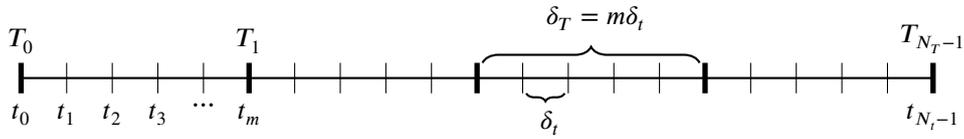
\begin{figure}[t]
	\centering
	\begin{tikzpicture}[xscale = 1.2]
	\draw [thick] (-4.5,0) -- (5.5,0);
	
	\draw[line width=2pt] (-4.5,-0.2) -- (-4.5,0.2);
	\draw (-4.0,-0.2) -- (-4.0,0.2);
	\draw (-3.5,-0.2) -- (-3.5,0.2);
	\draw (-3.0,-0.2) -- (-3.0,0.2);
	\draw (-2.5,-0.2) -- (-2.5,0.2);
	\draw[line width=2pt] (-2.0,-0.2) -- (-2.0,0.2);
	\draw (-1.5,-0.2) -- (-1.5,0.2);
	\draw (-1.0,-0.2) -- (-1.0,0.2); 
	\draw (-0.5,-0.2) -- (-0.5,0.2); 
	\draw (0,-0.2) -- (0,0.2); 
	\draw[line width=2pt] (0.5,-0.2) -- (0.5,0.2); 
	\draw (1.0,-0.2) -- (1.0,0.2);
	\draw (1.5,-0.2) -- (1.5,0.2); 
	\draw (2.0,-0.2) -- (2.0,0.2);
	\draw (2.5,-0.2) -- (2.5,0.2); 
	\draw[line width=2pt] (3.0,-0.2) -- (3.0,0.2); 
	\draw (3.5,-0.2) -- (3.5,0.2); 
	\draw (4.0,-0.2) -- (4.0,0.2); 
	\draw (4.5,-0.2) -- (4.5,0.2); 
	\draw (5.0,-0.2) -- (5.0,0.2); 
	\draw[line width=2pt] (5.5, -0.2) -- (5.5, 0.2);
	
	\node [below] at (-4.5, -0.2) {$t_0$};
	\node [above] at (-4.5, 0.2) {$T_0$};
	\node [below] at (-4.0, -0.2) {$t_1$};
	\node [below] at (-3.5, -0.2) {$t_2$};
	\node [below] at (-3.0, -0.2) {$t_3$};
	\node [below] at (-2.5, -0.2) {$\cdots$};
	\node [below] at (-2.0, -0.2) {$t_m$};
	\node [above] at (-2.0, 0.2) {$T_1$};
	\node [below] at (5.5, -0.2) {$t_{N_t-1}$};
	\node [above] at (5.5, 0.2) {$T_{N_T-1}$};
	
	\draw [thick,decorate,decoration={brace,amplitude=6pt,raise=0pt}] (0.56, 0.2) -- (2.94, 0.2);
	\node [above] at (1.8, 0.5) {$\delta_T = m \delta_t$};
	\draw [thick,decorate,decoration={brace,amplitude=3pt,raise=0pt, mirror}] (1.04, -0.2) -- (1.46, -0.2);
	\node [below] at (1.3, -0.3) {$\delta_t$};
	\end{tikzpicture}
	\caption{Uniform fine and coarse time-grid corresponding to coarsening factor $m$. The $T_i$ are the C-points and form the coarse-grid, while the small hashmarks are F-points.  Together, the F- and C-points form the fine-grid $\{ t_j \}$.}
	\label{mgrit_grid}
\end{figure}

\begin{figure}[t] 
	\centering 
	\begin{subfigure}[h]{0.4\linewidth}
		\begin{tikzpicture}[L/.style={very thick, draw=#1, looseness=2, ->}]
		\draw [thick] (-4.5,0) -- (-0.5,0);
		\foreach \x in {1,2,...,12}{\coordinate (x\x) at (\x,0);}
		
		\draw[line width=2pt] (-4.5,-0.2) -- (-4.5,0.2);
		\draw (-4.0,-0.2) -- (-4.0,0.2);
		\draw (-3.5,-0.2) -- (-3.5,0.2);
		\draw (-3.0,-0.2) -- (-3.0,0.2);
		\draw[line width=2pt] (-2.5,-0.2) -- (-2.5,0.2);
		\draw (-2.0,-0.2) -- (-2.0,0.2);
		\draw (-1.5,-0.2) -- (-1.5,0.2);
		\draw (-1.0,-0.2) -- (-1.0,0.2); 
		\draw [line width=2pt] (-0.5,-0.2) -- (-0.5,0.2);  
		
		\draw[L=orange] (-4.5,0) to [out=90, in=90] (-4.0,0);
		\draw[L=orange] (-4.0,0) to [out=90, in=90] (-3.5,0);
		\draw[L=orange] (-3.5,0) to [out=90, in=90] (-3.0,0);
		
		\draw[L=orange] (-2.5,0) to [out=90, in=90] (-2.0,0);
		\draw[L=orange] (-2.0,0) to [out=90, in=90] (-1.5,0);
		\draw[L=orange] (-1.5,0) to [out=90, in=90] (-1.0,0);
		
	    \node [below] at (-4.5, -0.2) {$T_{i-1}$};
	    \node [below] at (-2.5, -0.2) {$T_{i}$};
	    \node [below] at (-0.5, -0.2) {$T_{i+1}$};
	    
		\end{tikzpicture}
		\caption{\normalsize F-relaxation}
	\end{subfigure}
	\begin{subfigure}[h]{0.4\linewidth}
		\begin{tikzpicture}[
		L/.style={very thick, draw=#1, looseness=2, ->}
		]
		\draw [thick] (-4.5,0) -- (-0.5,0);
		\foreach \x in {1,2,...,12}{\coordinate (x\x) at (\x,0);}
		
		\draw[line width=2pt] (-4.5,-0.2) -- (-4.5,0.2);
		\draw (-4.0,-0.2) -- (-4.0,0.2);
		\draw (-3.5,-0.2) -- (-3.5,0.2);
		\draw (-3.0,-0.2) -- (-3.0,0.2);
		\draw[line width=2pt] (-2.5,-0.2) -- (-2.5,0.2);
		\draw (-2.0,-0.2) -- (-2.0,0.2);
		\draw (-1.5,-0.2) -- (-1.5,0.2);
		\draw (-1.0,-0.2) -- (-1.0,0.2); 
		\draw [line width=2pt] (-0.5,-0.2) -- (-0.5,0.2);  
		
		\draw[L=blue] (-3.0,0) to [out=90, in=90] (-2.5,0);
		\draw[L=blue] (-1.0,0) to [out=90, in=90] (-0.5,0);
		
	    \node [below] at (-4.5, -0.2) {$T_{i-1}$};
	    \node [below] at (-2.5, -0.2) {$T_{i}$};
	    \node [below] at (-0.5, -0.2) {$T_{i+1}$};
	    
		\end{tikzpicture}
		\caption{\normalsize C-relaxation}
	\end{subfigure}
	\caption{Schematic view of the action of (a) F-relaxation and (b) C-relaxation with a coarsening factor of $m=4$.}
	\label{F- and C-relaxation}
\end{figure}
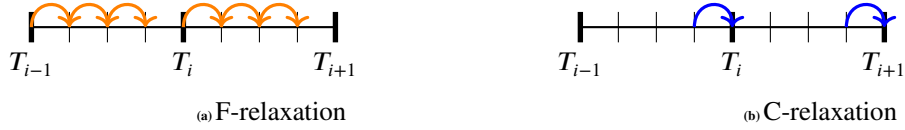

With the partition of F- and C-points as depicted in Figure \ref{mgrit_grid}, there are two fundamental types of relaxation: F- and C-relaxation. F-relaxation updates the F-point values based on the C-point values, i.e., one F-sweep updates each interval of F-points with 
\begin{equation}
 \mathbf{u}_i = \Phi \mathbf{u}_{i-1} + \mathbf{g}_i \quad \mbox{for } i = (km+1) \dots ((k+1)m-1),
 \label{eqn:Frelax}
 \end{equation}
and $k$ is the F-interval index from $0$ to $N_T-2$.
Similarly, C-relaxation updates each C-point value based on the preceding F-point value, i.e., the index $i$ becomes $km$ in equation \eqref{eqn:Frelax}. Each interval of F-points $(T_{i-1}, T_{i})$ for $i = 1,..., N_T-1$ can be updated simultaneously in parallel, and each C-point can also be updated simultaneously in parallel. Figure \ref{F- and C-relaxation} illustrates the action of these relaxations in parallel.
One application of F-relaxation followed by a C-relaxation updates each $\mathbf{u}_{km}$ based on $\mathbf{u}_{(k-1)m}$, which computes $\Phi^m$ applied to $\mathbf{u}_{(k-1)m}$ for $k = 1,..., N_T-1$. This FC-sweep corresponds to a block Jacobi iteration on the coarse-grid with $\mathbf{A}_{\triangle}$. Letting $k$ denote the current relaxation iteration, this block Jacobi scheme can be written as
\begin{equation} \label{block jacobi F- and C- relaxation}
\begin{aligned}
{\mathbf{u}_{\triangle}^{(k + 1)}} & {= \mathbf{u}^{(k)}_{\triangle} + D_\triangle^{-1} (\mathbf{g}_{\triangle} - \mathbf{A}_{\triangle} \mathbf{u}^{(k)}_{\triangle})} \\
&= \begin{bmatrix} \mathbf{u}_0^{(k)}\\ \mathbf{u}_m^{(k)} 
\\ \vdots \\ \mathbf{u}_{(N_T-1) m}^{(k)} \end{bmatrix} + D_\triangle^{-1} 
\begin{bmatrix}
\mathbf{g}_0 - \mathbf{u}_0^{(k)} \\ \tilde{\mathbf{g}}_m + \Phi^m \mathbf{u}_0^{(k)} - \mathbf{u}_m^{(k)} \\ \vdots \\ \tilde{\mathbf{g}}_{(N_T-1) m} + \Phi^m \mathbf{u}_{(N_T - 2)m}^{(k)} - \mathbf{u}_{(N_T-1) m}^{(k)}
\end{bmatrix}
&= \begin{bmatrix} \mathbf{g}_0 \\ \Phi^m \mathbf{u}_0^{(k)} + \tilde{\mathbf{g}}_m \\ \vdots \\ \Phi^m \mathbf{u}_{(N_T - 2)m}^{(k)} + \tilde{\mathbf{g}}_{(N_T-1) m}\end{bmatrix},
\end{aligned}
\end{equation}
where $D_\triangle$ is the diagonal of $A_\Delta$ and equal to the identity.
The MGRIT algorithm performs either an F-relaxation or an FCF-relaxation, which consists of the initial F-relaxation, a C-relaxation, and a second F-relaxation.

\subsubsection{Weighted-Jacobi variant of FCF-relaxation}
Here we introduce a \emph{weighted} Jacobi relaxation to the MGRIT framework. Weighted-Jacobi relaxation with weight $\omega_C > 0$ applied to \eqref{block jacobi F- and C- relaxation} takes the form
\begin{align}  \label{weighted jacobi eq}
\begin{split}
\mathbf{u}_\triangle^{(k + 1)} &= \omega_C \{(I - D_\triangle^{-1}A_\triangle) \mathbf{u}_\triangle^{(k)} + D_\triangle^{-1} \mathbf{g}_\triangle\} + (1 - \omega_C) \mathbf{u}_\triangle^{(k)}, \hspace{10pt} k = 0,1,2,... 
\end{split}
\end{align}
We use $\omega_C$ to denote the weight in \eqref{weighted jacobi eq}, because it will be shown that \eqref{weighted jacobi eq} is equivalent to applying a relaxation weight only during the C-relaxation step of an FC-sweep. Since the standard MGRIT FC-sweep corresponds to the block Jacobi method \eqref{block jacobi F- and C- relaxation}, it
is thus natural to instead consider the weighted variant \eqref{weighted jacobi eq} inside of MGRIT.  

In general, weighted relaxation has improved convergence for spatial multigrid methods applied to a variety of problems {\cite{AdBrHuTu2003,BaFaKoYa2011,TrOo2001,BrHeMc2000}}, and so the remainder of this paper explores the application of weighted-Jacobi \eqref{weighted jacobi eq} in MGRIT.
Regarding notation, the subscript $_F$ indicates the relaxation weight $\omega_F$ for F-relaxation, and subscript $_C$ indicates the weight $\omega_C$ for C-relaxation. Degree-two weighted-Jacobi will refer to two successive iterations of \eqref{weighted jacobi eq}, possibly with different weights.  The weight for the first C-relaxation{, for example,} is denoted $\omega_C$, while the weight for the second is denoted $\omega_{CC}$.
It is called degree-two, because the resulting update to
$\mathbf{u}_\triangle$ corresponds to a degree-two polynomial in $A_\triangle$.

\subsection{Convergence estimate for MGRIT with weighted-Jacobi relaxation}

We now extend existing linear two-level MGRIT convergence bounds \cite{Do2016,So2019} to account for the effects of weighted-Jacobi relaxation. 

\subsubsection{MGRIT error propagator for unweighted FCF-relaxation}

Let the fine-grid operator $\mathbf{A}$ in (\ref{time step eq matrix}) be reordered so that F-points appear first and C-points second.  Then by using the subscripts $F$ and $C$ to indicate the two sets of points, we have
$$\mathbf{A}=
\begin{bmatrix}
A_{FF} & A_{FC} \\
A_{CF} & A_{CC}\\
\end{bmatrix}.$$
Define the ideal interpolation operator $P\;$\footnote{$P$ is \textit{ideal} because if an exact solution is available at C-points, then
multiplication by $P$ plus a right-hand-side contribution will yield the exact solution at all C- and F-points.},
restriction by injection $R_I$, and a map to F-points $S$, respectively, as
$$ P \coloneqq \begin{bmatrix} -A_{FF}^{-1}A_{FC} \\
I_C \end{bmatrix}, \hspace{10pt}
R_I \coloneqq \begin{bmatrix} 0 & I_C \end{bmatrix}, \hspace{10pt}
S \coloneqq \begin{bmatrix}
I_F\\ 0 \end{bmatrix}.$$
{From \cite{Fa2014}, the two-level error propagator for linear MGRIT with unweighted FCF-relaxation is then given by}
\begin{equation} \label{two level error FCF}
(I - P B_{\triangle}^{-1} R_I A)P(I - A_{\triangle})R_I = P(I - B_{\triangle}^{-1} A_{\triangle}) (I - A_{\triangle})R_I.
\end{equation}

\subsubsection{Two-level error propagator for weighted C-relaxation}
Weighted-Jacobi for F-relaxation using the same structure as (\ref{weighted jacobi eq}) can be written as
\begin{align}
\begin{split}
\mathbf{u}^{k+1} &= \omega_F \{(I-S(S^TAS)^{-1}S^TA)\mathbf{u}^k + D^{-1}\mathbf{g}\} + (1-\omega_F) \mathbf{u}^k \\
&= (I - \omega_F S(S^TAS)^{-1}S^TA)\mathbf{u}^k + \omega_F D^{-1}\mathbf{g},
\end{split}
\end{align}
where the first term (without $\mathbf{g}$) is the error propagator.  
Similarly, weighted-Jacobi for C-relaxation can be written as
\begin{equation} \label{C-realx}
\mathbf{u}^{k+1} = (I - \omega_C R^T_I(R_IAR^T_I)^{-1}R_IA) \mathbf{u}^k + D^{-1}\mathbf{g},
\end{equation}
where the first term (without $\mathbf{g}$) is the error propagator.
Hence, the error propagator of FCF-relaxation with weighted-Jacobi is given by
the product of F-, C-, and F-relaxation error-propagators:
\begin{equation}
   (I-\omega_{FF} S(S^TAS)^{-1}S^TA)(I-\omega_C R^T_I(R_IAR^T_I)^{-1}R_IA)(I-\omega_F S(S^TAS)^{-1}S^TA),
\end{equation}
where $\omega_{FF}$ denotes the weight for the second F-relaxation. Despite the above generality, moving forward we only consider $\omega_F = \omega_{FF} = 1.0$.
If $\omega_{FF} \neq 1$, then MGRIT would no longer be an approximate reduction method.
In other words, if the exact solution were given at C-points, the final F-relax using $\omega_{FF}\neq 1$ would no longer be guaranteed to yield the exact solution at F-points.
We note that experiments also indicated $\omega_{FF} = 1$ performs best on model heat and advection problems. 
Similarly, letting $\omega_F \neq 1$ would restrict an inexact residual to the coarse grid problem, deviating from the principle of reduction methods. 

Thus, with this simplification, the error propagator for C-weighted FCF-relaxation takes the following block $2 \times 2$ form:
\begin{subequations}
\begin{align} \label{eq:error propagator}
&(I-S(S^TAS)^{-1}S^TA)(I-\omega_C R^T_I(R_IAR^T_I)^{-1}R_IA)(I- S(S^TAS)^{-1}S^TA)\\
& = \Bigg(I - \begin{bmatrix}
I_f & A_{ff}^{-1}A_{fc} \\
0 & 0\\
\end{bmatrix}\Bigg)
\Bigg(I - \omega_C \begin{bmatrix}
0 & 0 \\
A_{cc}^{-1}A_{cf} & I_c\\
\end{bmatrix}\Bigg)
\Bigg(I - \begin{bmatrix}
I_f & A_{ff}^{-1}A_{fc} \\
0 & 0\\
\end{bmatrix}\Bigg)\\
& = \begin{bmatrix}
0 & -A_{ff}^{-1}A_{fc} \{I_c - \omega_C A_{cc}^{-1}(A_{cc} - A_{cf}A_{ff}^{-1}A_{fc})\}\\
0 & I_c - \omega_C A_{cc}^{-1} (A_{cc} - A_{cf}A_{ff}^{-1}A_{fc})
\end{bmatrix}\\
& = \begin{bmatrix}
- A_{ff}^{-1}A_{fc}\\
I_c
\end{bmatrix}
\begin{bmatrix}
I_c - \omega_C A_{cc}^{-1} (A_{cc} - A_{cf}A_{ff}^{-1}A_{fc})
\end{bmatrix}
\begin{bmatrix}
0 & I_c
\end{bmatrix}\\
& = P(I - \omega_C \mathbf{A}_{\triangle})R_I.
\label{eq:wFCF}
\end{align}
\end{subequations}

Next, we take the two-level MGRIT error propagator with FCF-relaxation (\ref{two level error FCF}) and 
substitute in the new weighted variant \eqref{eq:wFCF} to yield the following
two-level error propagator for FCF-relaxation with weighted-C-Jacobi,
\begin{equation} \label{eq:wJ_ep}
(I-PB^{-1}_{\triangle}R_IA)P(I - \omega_C \mathbf{A}_{\triangle})R_I = P(I-\mathbf{B}_{\triangle}^{-1}\mathbf{A}_{\triangle})(I - \omega_C \mathbf{A}_{\triangle})R_I .
\end{equation}

Lastly, to derive our convergence bound, we follow the convention from \cite{Do2016, HeSoNoRoFaSc2018}
and examine the error propagator's effect only at C-points (i.e., drop the $P$ and $R_I$ from equation \eqref{eq:wJ_ep}).  This simplification is typically made with the 
following motivation.  If the solution at C-points is exact, then the final application of $P$ in 
\eqref{eq:wJ_ep} will produce the exact solution at F-points, i.e., a zero residual.  With this 
simplification, we denote the error propagator \eqref{eq:wJ_ep} at only C-points as $E_{\triangle, \hspace{1pt} \omega_C}^{FCF}$, which takes the form

\begin{subequations} \label{New Bound Error Matrix}
\begin{align} 
E_{\triangle, \hspace{1pt} \omega_C}^{FCF} &= (I-\mathbf{B}_{\triangle}^{-1}\mathbf{A}_{\triangle})(I - \omega_C \mathbf{A}_{\triangle})\\ 
&= \begin{bmatrix}
0 & \\
(1 - \omega_C)(\Phi^m - \Phi_{\triangle})  & 0 \\
(1 - \omega_C) \Phi_{\triangle} (\Phi^m - \Phi_{\triangle})  + \omega_C (\Phi^m - \Phi_{\triangle}) \Phi^m & (1 - \omega_C)(\Phi^m - \Phi_{\triangle}) & 0 \\
\vdots & \vdots & \ddots & 0 \\
(1 - \omega_C) \Phi_{\triangle}^{N_T - 1} (\Phi^m - \Phi_{\triangle}) + \omega_C \Phi_{\triangle}^{N_T - 2} (\Phi^m - \Phi_{\triangle}) \Phi^m & \cdots & \cdots & (1 - \omega_C) (\Phi^m - \Phi_{\triangle}) & 0 \\
\end{bmatrix}.  
\end{align}
\end{subequations}

\subsubsection{Two-grid eigenvalue convergence analysis}
\label{sec:proof}

To guarantee convergence, ideally we bound \eqref{New Bound Error Matrix} in some norm (e.g., see \cite{So2019}). However, working in a norm can be difficult; thus we take the more tractable approach of considering convergence for individual eigenvectors \cite{Do2016,So2019}. Thus, assume that $\Phi$ and $\Phi_\Delta$ have the same set of eigenvectors, $\{v_\gamma\}$, as occurs when the same spatial discretization is used on the coarse and fine grid in time, and let $\{\lambda_{\gamma}\}$ be the eigenvalues of  $\Phi$ and $\{\mu_{\gamma}\}$ be the eigenvalues of $\Phi_{\triangle}$. For instance, let $\kappa_{\gamma} \geq 0$ denote an eigenvalue of the linear operator $G$ in (\ref{ODE}); if backward Euler is used on the coarse and fine grid, we have
\begin{equation} \label{eig2} 
\lambda_{\gamma} = (1 - h_t \kappa_{\gamma})^{-1}, \mbox{ and } 
\mu_{\gamma} = (1 - m h_t \kappa_{\gamma})^{-1} \hspace{10pt} \text{for} \hspace{5pt} \gamma = 1,2,...,N_x.
\end{equation}
Define $\widetilde{U}$ as a block-diagonal operator, with diagonal blocks given by the eigenvector matrix for $\Phi$ and $\Phi_\Delta$. Following the discussion of Section 5 in \cite{So2019}, we can apply $\widetilde{U}$ to the left and $\widetilde{U}^{-1}$ to the right of \eqref{New Bound Error Matrix}. The resulting operator is then block diagonal, with diagonal blocks corresponding to a single pair of eigenvalues $\{\lambda_{\gamma}, \mu_\gamma\}$, and takes the following form: 
\begin{align}\label{eq:E-eig}
\widetilde{E}_{\triangle, \hspace{1pt} \omega_C}^{FCF} =
\begin{bmatrix} 0 \\
		(1-\omega_C)(\lambda_\gamma^m - \mu_\gamma) & 0 \\
		(1-\omega_C)\mu_\gamma(\lambda_\gamma^m - \mu_\gamma) + \omega_C(\lambda_\gamma^m - \mu_\gamma)\lambda_\gamma^m & (1-\omega_C)(\lambda_\gamma^k - \mu_\gamma) & 0 \\
		\vdots & \ddots & \ddots & \ddots \\
		(1-\omega_C)\mu_\gamma^{N_T-1}(\lambda_\gamma^m - \mu_\gamma) + \omega_C\mu_\gamma^{N_T-2}(\lambda_\gamma^m - \mu_\gamma)\lambda_\gamma^m & \hdots & \hdots & (1-\omega_C)(\lambda_\gamma^m - \mu_\gamma) & 0 \\ \end{bmatrix}.
\end{align}
Following the analysis in \cite{Do2016,So2019}, we can provide bounds on \eqref{New Bound Error Matrix} in a certain eigenvector-induced $(\widetilde{U}\widetilde{U}^*)^{-1}$-norm by bounding \eqref{eq:E-eig} in norm and taking the maximum over $\gamma$ (note, if the spatial matrix is SPD, $\widetilde{U}$ is unitary, and the $(\widetilde{U}\widetilde{U}^*)^{-1}$-norm is simply the $\ell^2$-norm).
Note that \eqref{eq:E-eig} is a Toeplitz matrix, with asymptotic generating function
\begin{align*}
\mathcal{F}_\gamma(x) & \coloneqq (\lambda_\gamma^m - \mu_\gamma)\left[
	(1-\omega_C)\sum_{\ell=1}^{\infty} \mu_\gamma^{\ell-1}e^{i\ell x} + 
	\omega_C\lambda_\gamma^m \sum_{\ell=2}^{\infty} \mu_\gamma^{\ell-2}e^{i\ell x}\right] \\
& = e^{ix}(\lambda_\gamma^m - \mu_\gamma)\left[
	(1-\omega_C)\sum_{\ell=0}^{\infty} (\mu_\gamma e^{i x})^{\ell} + 
	e^{ix}\omega_C\lambda_\gamma^m \sum_{\ell=0}^{\infty} (\mu_\gamma e^{i x})^{\ell}\right] \\
& = e^{ix} \frac{(\lambda_\gamma^m - \mu_\gamma)}{1 - e^{ix}\mu_\gamma}\left[1-\omega_C
	+ e^{ix}\omega_C\lambda_\gamma^m\right].
\end{align*}
Noting that $\mathcal{F}_\gamma(x)\in L^1[-\pi,\pi]$, from \cite{widom1989singular} (see also
\cite[Th. 2.1]{capizzano1999extreme}), we have that
\begin{align}\nonumber
\sigma_{max, \gamma}(\widetilde{E}_{\triangle, \hspace{1pt} \omega_C}^{FCF}) & \leq \max_{x\in[0,2\pi]} |\mathcal{F}_\gamma(x)| \\
& = \max_{x\in[0,2\pi]} \frac{|\lambda_\gamma^m - \mu_\gamma|}{|1 - e^{ix}\mu_\gamma|}
	|1-\omega_C + e^{ix}\omega_C\lambda_\gamma^m| \label{eq:bound}.
\end{align}
Taking the maximum over $\gamma$, corresponding to all (shared) eigenvectors of
$\Phi$ and $\Phi_\Delta$ yields the following final result.
\begin{theorem}\label{th:bound}
Assume that $\Phi$ and $\Phi_\Delta$ have the same set of eigenvectors, with eigenvalues
$\{\lambda_{\gamma}\}$ and $\{\mu_{\gamma}\}$, respectively, where $|\lambda_\gamma|,
|\mu_\gamma| < 1$ for all $\gamma\in[1,N_x]$. Let $\widetilde{U}$ denote a block-diagonal operator,
with diagonal blocks given by the eigenvector matrix of $\Phi$ and $\Phi_\Delta$. Then,
\begin{align} \label{sharper bd}
\|E_{\triangle, \hspace{1pt} \omega_C}^{FCF}\|_{(\widetilde{U}\widetilde{U}^*)^{-1}} \leq
    \max_\gamma \max_{x\in[0,2\pi]} \frac{|\lambda_\gamma^m - \mu_\gamma|}{|1 - e^{ix}\mu_\gamma|}
	|1-\omega_C + e^{ix}\omega_C\lambda_\gamma^m|.
\end{align}
\end{theorem}
\begin{proof}
The proof follows from the above discussion.
\end{proof}
For fixed $\gamma$, a closed form for the maximum over $x$ in \eqref{sharper bd} to
allow for easier computation is provided in the Supplemental materials.

We numerically verify the convergence bound \eqref{sharper bd} in \Cref{sec: Numerial veri bound}
for model 1D heat and advection equations, respectively. In some cases, the bound is quite tight,
while for others the general behavior is right, but bounds are not exact. This is likely due to
\Cref{th:bound} providing an \emph{upper} bound on worst-case convergence; even if the upper
bound is tight (which \Cref{th:bound} is asymptotically in $N_T$), it is possible that better
convergence can be observed in practice, depending on the problem and right-hand side.


\begin{remark}
We also note that one can approximate the maximum over $x$ in \Cref{th:bound} by assuming
a fixed $x$ rotates $\lambda_\gamma$ and $\mu_\gamma$ to the real-axis. Experiments have indicated this to be a reasonable assumption for eigenvalues with dominant real-part, although less so for eigenvalues with large imaginary component. Nevertheless, it does yield a simpler measure to compute, and can be applied to weighted FCF- and FCFCF-relaxation (degree-two weighted-Jacobi), with approximate bounds
\begin{align} \label{bd FCFCF-relaxation}
\begin{split}
\|E_{\triangle, \hspace{1pt} \omega_C}^{FCF}\|_{(\widetilde{U}\widetilde{U}^*)^{-1}} & \lessapprox
\max_\gamma \frac{|\lambda_\gamma^m - \mu_\gamma|}{1 - |\mu_\gamma|}
|1-\omega_C + \omega_C|\lambda_\gamma^m||\textbf{}, \\
\|{E}_{\triangle, \hspace{1pt} \{\omega_C,\omega_{CC}\}}^{FCFCF}\|_{(\widetilde{U}\widetilde{U}^*)^{-1}} &\lessapprox
\max_\gamma \frac{|\lambda_\gamma^m - \mu_\gamma|}{1 - |\mu_\gamma|}
|1-\omega_C + \omega_C|\lambda_\gamma^m||\, |1-\omega_{CC} + \omega_{CC}|\lambda_\gamma^m||.
\end{split}
\end{align}
For the derivation of the FCFCF-bound, see Appendix \ref{app1}.
\end{remark}

\section{Verifying the Convergence Bound}
\label{sec:verify}
\subsection{Numerical verification of the convergence bound} \label{sec: Numerial veri bound}
We focus our verification tests on three model problems with the following spatial discretizations, the 1D heat equation (second-order central differencing in space), the 1D advection equation with purely imaginary
spatial eigenvalues (second-order central differencing in space), and the 1D advection equation with complex spatial eigenvalues (first-order upwinding in space).  In all cases, backward Euler is used in time.\footnote{For a complete description of these problems, see the Supplemental Materials for the heat equation in Section \ref{sec:results_heat},  the advection equation with purely imaginary spatial eigenvalues in Section \ref{sec:results_adv_imag}, and the advection equation with complex spatial eigenvalues in Section \ref{sec:1dadv_complex}.}  We choose these model problems because the theoretical motivation of equation \eqref{sharper bd} indicates that it is the character of the spatial eigenvalues and the time-stepping method that determine the convergence of MGRIT, i.e., not the dimensionality of the problem, the complexity of the governing PDE, or the nature of the forcing term and boundary conditions.  Thus, we choose these three representative cases, similar to \cite{Do2016, HeSoNoRoFaSc2018}.

We consider the 1D heat equation subject to an initial condition and homogeneous Dirichlet boundary conditions,
\begin{align}
\frac{\partial u}{\partial t} - \alpha \frac{\partial^2 u}{\partial x^2} &= f(x,t), \quad\alpha > 0, && x \in \Omega = [0, L],\hspace{10pt} t \in [0, T], \nonumber \\
u(x, 0) &= u_0(x), && x \in \Omega, \\
u(x, t) &= 0, && x \in \partial \Omega, \hspace{10pt} t \in [0, T].\nonumber
\end{align} For numerical experiments, we use the space-time domain $[0, 1]\times [0, 0.625]$, the diffusivity constant $\alpha = 1$, and the right-hand side $f(x,t) = \sin(\pi x) [\sin(t) - \pi^2 \cos(t)]$. Note that with these choices, the analytical solution is given by $u(x, t) = \sin(\pi x)\cos(t)$.
A random initial guess and a residual norm halting tolerance of $ 10^{-10} / \sqrt{h_x h_t}$ are used. 
Reported convergence rates are taken as an average over the last five MGRIT
iterations, where $\| r_k\|_2 / \| r_{k-1} \|_2$ is the convergence rate at iteration $k$ and $r_k$ is the residual from equation \eqref{time step eq matrix} at iteration $k$.
The combination of grid points in space $N_x$ and time $N_t$ are chosen so that 
$\frac{h_t}{h_x^2}  = 12.8$.  This value was chosen to be of moderate magnitude and 
consistent with other MGRIT literature, namely the work \cite{Do2016}.

We also consider the 1D advection equation with purely imaginary spatial eigenvalues, subject to an initial condition and periodic spatial boundary conditions,
\begin{align}
\frac{\partial u}{\partial t} - \alpha \frac{\partial u}{\partial x} &= 0, 
\quad \alpha > 0, && x \in \Omega = [0, L], \hspace{10pt} t \in [0, T], \nonumber \\
u(x, 0) &= u_0(x), && x \in \Omega, \\
u(0, t) &= u(L, t), && t \in [0, T].\nonumber
\end{align}
The space-time domain considered is $[0, 1]\times [0, 1]$, the velocity constant $\alpha = 1$, and the analytical solution 
$u(x, t) = e^{-25((x - t) - 0.5)^2}$.
The solution is chosen as a standard test problem that satisfies the spatially periodic boundary conditions.
A random initial guess and a residual norm halting tolerance of $ 10^{-8} / \sqrt{h_x h_t}$ are used.  The maximum allowed
iterations is set to $70$, because some cases will fail to quickly converge. 
Reported convergence rates are taken as $( \| r_k \|_2 / \| r_0 \|_2 )^{1/k}$ at the final iteration $k$.  The geometric average is used (as opposed to the heat equation case above) because the per iteration convergence rate here can vary significantly. The combination of grid points in space $N_x$ and time $N_t$ are chosen so that $\frac{h_t}{h_x} = 0.5$.

Figure \ref{fig:Heat new theoretical bound} (a) and Figure \ref{fig:AdvcC new theoretical bound} (a) depict the convergence bound (dashed line) and experimental convergence rates (solid line) against various relaxation weights $\omega_C$ for the 1D heat equation and the 1D advection equation with purely imaginary spatial eigenvalues, respectively. Figure \ref{fig:Heat new theoretical bound} (b) and Figure \ref{fig:AdvcC new theoretical bound} (b) show the iterations associated with the experimental convergence rates. For Figure \ref{fig:Heat new theoretical bound}, the theoretical bound is very tight and predicts the optimal $\omega_C$. For the advective case in Figure \ref{fig:AdvcC new theoretical bound}, the bound is predictive, but not quite sharp enough to predict the best weight. The results for the 1D advection equation with complex spatial eigenvalues are similar to the 1D advection equation with purely imaginary spatial eigenvalues and, thus, are omitted.

\begin{figure}[h!]
    \centering
    \begin{subfigure}[b]{0.4\textwidth}
    \includegraphics[width=\textwidth]{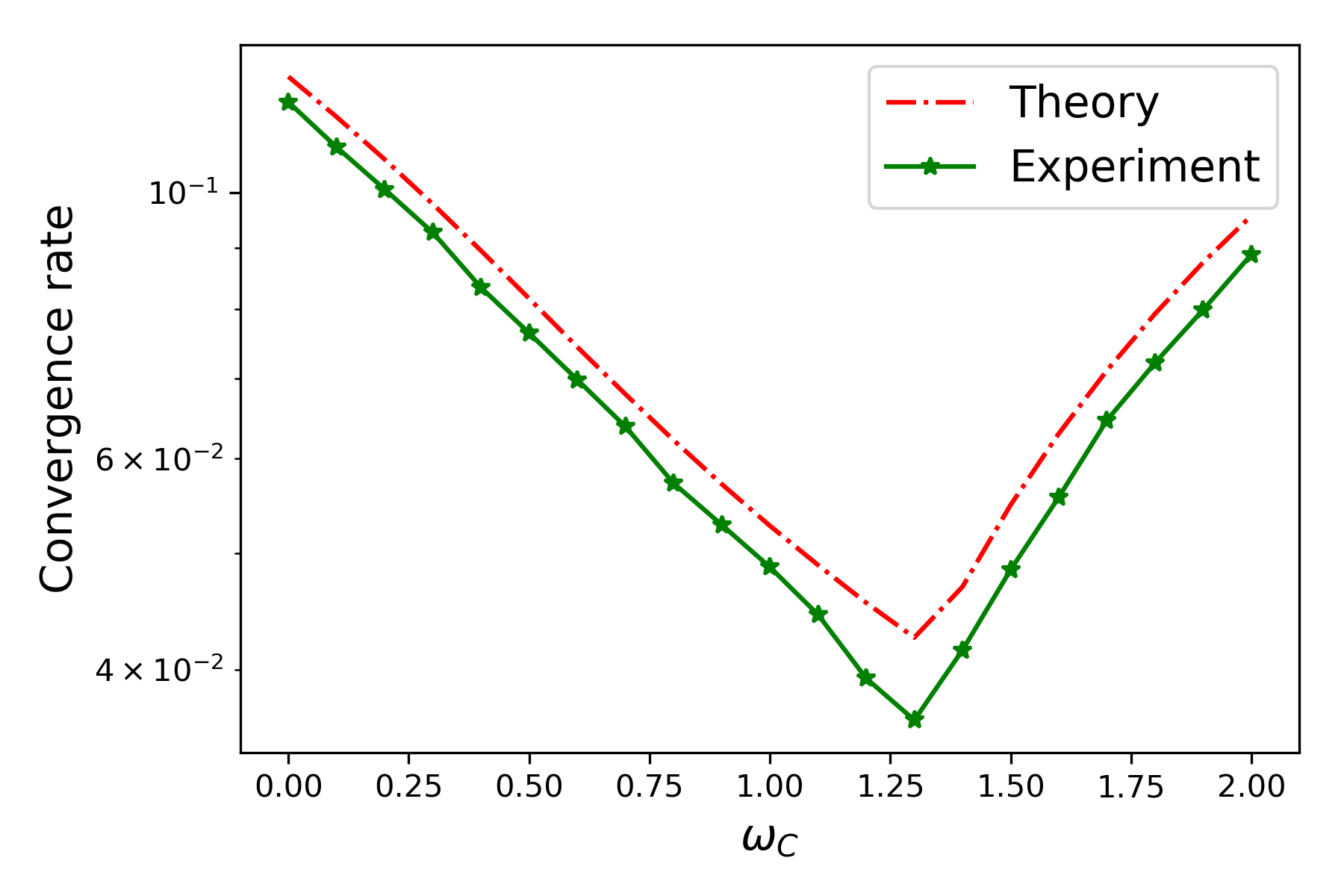}
    \caption{\normalsize Convergence Rate}
    \end{subfigure}
     \begin{subfigure}[b]{0.4\textwidth}
    \includegraphics[width=\textwidth]{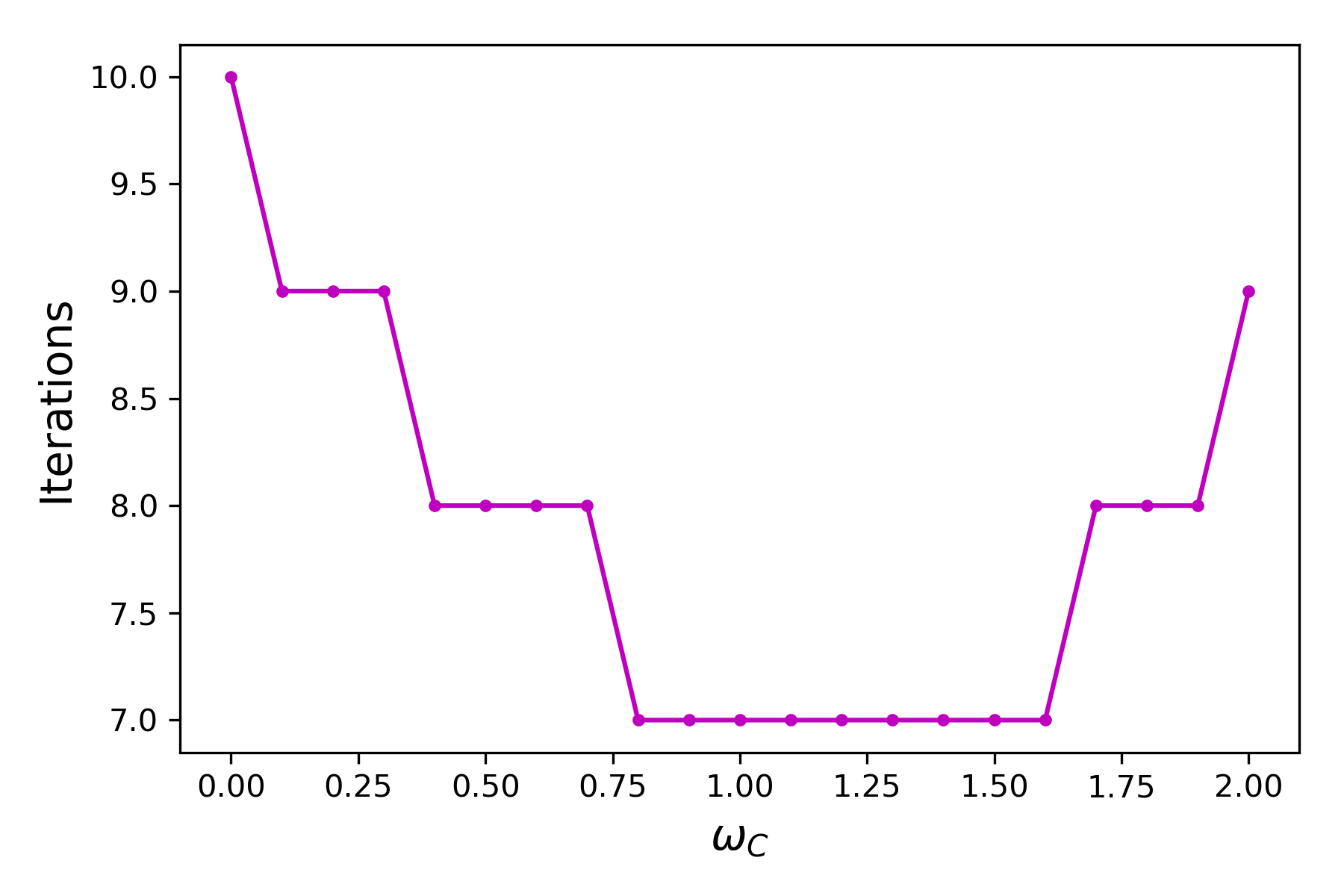}
    \caption{\normalsize Iterations}
    \end{subfigure}
    \caption{Two-level MGRIT theoretical bound (dashed line in left plot), experimental convergence rates (solid line in left plot), and iteration counts (right plot) as a function of relaxation weights $\omega_C$ for the one-dimensional heat equation, coarsening factor $m = 2$, and grid size $(N_x, N_t) = (291, 4097)$.}
    \label{fig:Heat new theoretical bound}
\end{figure}

\begin{figure}[h!]
    \centering
    \begin{subfigure}[b]{0.4\textwidth}
    \includegraphics[width=\textwidth]{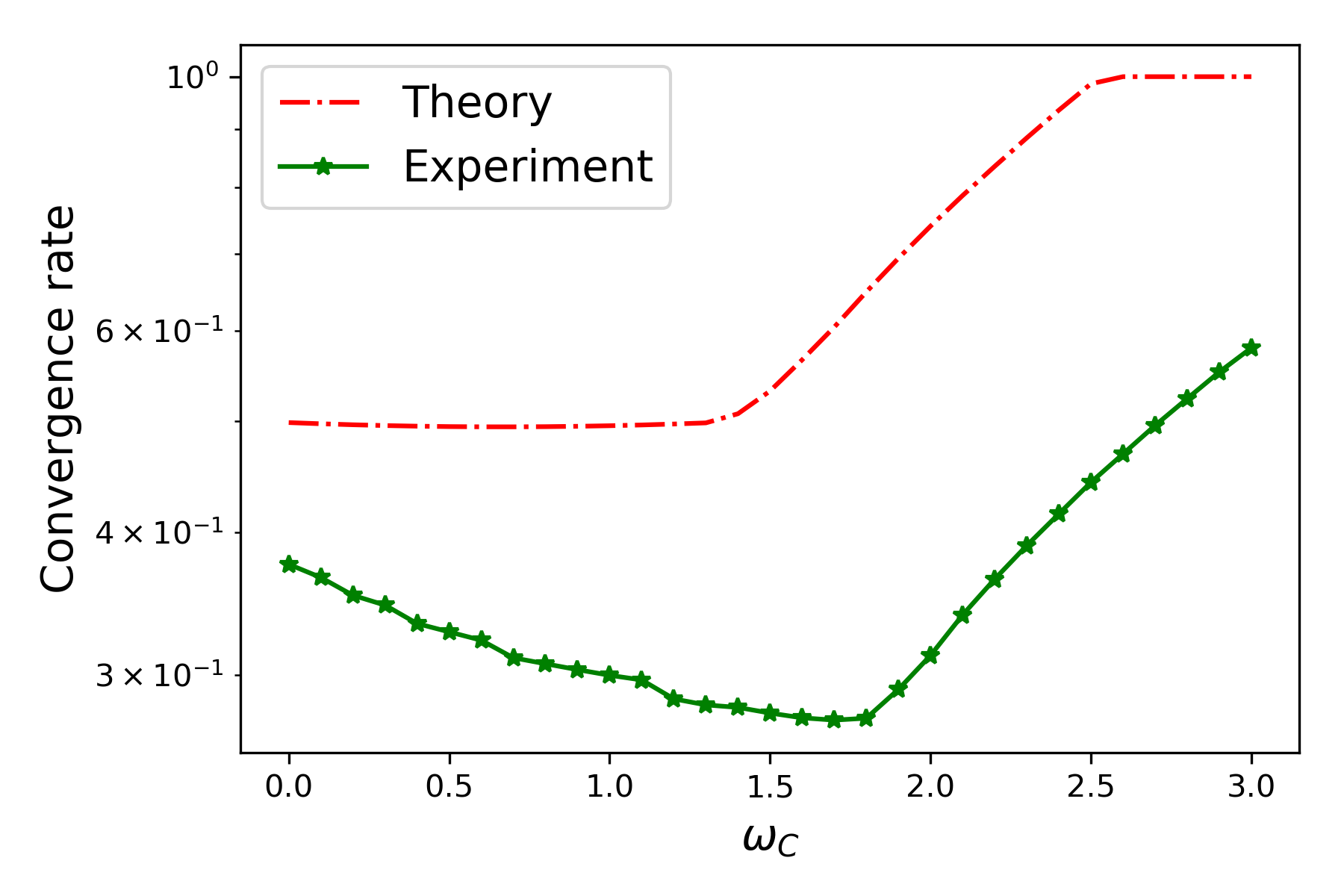}
    \caption{\normalsize Convergence Rate}
    \end{subfigure}
     \begin{subfigure}[b]{0.4\textwidth}
    \includegraphics[width=\textwidth]{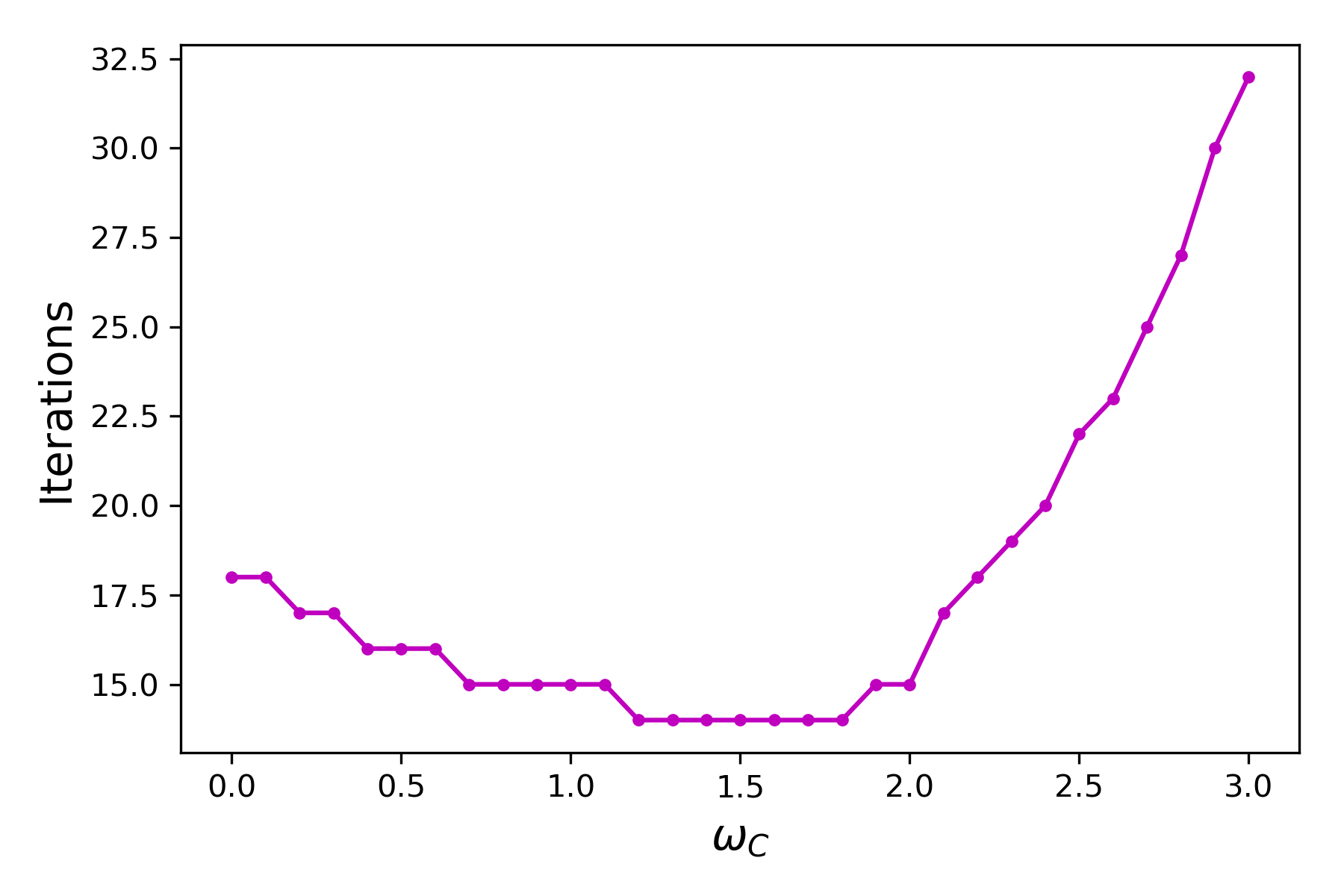}
    \caption{\normalsize Iterations}
    \end{subfigure}
    \caption{Two-level MGRIT theoretical bound (dashed line in left plot), experimental convergence rates (solid line in left plot) and iteration counts  (right plot) as a function of relaxation weights $\omega_C$ for the one-dimensional linear advection equation with purely imaginary spatial eigenvalues, coarsening factor $m = 2$, and grid size $(N_x, N_t) = (1025, 1025)$.}
    \label{fig:AdvcC new theoretical bound}
\end{figure}

Next, we summarize the experimentally best relaxation weights for the 1D heat equation and the 1D advection equation with purely imaginary spatial eigenvalues. 
For the full multilevel experiments, V-cycles are used and we coarsen down to a grid of size 4 or less in time.  During searches in the weight-space for experimentally optimal weights, we use a step size of 0.1, and in these tables we report only the best weight in comparison to a unitary weight of 1.0.  For expanded versions of these tables, please see Supplemental Materials \ref{app2}, 
Tables \ref{tab:Heat Conv and Iter Two level}, \ref{tab:Heat Conv and Iter Multi level}, \ref{tab:LA Conv and Iter for Twolevel}, and \ref{tab:LA Conv and Iter for Multi level}.   Regarding notation, $\omega_{CC}$ denotes the weight for the second weighted
relaxation, if degree-two (FCFCF) weighted relaxation is used.  If only
$\omega_C$ is given, then only degree-one (FCF) weighted relaxation is used. 

Tables \ref{tab:1D Heat results Two-level} and \ref{tab:1D Heat results Multilevel} depict the results for the 1D heat equation for a two-level and multi-level solver, respectively.  The best experimental weight for degree-one relaxation in both cases is $\omega_C=1.3$ and saves 1 iteration on the largest problem, or approximately 10\%--14\%.  The best weights $(\omega_C,\omega_{CC})$ for degree-two relaxation differ between two-level and multilevel, but similarly save 1 iteration. Other coarsening factors $m$ were tested, but generated the same experimentally best weights (see Supplemental Results Section  \ref{sec:results_heat} for more details).

Tables \ref{tab:1D Advc results Two-level} and \ref{tab:1D Advc results Multilevel} depict the results for the 1D advection equation with purely imaginary spatial eigenvalues for a two-level and multilevel solver, respectively. The best experimental weights for degree-one relaxation differ between not only two-level and multilevel but also coarsening factors $m=2$ and $m=4$. The best experimental weight in the two-level case with $m=4$ is $\omega_C=1.5$ and saves 2--3 iterations on the larger problems, or approximately 5\%-9\%. The best experimental weight in the multilevel case with $m=2$ is $\omega_C=1.5$ and saves 15 iterations on the second largest problem, or approximately 22\%.  The best weights $(\omega_C,\omega_{CC})$ for degree-two relaxation have been omitted for brevity, but are in Supplemental Materials Section \ref{sec:results_adv_imag}.

\begin{table}[h!]
	\centering
	\begin{tabular}{c r|c|c|c|c}
		
		& $N_x \times N_t$ & $291 \times 4097$ & $411 \times 8193$ & $581 \times 16385$ & $821 \times 32769$ \\ \toprule
		\multirow{4}{*}{$m=2$} & $\omega_C=1.0$                      & 0.049 (7) & 0.048 (7) & 0.039 (7) & 0.039 (7) \\ 
		& $1.3$                                                      & 0.036 (7) & 0.036 (7) & 0.034 (6) & 0.034 (6) \\
		& $(\omega_C,\omega_{CC})=(1.0, 1.0)$ & 0.029 (6) & 0.029 (6) & 0.029 (6) & 0.028 (6) \\ 
		& $(1.7, 0.9)$                                      & 0.020 (6) & 0.020 (6) & 0.019 (6) & 0.016 (5) \\  \bottomrule
	\end{tabular}
	\caption{1D heat equation, two-level MGRIT convergence rates (iterations) for weighted FCF- and FCFCF-relaxation with unitary weights and the experimentally best weights. }
	\label{tab:1D Heat results Two-level}
\end{table}

\begin{table}[h!]
	\centering
	\begin{tabular}{c r|c|c|c|c}
		
		& $N_x \times N_t$ & $291 \times 4097$ & $411 \times 8193$ & $581 \times 16385$ & $821 \times 32769$ \\   \toprule
		\multirow{4}{*}{$m=2$} & $\omega_C=1.0$                      & 0.118 (9) & 0.121 (9) & 0.123 (9) & 0.125 (9) \\ 
		& $1.3$                                                      & 0.092 (8) & 0.095 (8) & 0.096 (8) & 0.096 (8) \\ 
		& $(\omega_C,\omega_{CC})=(1.0, 1.0)$ & 0.065 (7) & 0.066 (7) & 0.067 (7) & 0.068 (7) \\ 
		& $(2.0, 0.9)$                                      & 0.032 (6) & 0.032 (6) & 0.032 (6) & 0.032 (6) \\ \bottomrule
	\end{tabular}
	\caption{1D heat equation, multilevel MGRIT convergence rates (iterations) for weighted FCF- and FCFCF-relaxation with unitary weights and the experimentally best weights.}
	\label{tab:1D Heat results Multilevel}
\end{table}

\begin{table}[h!]
	\centering
	\begin{tabular}{l r|c|c|c|c}
		
		& $N_x \times N_t$ & $513 \times 513$ & $1025 \times 1025$ & $2049 \times 2049$ & $4097 \times 4097$ \\ \toprule
		\multirow{2}{*}{$m=2$} & $\omega_C = 1.0$                   & 0.304 (15) & 0.307 (15) & 0.308 (15) & 0.309 (15) \\ 
		& $1.8$                                                     & 0.280 (14) & 0.282 (14) & 0.284 (14) & 0.285 (14) \\ \midrule
		\multirow{2}{*}{$m=4$} & $\omega_C =1.0$                    & 0.564 (30) & 0.607 (34) & 0.617 (35) & 0.619 (35) \\ 
		& $1.5$                                                     & 0.568 (30) & 0.581 (31) & 0.591 (32) & 0.596 (33) \\ \bottomrule
	\end{tabular}
	\caption{1D linear advection equation, two-level MGRIT convergence rates (iterations) for weighted FCF-relaxation with unitary weights and the  experimentally best weights.}
	\label{tab:1D Advc results Two-level}
\end{table}

\begin{table}[h!]
	\centering
	\begin{tabular}{c r|c|c|c|c}
		&$N_x \times N_t$ & $513 \times 513$ & $1025 \times 1025$ & $2049 \times 2049$ & $4097 \times 4097$ \\ \toprule 
		\multirow{2}{*}{$m=2$} &   $\omega_C = 1.0$                 & 0.560 (30) & 0.675 (44) & 0.771 (67) & $(>100)$ \\ 
		& $1.5$                                                     & 0.495 (24) & 0.606 (35) & 0.718 (52) & 0.810 (82) \\ \midrule
		\multirow{2}{*}{$m=4$} & $\omega_C = 1.0$                   & 0.581 (32) & 0.666 (42) & 0.757 (61) & 0.838 (95) \\ 
		& $1.4$                                                     & 0.535 (27) & 0.611 (34) & 0.712 (50) & 0.802 (77) \\ \bottomrule
	\end{tabular}
	\caption{1D linear advection equation, multilevel MGRIT convergence rates (iterations) for weighted FCF-relaxation with unitary weights and the experimentally best weights.}
	\label{tab:1D Advc results Multilevel}
\end{table}

\subsection{Visualizing the convergence bound}
\label{sec:verify:vis}

Recall that $\{\lambda_{\gamma}\}$ and $\{\mu_{\gamma}\}$ are the eigenvalues of $\Phi$ and $\Phi_{\triangle}$, respectively corresponding to the same set of eigenvectors $\{v_\gamma\}$. That is, $\Phi$ and $\Phi_{\triangle}$ are diagonalized by the eigenvectors 
$\{v_\gamma\}$.
If $\kappa_{\gamma} \geq 0$ is an eigenvalue of the linear operator $G$ in (\ref{ODE}), the corresponding eigenvalue of $\Phi$ is given by
\begin{equation}
\lambda_\gamma = 1 + h_t \kappa_\gamma \mathbf{b}_0^T (I - h_t \kappa_\gamma A_0)^{-1} \mathbf{1}, \mbox{ and }
\mu_\gamma = 1 + m h_t \kappa_\gamma \mathbf{b}_0^T (I - m h_t \kappa_\gamma A_0)^{-1} \mathbf{1}
\end{equation}
where the Runge-Kutta matrix $A_0 = (a_{i,j})$ and weight vector $\mathbf{b}_0^T = (b_1,...,b_s)^T$ are taken from the Butcher tableau of an s-stage Runge-Kutta method \cite{FrSo2020}.

Here, we consider A-stable two-stage third-order SDIRK-23, L-stable two-stage second-order SDIRK-22, and L-stable three-stage third-order SDIRK-33 methods (see Appendix of \cite{FrSo2020} for coefficients), where SDIRK refers to singly diagonally implicit Runge-Kutta. Figures \ref{heat map: A-stable SDIRK23} -- \ref{heat map: L-stable SDIRK33} depict the convergence bound \eqref{eq:bound} in the complex plane as a function of $h_t \kappa_\gamma$ over various $\omega_C$ for these methods, respectively. Overall, the L-stable schemes lead to significantly better MGRIT convergence bounds than the A-stable scheme, consistent with the discussion and results for unweighted relaxation in \cite{FrSo2020}, and, more importantly, numerical results using weighted relaxation in Section
\ref{sec:lstable}. Additionally, note from Figure \ref{heat map: A-stable SDIRK23} that for unweighted relaxation ($\omega_C =1$), two-level MGRIT is divergent in much of the complex plane (a known phenomenon \cite{FrSo2020}). However, applying under-relaxation with $\omega_C = 0.8$ restores reasonable convergence in much of the complex plane. This behavior is confirmed in practice in Section \ref{sec:astable}. Similarly, applying under-relaxation to L-stable SDIRK-33 in Figure \ref{heat map: L-stable SDIRK33} yields convergence, albeit slow, along the imaginary axis. Spatial eigenvalues on the imaginary axis are notoriously difficult for MGRIT to converge on, as can be seen with the theoretical bounds for $\omega_C=1$. To the best of our knowledge, backward Euler is the \emph{only} one-step time-integration scheme that yields convergence on the imaginary axis.\footnote{It is important to note that for unweighted relaxation, two-level convergence bounds are \emph{necessary and sufficient}\cite{So2019}.} Here, we see that weighted relaxation can yield convergence on higher-order integration schemes as well.

\begin{figure}[htb]
    \centering 
\begin{subfigure}{0.3\textwidth}
   \includegraphics[width=\linewidth]{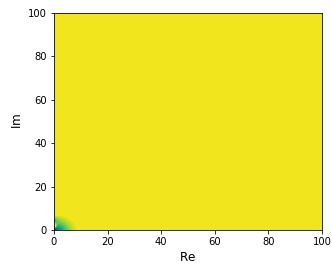}
  \caption{\normalsize $\omega_C=0.5$}
\end{subfigure}\hfil
\begin{subfigure}{0.3\textwidth}
   \includegraphics[width=\linewidth]{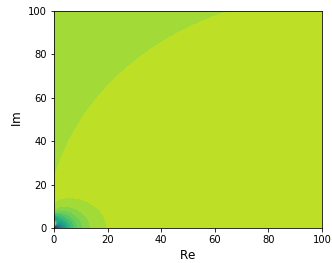}
  \caption{\normalsize $\omega_C=0.6$}
\end{subfigure}\hfil
\begin{subfigure}{0.35\textwidth}
   \includegraphics[width=\linewidth]{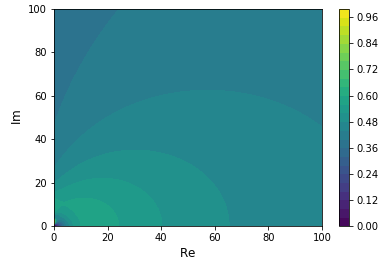}
  \caption{\normalsize $\omega_C=0.7$}
\end{subfigure}

\medskip
\begin{subfigure}{0.3\textwidth}
   \includegraphics[width=\linewidth]{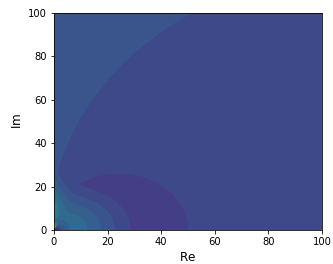}
  \caption{\normalsize $\omega_C=0.8$}
\end{subfigure}\hfil
\begin{subfigure}{0.3\textwidth}
   \includegraphics[width=\linewidth]{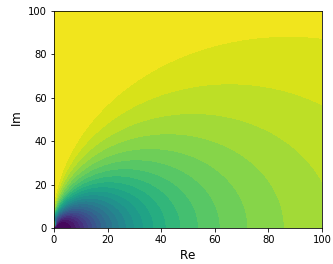}
  \caption{\normalsize $\omega_C=1.0$}
\end{subfigure}\hfil
\begin{subfigure}{0.3\textwidth}
   \includegraphics[width=\linewidth]{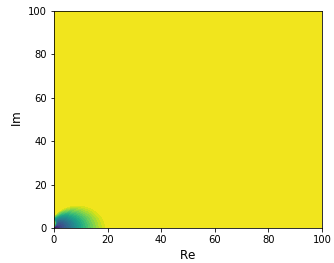}
  \caption{\normalsize $\omega_C=1.2$}
\end{subfigure}
\caption{Two-level MGRIT theoretical convergence bound as a function of Re($h_t \kappa_\gamma$) and Im($h_t \kappa_\gamma$), for $m=2$ and A-stable 2-stage SDIRK-23.}
\label{heat map: A-stable SDIRK23}
\end{figure}

\begin{figure}[htb]
    \centering 
\begin{subfigure}{0.3\textwidth}
   \includegraphics[width=\linewidth]{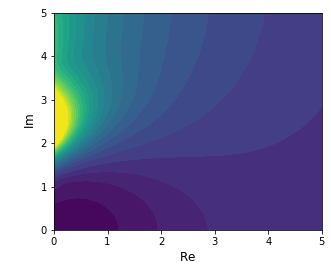}
  \caption{\normalsize $\omega_C=0.5$}
\end{subfigure}\hfil
\begin{subfigure}{0.3\textwidth}
   \includegraphics[width=\linewidth]{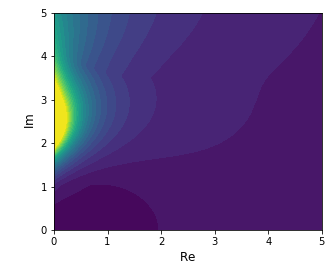}
  \caption{\normalsize $\omega_C=0.75$}
\end{subfigure}\hfil
\begin{subfigure}{0.35\textwidth}
   \includegraphics[width=\linewidth]{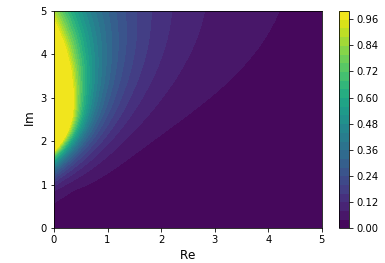}
  \caption{\normalsize $\omega_C=1.0$}
\end{subfigure}

\medskip
\begin{subfigure}{0.3\textwidth}
   \includegraphics[width=\linewidth]{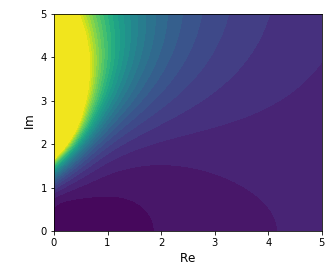}
  \caption{\normalsize $\omega_C=1.25$}
\end{subfigure}\hfil
\begin{subfigure}{0.3\textwidth}
   \includegraphics[width=\linewidth]{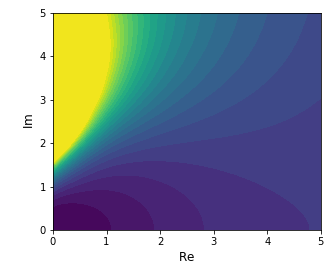}
  \caption{\normalsize $\omega_C=1.5$}
\end{subfigure}\hfil
\begin{subfigure}{0.3\textwidth}
   \includegraphics[width=\linewidth]{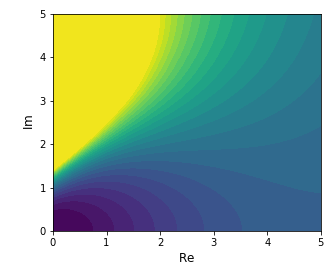}
  \caption{\normalsize $\omega_C=2.0$}
\end{subfigure}
\caption{Two-level MGRIT theoretical convergence bound as a function of Re($h_t \kappa_\gamma$) and Im($h_t \kappa_\gamma$), for $m=2$ and L-stable 2-stage SDIRK-22.}
\label{heat map: L-stable SDIRK22}
\end{figure}

\begin{figure}[htb]
    \centering 
\begin{subfigure}{0.3\textwidth}
   \includegraphics[width=\linewidth]{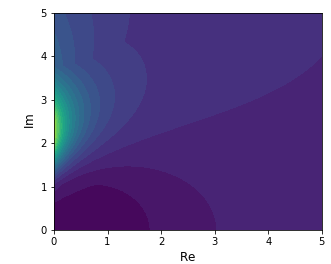}
  \caption{\normalsize $\omega_C=0.7$}
\end{subfigure}\hfil
\begin{subfigure}{0.3\textwidth}
   \includegraphics[width=\linewidth]{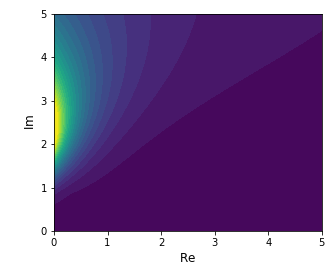}
  \caption{\normalsize $\omega_C=1.0$}
\end{subfigure}\hfil
\begin{subfigure}{0.35\textwidth}
   \includegraphics[width=\linewidth]{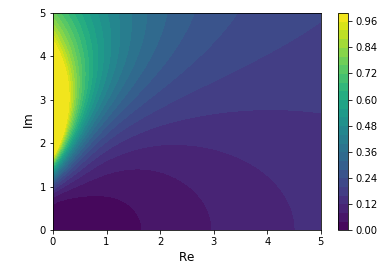}
  \caption{\normalsize $\omega_C=1.3$}
\end{subfigure}

\medskip
\begin{subfigure}{0.3\textwidth}
   \includegraphics[width=\linewidth]{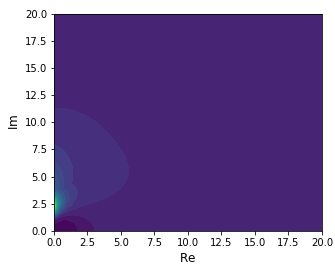}
  \caption{\normalsize $\omega_C=0.7$, and the axes go up to 20.}
\end{subfigure}\hfil
\begin{subfigure}{0.3\textwidth}
   \includegraphics[width=\linewidth]{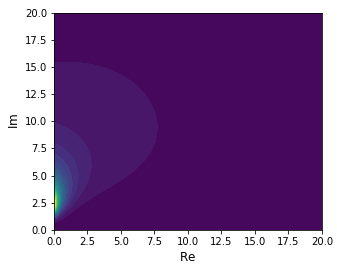}
  \caption{\normalsize $\omega_C=1.0$, and the axes go up to 20.}
\end{subfigure}\hfil
\begin{subfigure}{0.3\textwidth}
   \includegraphics[width=\linewidth]{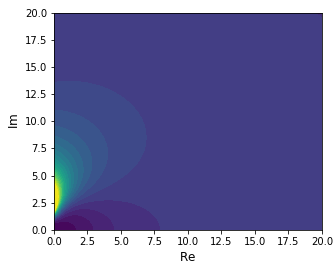}
  \caption{\normalsize $\omega_C=1.3$, and the axes go up to 20.}
\end{subfigure}
\caption{Two-level MGRIT theoretical convergence bound as a function of Re($h_t \kappa_\gamma$) and Im($h_t \kappa_\gamma$), for $m=2$ and L-stable 3-stage SDIRK-33.}
\label{heat map: L-stable SDIRK33}
\end{figure}


\section{Results}
\label{sec:results}

This section demonstrates MGRIT with weighted relaxation on a 2D advection-diffusion problem and a nonlinear eddy current problem. 
%

\subsection{2D Convection-Diffusion with discontinuous Galerkin elements}
\label{sec:2Dresults}

To indicate generality of the proposed weighted relaxation scheme, we now consider the advection-diffusion 
problem 
\begin{align}
   \label{eqn:adv-diff}
   \frac{\partial u}{\partial t} + \mathbf{b}(t,\mathbf{x})\cdot \nabla u - \epsilon \nabla \cdot \nabla u & = 0, \quad \mathbf{x} \in \Omega, \quad t \in [0,T] \\
      u(\mathbf{x},0) &= u_0(\mathbf{x}),\quad  \mathbf{x} \in \Omega,
\end{align}
where $\epsilon > 0$ is the diffusion constant, $\Omega$ is a bounded convex domain in  2D, and the
boundary conditions are periodic in space. The final time $T$ is set to 20 and
$\mathbf{b}=(\sqrt{2/3},\sqrt{1/3})$.  Letting $\mathbf{x}=(x_1,x_2)$, the initial condition is 
\begin{align*}
u_0(\mathbf{x})
  &= \frac{1}{16}
  \operatorname{erfc}[w (x_1-c_1-r_1)] \operatorname{erfc}[-w (x_1-c_1+r_1)] \\
  &\phantom{= \frac{1}{16}} ~\times
  \operatorname{erfc}[w (x_2-c_2-r_2)] \operatorname{erfc}[-w (x_2-c_2+r_2)],
\end{align*}
which defines a smooth rectangular hump with $\operatorname{erfc}(x)$ the complementary error function, 
$(c_1,c_2)=(0,-0.2)$, $(r_1,r_2)=(0.45,0.25)$, and $w=10$.  


We use the MFEM library \cite{mfem} to discretize over a regular quadrilateral
grid on a hexagonal domain $\Omega$, corresponding to the file
\texttt{mfem/data/periodic-hexagon.mesh}.  In space, we use $Q_1$ (bi-linear)
or $Q_3$ (bi-cubic) discontinuous Galerkin (DG) elements with a standard upwind
scheme for the advective term and the interior penalty (IP) \cite{ArBrCoMa2002}
scheme for the diffusion term.  In time, we consider backward Euler (L-stable), 
the A-stable two-stage third-order SDIRK-23 method, 
and the L-stable three-stage third-order SDIRK-33 
method.  


%

The numerical setup uses MGRIT V-cycles with a random initial guess and a residual halting tolerance of $10^{-10} / (h_x \sqrt{\delta_t}) $.  
The iterations are capped at 125, with ``125+" indicating that this maximum was
reached.  The value $N_x$ represents the total number of spatial
degrees-of-freedom, and grows by a factor of 4 each uniform refinement because space is now
2D.  The number of time points grows by a factor of 2, so that $\delta_t /
h_x = 0.477$ is fixed for all test problems, where $h_x$ refers to the spatial mesh size. 
Regarding the diffusive term, the 
ratio $\delta_t / h_x^2$ varies from $1.9245$ for the smallest problem, to $15.396$
on the largest problem, representing moderate ratios typical for an implicit scheme.

\subsubsection{Results for L-Stable Schemes}
\label{sec:lstable}
Tables \ref{tab:order1_CD} and \ref{tab:order3_CD} depict these results for the
case of bilinear DG elements with backward Euler and bi-cubic DG elements with L-stable
SDIRK-33, respectively.   Three diffusion constants, $\epsilon = 0.1, 0.01$, and
$0.001$, are depicted to highlight the benefits of weighted relaxation for
three different MGRIT convergence regimes.  The first regime concerns
sufficiently diffusive problems, where MGRIT convergence is bounded with
growing problem size \cite{Do2016}.  This is observed for the $\epsilon=0.1$
case.  For the next regime when $\epsilon=0.01$, the problem is on the cusp of
sufficient diffusiveness, as evidenced by the growing iteration counts for
backward Euler in Table \ref{tab:order1_CD}, but flat iteration counts in Table
\ref{tab:order3_CD} for some weight values.\footnote{Note that SDIRK-33 is a
more favorable time-stepping scheme for MGRIT convergence and diffusive
problems \cite{Do2016}, thus it is not surprising that it provides better
performance here.  In fact, if these experiments are repeated with bi-cubic DG
elements and backward Euler, the results are almost identical to Table
\ref{tab:order1_CD} for bilinear DG elements and backward Euler, thus
indicating that the use SDIRK-33 is the factor leading to the improved
convergence.} When $\epsilon = 0.001$, convergence is poor in both cases.

In all three regimes, the benefits of weighted relaxation can be observed and
are similar to those benefits observed for the 1D model problems in the Supplemental 
Materials \ref{app2}. 
For the first-order discretizations in Table 
\ref{tab:order1_CD}, a weight choice of 1.6 is experimentally found to be best, saving
15\%--20\% of iterations, which aligns with the best weight choice for 1D advection in 
Appendix \ref{app2}.\footnote{We note that while the tables only show
a handful of weight choices, thorough experimentation with under- and over-relaxation using a weight step-size of 0.1 was done to find the experimentally best 
choices.  } For the third-order discretizations in Table \ref{tab:order3_CD}, a weight choice 
of 1.3 is  experimentally found to be best, saving 10\%--15\% of iterations.  This does not 
align with  the best weight choice for 1D advection in Appendix \ref{app2}, but instead  
aligns with the best weight choice for 1D diffusion.  Thus, we can say that the simple 1D 
model problems from Appendix \ref{app2} provide a useful, but rough guide for choosing 
relaxation weights for more complicated problems.  Lastly, we note that under-relaxation was 
not beneficial for these cases, as indicated by the $\omega_C = 0.7$ case.

\begin{table}[h!]
\centering
\begin{tabular}{c r|c|c|c|c}
     & $N_x \times N_t$ & $192 \times 192$ & $768 \times 384$ & $3072 \times 768$ & $12288 \times 1536$ \\  \toprule
     \multirow{5}{*}{$\epsilon=0.001$} & $\omega_C=0.7$  & 29  & 39  & 56  & 125+ \\ 
     &$1.0$           & 25 & 32 & 47 & 65 \\ 
     &$1.3$           & 22 & 28 & 42 & 58 \\ 
     &$1.6$           & 29 & 38 & 40 & 52 \\ 
     &$1.9$           & 38 & 63 & 112& 125+ \\ \midrule
     \multirow{5}{*}{$\epsilon=0.01$} & $\omega_C=0.7$  & 28 & 34 & 45 & 53 \\ 
     & $1.0$          & 24 & 30 & 38 & 46 \\ 
     &$1.3$           & 21 & 27 & 32 & 41 \\ 
     &$1.6$           & 28 & 30 & 28 & 37 \\ 
     &$1.9$           & 37 & 58 & 81 & 76 \\  \midrule
     \multirow{5}{*}{$\epsilon=0.1$} & $\omega_C=0.7$ & 16 & 19 & 21 & 23 \\ 
     & $1.0$          & 13 & 16 & 18 & 19 \\ 
     &$1.3$           & 12 & 14 & 16 & 17 \\ 
     &$1.6$           & 15 & 16 & 14 & 16 \\ 
     &$1.9$           & 24 & 29 & 26 & 26 \\ \bottomrule
\end{tabular}
\caption{Multilevel MGRIT iterations for 2D advection-diffusion over various diffusion constants $\epsilon$, with bilinear 1 DG elements, backward Euler in time, FCF-relaxation, and $m=2$.  For the cases labeled ``125+", the solver is still diverging with a convergence rate over 1 at iteration 125.}
\label{tab:order1_CD}
\end{table}

\begin{table}[h!]
\centering
\begin{tabular}{c r|c|c|c|c}
    
     & $N_x \times N_t$ & $768 \times 192$ & $3072 \times 384$ & $12288 \times 768$ & $49152 \times 1536$ \\  \toprule
     \multirow{5}{*}{$\epsilon=0.01$} & $\omega_C=0.7$   & 32 & 31 & 29 & 29 \\ 
     & $1.0$           & 27 & 25 & 25 & 25 \\ 
     &$1.3$            & 25 & 22 & 22 & 22 \\ 
     &$1.6$            & 37 & 43 & 32 & 27 \\  
     &$1.9$            & 52 & 66 & 73 & 68 \\   \midrule
     \multirow{5}{*}{$\epsilon=0.1$} & $\omega_C=0.7$   & 11  & 10  & 10  & 10 \\ 
     &$1.0$            & 9  & 9  & 9  & 9 \\ 
     &$1.3$            & 9  & 8  & 8  & 8 \\ 
     &$1.6$            & 12 & 10 & 10 & 9 \\ 
     &$1.9$            & 19 & 17 & 18 & 15 \\  \bottomrule
\end{tabular}
\caption{Multilevel MGRIT iterations for 2D advection-diffusion over various diffusion constants $\epsilon$, with bi-cubic 3 DG elements, SDIRK-33 in time, FCF-relaxation, and $m=2$. Results for $\epsilon=0.001$ are omitted because all test cases larger than the smallest took 125+ iterations. }
\label{tab:order3_CD}
\end{table}

\subsubsection{A-stable Results}
\label{sec:astable}

Table \ref{tab:order3_Astable_CD} repeats the above experiments for the A-stable 
SDIRK-23 scheme with bi-cubic DG elements in space. We also consider larger $\epsilon$ (i.e., stronger diffusion) as
this highlights the benefits of weighted-relaxation. Results for $\epsilon =
0.001$ are omitted because all test cases larger than the smallest took
125+ iterations. Weights larger than 1.0 are also omitted as they did not improve
convergence.

Consistent with the discussion in Section \ref{sec:verify:vis}, we find that
under-relaxation ($\omega_C < 1.0$) is beneficial, with $\omega_C = 0.7$
providing the best performance.  In fact, in most cases this under-relaxation even restores
convergence compared with unweighted relaxation, where the 125+ label for
$\omega_C = 1.0$ corresponds to a convergence rate larger than one. 
This divergence for $\omega_C = 1.0$ is not surprising, as the work \cite{FrSo2020}
shows that A-stable schemes do not generally yield good MGRIT convergence and
often lead to divergence, even for problems of a parabolic character.

Lastly, we compare Table \ref{tab:order3_Astable_CD} to the convergence plots in Figure \ref{heat map: A-stable SDIRK23}. Convergence for $\omega_C = 0.7$
improves as the problem size increases.  This is most likely due to increasing numerically diffusivity
as the grid is refined, which results in the 
spectrum being pushed into the region of more rapid convergence close to the real axis 
in Figure \ref{heat map: A-stable SDIRK23}.  Additionally, overall performance degrades for larger
$\epsilon$, which is due to the spectrum being pushed out of the region of convergence (i.e., farther up the positive real axis) in Figure \ref{heat map: A-stable SDIRK23}.  
Similarly, as $\epsilon$ decreases, the spectrum is 
pushed to the imaginary axis in Figure \ref{heat map: A-stable SDIRK23}, and convergence eventually degrades, as
is observed for $\epsilon=0.001$.  
For this problem and time-discretization, MGRIT convergence is best for $\epsilon = 0.1$, and interestingly, the advection terms actually \emph{help} MGRIT converge for this problem.

\begin{table}[h!]
\centering
\begin{tabular}{c r|c|c|c|c}
    
     & $N_x \times N_t$ & $192 \times 192$ & $768 \times 384$ & $3072 \times 768$ & $12288 \times 1536$ \\  \toprule
     \multirow{4}{*}{$\epsilon=0.01$} & $\omega_C=0.6$  & 51 & 60 & 55 & 50 \\ 
                                                &$0.7$  & 47 & 54 & 49 & 45 \\ 
                                                &$0.8$  & 43 & 50 & 44 & 42 \\ 
                                                &$1.0$  & 43 & 85 &125+& 125+\\ \midrule
     \multirow{4}{*}{$\epsilon=0.1$} & $\omega_C=0.6$   & 38 & 38 & 32 & 27 \\ 
                                                &$0.7$  & 32 & 32 & 27 & 23 \\ 
                                                &$0.8$  & 36 & 47 & 47 & 42 \\ 
                                                &$1.0$  & 48$^{*}$ & 96$^{*}$&125+&125+ \\ \midrule
     \multirow{4}{*}{$\epsilon=1.0$}  & $\omega_C=0.6$  & 44 & 43 &38  & 30 \\ 
                                                &$0.7$  & 38 & 38 &33  & 26 \\ 
                                                &$0.8$  & 41 & 57 &63  & 53 \\ 
                                                &$1.0$  & 48$^{*}$ & 96$^{*}$ &125+&125+ \\ \midrule
     \multirow{4}{*}{$\epsilon=100.0$}& $\omega_C=0.6$  & 52 & 59 &60  & 59 \\ 
                                                &$0.7$  & 44 & 52 &52  & 51 \\ 
                                                &$0.8$  & 44 & 66 &90  & 98 \\ 
                                                &$1.0$  & 48$^{*}$ & 96$^{*}$ &125+& 125+ \\ \midrule
\end{tabular}
   \caption{Multilevel MGRIT iterations for 2D advection-diffusion over various diffusion constants $\epsilon$, with bi-cubic DG elements, SDIRK-23 in time, FCF-relaxation, and $m=2$.  The asterisk $^{*}$ refers to convergence due only to the exactness property of FCF-relaxation, where FCF-relaxation reproduces sequential time-stepping in $(N_t-1)/2m$ iterations \cite{Fa2014}.  For all cases labeled ``125+", the solver is still diverging with a convergence rate over 1 at iteration 125. }
\label{tab:order3_Astable_CD}
\end{table}

\subsection{Nonlinear Eddy Current Problem}
The last example illustrates the performance of the new relaxation scheme for a nonlinear eddy current problem. The eddy current problem is an approximation of Maxwell's equations that is commonly used in the simulation of electrical machines, such as induction machines, transformers, or cables. Here, we consider a coaxial cable model. Let $\Omega = \Omega_1\cup\Omega_2\cup\Omega_3$ denote a 2D cross-section of the 3D cable model, as depicted in Figure \ref{fig:cable_model}. 

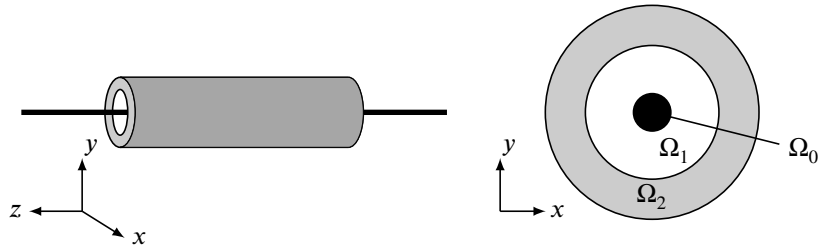
\begin{figure}[h!]
    \centering
	\begin{tikzpicture}
		\draw[semithick,fill=gray!40] (0,0) ellipse (0.2 and 0.46);
		\draw[semithick,fill=white] (0,0) ellipse (0.1 and 0.3);
		\draw[line width=2pt,color=black] (-1.3,0) -- (.1,0);
		\draw[line width=2pt,color=black] (.5,0) -- (4.3,0);
		\draw[semithick] (0,0.46) -- (3,0.46);
		\draw[semithick] (0,-0.46) -- (3,-0.46);
		\draw[semithick] (3,-0.46) arc(-90:90:0.2 and 0.46);
		\filldraw[fill=gray!70, draw=black]
			(0,-.46) -- (3,-.46) arc (-90:90:0.2 and 0.46) -- (0,.46) arc (90:-90:0.2 and 0.46);
			
		\draw[semithick,->,>=latex] (-.5,-1.3) -- (.06,-1.65) node[right=-1pt] {$\small x$};
		\draw[semithick,->,>=latex] (-.5,-1.3) -- (-.5,-.6) node[above right=-3pt] {$\small y$};
		\draw[semithick,->,>=latex] (-.5,-1.3) -- (-1.2,-1.3) node[left=-1pt] {$z$};
		
		\draw[semithick,fill=gray!40] (0+7,0) circle (40pt);
		\draw[semithick,fill=white] (0+7,0) circle (25pt);
		\draw[semithick,fill=black] (0+7,0) circle (7pt);
		
		\draw (0+9,-.5) node (omega0) {$\Omega_0$};
		\draw[semithick] (0+7,0) -- (omega0);
		\draw (0+7.3,-.5) node (omega1) {$\Omega_1$};
		\draw (0+7,-1.1) node (omega2) {$\Omega_2$};
		
		\draw[semithick,->,>=latex] (-2+7,-1.3) -- (-1.4+7,-1.3) node[right=-2pt] {$x$};
		\draw[semithick,->,>=latex] (-2+7,-1.3) -- (-2+7,-.6) node[above right=-3pt] {$y$};
	\end{tikzpicture}
\caption{\label{fig:cable_model} Coaxial cable model and its cross section. The inner, black region $\Omega_0$ models the copper wire, the white region $\Omega_1$ the air insulator and the outer, gray region $\Omega_2$ the conducting shield \cite{FEMM}.
}
\end{figure}

For a voltage-driven system, the eddy current problem is coupled with an additional equation, resulting in the following system for unknown magnetic vector potential $A:\Omega\times (0, T] \to \mathbb{R}$ and the electric current $i_s: (0, T] \to \mathbb{R}$:
\begin{align}
\sigma \partial_t A - \nabla \cdot \big(\nu\nabla A) - \chi_{s} i_s &= 0, \label{eq:ecp} \\
\frac{d}{dt}\int_{\Omega} \chi_{s}\cdot A\; dV & = \upsilon_{s},\label{eq:vi}
\end{align}
with homogeneous Dirichlet boundary condition $A = 0$ on $\partial\Omega$ and the initial value $A(\mathbf{x}, 0) = 0, ~\mathbf{x}\in\Omega$. The electrical conductivity $ \sigma \geq 0$ is only non-zero in the tube region $\Omega_2 \; (\text{here set to } 10$ MS/m$)$, and the (isotropic, nonlinear) magnetic reluctivity $ \nu(\mathbf{x}, |\nabla A|)$ is modeled by a  vacuum $(1 / \mu_0)$ in $\Omega_0$ and $\Omega_1$ and by a monotone cubic spline curve in $\Omega_2$. The current distribution function $\chi_s: \Omega \rightarrow \mathbb{R}$ represents a stranded conductor in the model \cite{schoeps2013}. The relationship between the spatially integrated time derivative of the magnetic vector potential, called flux linkage, and the voltage $v_s$ is modeled by Equation \eqref{eq:vi}. The voltage is a pulsed voltage source, produced by comparing a reference wave with a triangular wave,
\[
    v_s(t) = 0.25\mathrm{sign}\left[r_s(t) - s_n(t)\right], \quad t\in (0,T],
\]
with reference signal
\[
    r_s(t) = \sin\left(\dfrac{2\pi}{T} t\right)
\]
and bipolar trailing-edge modulation using a sawtooth carrier signal
\[
    s_n(t) = \dfrac{n}{T}t - \left\lfloor\dfrac{n}{T}t\right\rfloor,
\]
with $n=200$ teeth and electrical period $T=0.02\;$s \cite{gander2019}. 

We use linear edge shape functions with 2269 degrees of freedom in space to discretize \eqref{eq:ecp}--\eqref{eq:vi}. The resulting system of index-$1$ differential-algebraic equations (DAEs) is integrated on an equidistant time grid with $2^{14}$ intervals using the backward Euler method to resolve the pulses. For each time step $t_j$, we obtain a nonlinear system of the form $\Phi(\mathbf{u}_j)=\mathbf{g}_j$, with $\mathbf{u}^{\!\top}_j=(\mathbf{a}^{\!\top}, i)$ and where $\mathbf{a}$ is the vector of discrete vector potentials and $i$ is an approximation of the current. Considering all time steps at once results in a space-time system of the form $\mathcal{A}(\textbf{u}) = \textbf{g}$, where each block row corresponds to one time step, i.e., the nonlinear extension of equation \eqref{time step eq matrix}. This space-time system is solved using MGRIT V-cycles with a random initial guess, a  residual halting tolerance of $10^{-7}$ and factor-4 coarsening ($m=4$). The method is fully multilevel with the system on the coarsest grid consisting of four time points. 
For all spatial problems, Newton's method is used with a direct LU solver. For the experiments, we use the model \texttt{tube.fem} from the finite element package FEMM\cite{FEMM} and the Python framework PyMGRIT \cite{PyMGRIT,JHahne2021}. 

Figure \ref{fig:cable_conv} shows MGRIT convergence for the eddy current problem and various relaxation weights for FCF- and FCFCF-relaxation\footnote{We note that also for this problem thorough experimentation with under- and over-relaxation using a weight step-size of 0.1 was done.}. The results show that non-unitary weights improve MGRIT convergence for both relaxation schemes. For this particular problem, the best weight choice for FCF-relaxation of $\omega_C = 1.5$ yields a saving of one iteration, or 10\%, over a unitary weight choice. For degree-two relaxation, the experimentally optimal pair of weights $(\omega_C, \omega_{CC}) = (2.0, 0.9)$ even allows for a saving of two iterations, or 22\%, over a unitary weight choice of $(\omega_C, \omega_{CC}) = (1.0, 1.0)$. Again, as for the 2D advection-diffusion problem, the benefits of weighted relaxation on MGRIT convergence for this problem are similar to the benefits observed for the 1D heat equation in Section \ref{sec: Numerial veri bound}. For FCF-relaxation, the best weight choice for 1D diffusion of $\omega_C = 1.3$ results in slightly slower convergence for the 2D eddy current problem, compared to the weight $\omega_C=1.5$, but both weight choices allow for the same saving of one iteration over a unitary weight choice. For FCFCF-relaxation, the best weight choice of $(\omega_C, \omega_{CC}) = (2.0, 0.9)$ corresponds to the best weight choice for 1D diffusion. Thus again, the simple linear 1D model problem provides good guidance for choosing relaxation weights for a more complicated problem, particularly in choosing over- and/or under-relaxation. 
Lastly, comparing total runtimes of MGRIT with weighted FCF- and FCFCF-relaxation with the experimentally optimal weight choices of $\omega_C=1.5$ and $(\omega_C, \omega_{CC}) = (2.0, 0.9)$, respectively, FCF-relaxation is about 4 \% faster than FCFCF-relaxation. For this particular problem, MGRIT with weighted FCF-relaxation is the most efficient solver.

\begin{figure}[h!]
    \centering
    \begin{subfigure}[b]{0.46\textwidth}
    \includegraphics[width=\textwidth]{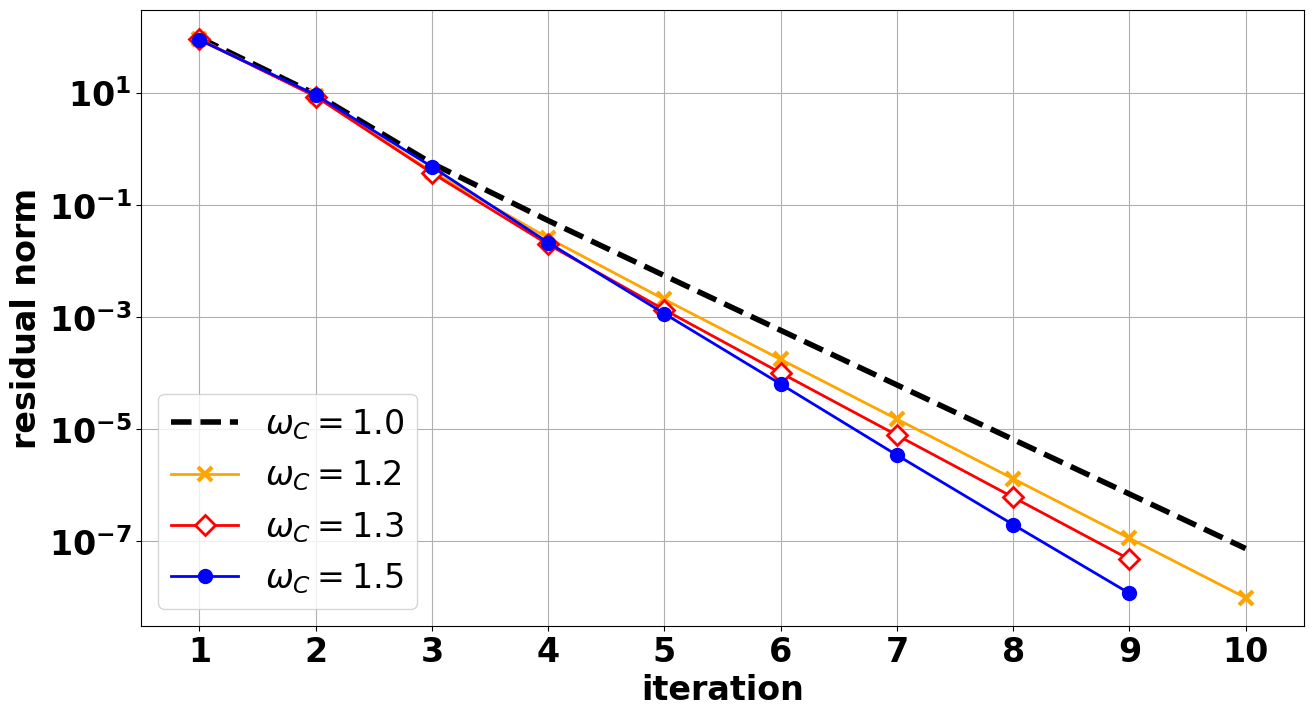}
    \caption{\normalsize FCF-relaxation}
    \label{fig:cable_conv_fcf}
    \end{subfigure}
     \begin{subfigure}[b]{0.46\textwidth}\qquad
    \includegraphics[width=\textwidth]{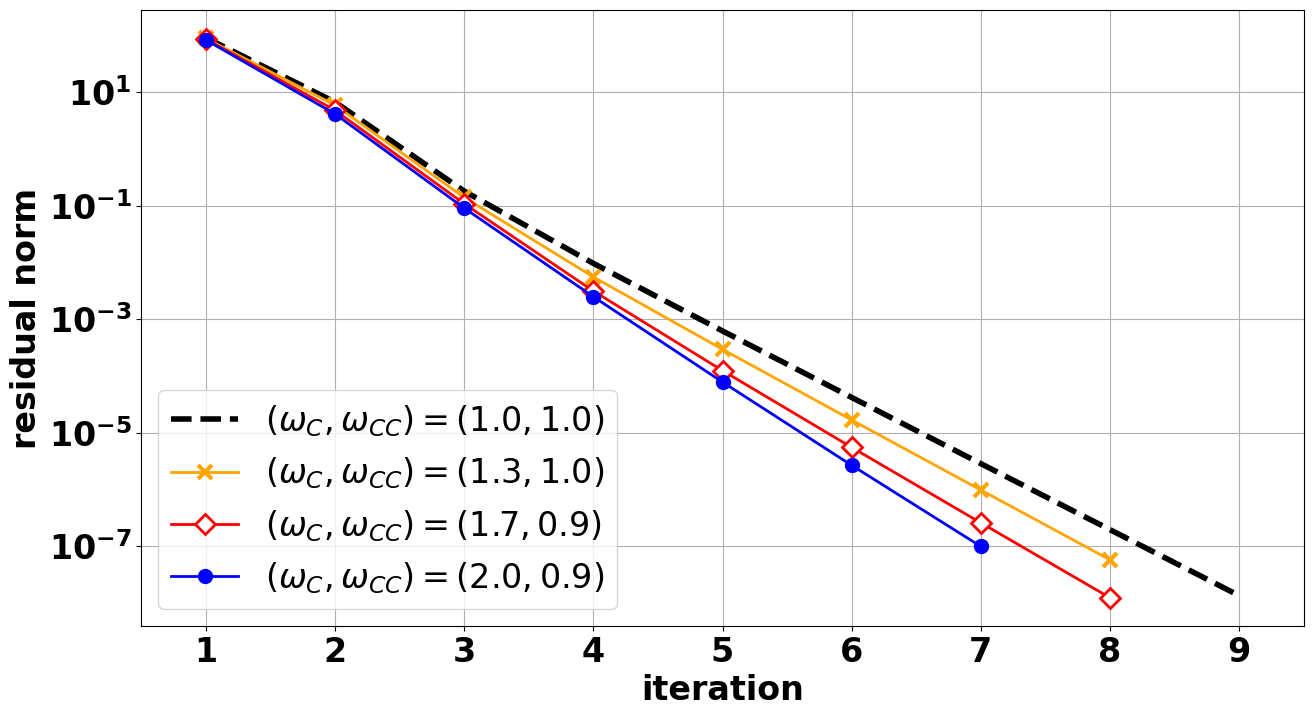}
    \caption{\normalsize FCFCF-relaxation}
    \label{fig:cable_conv_fcfcf}
    \end{subfigure}
    \caption{Experimental MGRIT convergence using weighted FCF- (left) and FCFCF-relaxation (right), $m=4$, and various 
    relaxation weights $\omega_C$ and $\omega_{CC}$ for the eddy current problem.\label{fig:cable_conv}}
\end{figure}


\section{Conclusions}
\label{sec:conc}
In this work, we introduced the concept of weighted relaxation to MGRIT, which until now has used only unweighted relaxation. We derived a new convergence analysis for linear two-grid MGRIT with degree-1 weighted-Jacobi relaxation, and used this analysis to guide and explore  the selection of relaxation weights. The theory was verified with simple numerical examples in \Cref{sec:verify}, and the utility of weighted relaxation was demonstrated on more complex problems in \Cref{sec:results}, including a 2D advection-diffusion problem and a 2D nonlinear eddy current problem. The simple linear 1D model problems from \Cref{sec: Numerial veri bound} provide useful guidance when choosing relaxation weights for more complicated linear and nonlinear problems, and are intended in part to guide future weight choices.

With an appropriate choice of weight, the numerical results demonstrated that MGRIT with weighted relaxation consistently offers improved convergence rates and lower iteration counts when compared with standard (unweighted) MGRIT, at almost no additional computational work. In most cases, weighted relaxation yields a 10\%--20\% savings in iterations, while for the A-stable scheme, the results show that under-relaxation can \emph{restore convergence} in some cases where unweighted relaxation does not converge.

\section*{Acknowledgments}
Los Alamos National Laboratory report number LA-UR-21-26114.







\bibliography{refs}%


\appendix

\section{Bound with FCFCF-Relaxation} \label{app1}
The derivation of the theoretical convergence bound for weighted FCFCF-relaxation (degree-two weighted-Jacobi) is shown in this section.
Remembering expression \eqref{eq:error propagator}, the error propagator for stand-alone weighted FCF-relaxation takes the form
\begin{equation} \label{ex:FCF}
\begin{aligned}
&(I-S(S^TAS)^{-1}S^TA)(I-\omega_C R^T_I(R_IAR^T_I)^{-1}R_IA)(I-S(S^TAS)^{-1}S^TA)\\
&= P(I - \omega_C \mathbf{A}_{\triangle})R_I .
\end{aligned}
\end{equation}
Applying expression \eqref{ex:FCF} twice, once with weight $\omega_C$ and once with another weight $\omega_{CC}$, the error propagator for stand-alone weighted FCFCF-relaxation can be expressed as
\begin{equation}
\begin{aligned} \label{eq:FCFCF}
P(I - \omega_{CC} \mathbf{A}_{\triangle})(I - \omega_C \mathbf{A}_{\triangle})R_I .
\end{aligned}
\end{equation}
Combining the effect of FCFCF-relaxation \eqref{eq:FCFCF} with the previous two-level error propagator \eqref{two level error FCF}, yields the following two-level MGRIT error propagator for FCFCF-relaxation
\begin{equation}
\begin{aligned}
&(I-PB^{-1}_{\triangle}R_IA) P(I - \omega_{CC} \mathbf{A}_{\triangle})(I - \omega_C \mathbf{A}_{\triangle})R_I \\
&= P(I-\mathbf{B}_{\triangle}^{-1}\mathbf{A}_{\triangle})(I - \omega_{CC} \mathbf{A}_{\triangle})(I - \omega_C \mathbf{A}_{\triangle})R_I
\end{aligned}.
\end{equation}
Simplifying the error propagator to consider only C-points yields 
\begin{equation}
\begin{aligned}
E_{\triangle, \hspace{1pt} \{\omega_C,\omega_{CC}\}}^{FCFCF} = (I-\mathbf{B}_{\triangle}^{-1}\mathbf{A}_{\triangle})(I - \omega_{CC} \mathbf{A}_{\triangle})(I - \omega_C \mathbf{A}_{\triangle})
\end{aligned}.
\end{equation}

Similar to Section \ref{sec:proof}, we next use the set of eigenvectors $\{v_\gamma\}$ and 
corresponding eigenvalues $\{\lambda_{\gamma}\}$ of $\Phi$ and  $\{\mu_{\gamma}\}$ of $\Phi_{\triangle}$ to diagonalize $E_{\triangle, \hspace{1pt} \{\omega_C,\omega_{CC}\}}^{FCFCF}$ with the block diagonal eigenvector matrix $\widetilde{U}$.  The resulting matrix $\widetilde{E}_{\triangle, \hspace{1pt} \{\omega_C,\omega_{CC}\}}^{FCFCF}$ is Toeplitz with the following asymptotic generating function,
\begin{align}
\mathcal{F}(x) & \coloneqq (\lambda^m_\gamma - \mu_\gamma)\left[
(1-\omega_{CC})(1-\omega_C)\sum_{\ell=1}^{\infty} \mu^{\ell-1}_\gamma e^{i\ell x} +\{\omega_{CC}(1-\omega_C)+\omega_C(1-\omega_{CC})\}\lambda^m_\gamma \sum_{\ell=2}^{\infty} \mu^{\ell-2}_\gamma e^{i\ell x} + \omega_{CC}\omega_C \lambda^{2m}_\gamma \sum_{\ell=3}^{\infty} \mu^{\ell-3}_\gamma e^{i\ell x}\right] \nonumber \\
& = e^{ix}(\lambda^m_\gamma - \mu_\gamma)\left[(1-\omega_{CC})(1-\omega_C)\sum_{\ell=0}^{\infty} (\mu_\gamma e^{i x})^{\ell} + 
e^{ix}\{\omega_{CC}(1-\omega_C)+\omega_C(1-\omega_{CC})\}\lambda^m_\gamma \sum_{\ell=0}^{\infty} (\mu_\gamma e^{i x})^{\ell} + e^{i2x}\omega_{CC}\omega_C \lambda^{2m}_\gamma \sum_{\ell=0}^{\infty} (\mu_\gamma e^{i x})^{\ell}\right] \nonumber \\
& = e^{ix} \frac{(\lambda^m_\gamma - \mu_\gamma)}{1 - e^{ix}\mu_\gamma}\left[(1-\omega_{CC})(1-\omega_C) + 
e^{ix}\{\omega_{CC}(1-\omega_C)+\omega_C(1-\omega_{CC})\}\lambda^m_\gamma + e^{i2x}\omega_{CC}\omega_C \lambda^{2m}_\gamma \right].
\end{align}

Again following Section \ref{sec:proof}, we bound the maximum singular value of $E_{\triangle, \hspace{1pt} \{\omega_C,\omega_{CC}\}}^{FCFCF}$ with
\begin{align}
\sigma_{max, \gamma}(\widetilde{E}_{\triangle, \hspace{1pt} \{\omega_C,\omega_{CC}\}}^{FCFCF}) & \leq \max_{x\in[0,2\pi]} |\mathcal{F}(x)| \nonumber \\
& = \max_{x\in[0,2\pi]} \frac{|\lambda^m_\gamma - \mu_\gamma|}{|1 - e^{ix}\mu_\gamma|}
|(1-\omega_{CC})(1-\omega_C) + 
e^{ix}\{\omega_{CC}(1-\omega_C)+\omega_C(1-\omega_{CC})\}\lambda^m_\gamma + e^{i2x}\omega_{CC}\omega_C \lambda^{2m}_\gamma| . \nonumber \\
\end{align}
Next by taking the maximum over $\gamma$, we have the following result, similar to Theorem \ref{th:bound},
\begin{equation*}
||E_{\triangle, \hspace{1pt} \{\omega_C,\omega_{CC}\}}^{FCFCF}||_{(\widetilde{U}\widetilde{U}^*)^{-1}} \leq \max_\gamma \max_{x\in[0,2\pi]} \frac{|\lambda^m_\gamma - \mu_\gamma|}{|1 - e^{ix}\mu_\gamma|}
|(1-\omega_{CC})(1-\omega_C) + 
e^{ix}\{\omega_{CC}(1-\omega_C)+\omega_C(1-\omega_{CC})\}\lambda^m_\gamma + e^{i2x}\omega_{CC}\omega_C \lambda^{2m}_\gamma| .
\end{equation*}
Finally, the approximation of the maximum over $x$ yields the theoretical convergence bound for weighted FCFCF-relaxation given in equation \eqref{bd FCFCF-relaxation},
\begin{align}
||E_{\triangle, \hspace{1pt} \{\omega_C,\omega_{CC}\}}^{FCFCF}||_{(\widetilde{U}\widetilde{U}^*)^{-1}} & \lessapprox \max_\gamma \frac{|\lambda^m_\gamma - \mu_\gamma|}{1 - |\mu_\gamma|}
|(1-\omega_{CC})(1-\omega_C) + \{\omega_{CC}(1-\omega_C)+\omega_C(1-\omega_{CC})\}|\lambda^m_\gamma| + \omega_{CC}\omega_C |\lambda^{2m}_\gamma|| \nonumber \\
& = \max_\gamma \frac{|\lambda_\gamma^m - \mu_\gamma|}{1 - |\mu_\gamma|}
|1-\omega_C + \omega_C|\lambda_\gamma^m||\, |1-\omega_{CC} + \omega_{CC}|\lambda_\gamma^m|| .
\end{align}

\clearpage
\noindent\large{\textbf{SUPPLEMENTAL MATERIALS}}
\setcounter{section}{0}
\renewcommand{\thesection}{S\arabic{section}}
\renewcommand{\thetable}{S\arabic{table}}
\renewcommand{\thefigure}{S\arabic{figure}}

\section{Max over $x$}

Here we derive a closed form for the maximum over $x$ that arises
in theoretical bounds to allow easier computation. Consider 
\begin{align}\label{eq:bound0}
\max_{x\in[0,2\pi]} \frac{|\lambda^k - \mu|}{1 - e^{ix}\mu}
	|1-\omega + e^{ix}\omega\lambda^k|.
\end{align}
This function is not differentiable due to the absolute values, but the
maximum is obtained at the same $x$ if we square the underlying function.
Noting that for complex $f$, $|f|^2 = ff^*$; thus, consider
\begin{align}
\max_{x\in[0,2\pi]} & |\lambda^k - \mu|^2
	\frac{(1-\omega + e^{ix}\omega\lambda^k)(1-\omega + e^{-ix}\omega(\lambda^*)^k)}
	{(1 - e^{ix}\mu)(1 - e^{-ix}\mu^*)} \nonumber\\
& = |\lambda^k - \mu|^2 \max_{x\in[0,2\pi]}
	\frac{(1-\omega + e^{ix}\omega\lambda^k)(1-\omega + e^{-ix}\omega(\lambda^*)^k)}
	{(1 - e^{ix}\mu)(1 - e^{-ix}\mu^*)} \nonumber\\
& = |\lambda^k - \mu|^2 \max_{x\in[0,2\pi]}
	\frac{(\omega - 1)^2 + \omega^2|\lambda^k|^2- 2\omega(\omega - 1)\Rea(\lambda^k)\cos(x) +
		2\omega(\omega - 1)\Ima(\lambda^k)\sin(x)}
	{1 + |\mu|^2 - 2\Rea(\mu)\cos(x) + 2\Ima(\mu)\sin(x)} \nonumber \\
& \coloneqq |\lambda^k - \mu|^2 \max_{x\in[0,2\pi]}
	\frac{C_\lambda - 2a\cos(x) + 2b\sin(x)}
	{C_\mu - 2c\cos(x) + 2d\sin(x)} \label{eq:frac}.
\end{align}
Note that by assumption $|\mu| < 1$, which implies $|1 - |\mu|| > 0$, and the
denominator of \eqref{eq:frac} is necessarily nonzero. Thus the function we are
maximizing is well-defined at all $x$ (i.e., has non zero denominator). To find
the maximum, we differentiate in $x$, where
\begin{align*}
\frac{\partial}{\partial x} \frac{C_\lambda - 2a\cos(x) + 2b\sin(x)}
	{C_\mu - 2c\cos(x) + 2d\sin(x)} 
& = \frac{2\sin(x)(aC_\mu - cC_\lambda) + 2\cos(x)(bC_\mu - dC_\lambda) + 4(ad -bc)}
	{(C_\mu - 2c\cos(x) + 2d\sin(x))^2}.
\end{align*}
To set the derivative equal to zero, we only need to worry about the numerator,
so we seek $x$ such that
\begin{align}\label{eq:sol-x}
\sin(x)(aC_\mu - cC_\lambda) + \cos(x)(bC_\mu - dC_\lambda) + 2(ad -bc) = 0.
\end{align}
Note if $\omega = 1$ (unweighted relaxation),
\begin{align}\label{eq:angle}
ad - bc = \omega(1-\omega)(\Rea(\lambda^k)\Ima(\mu) - \Ima(\lambda^k)\Rea(\mu)) = 0,
\end{align}
in which case we can directly compute the solution $x_0$ to \eqref{eq:sol-x}
via the arctangent. The perturbation term in \eqref{eq:angle} arises for
$\omega \neq 1$. If $\mu$ and $\lambda^k$ have the same angle in the complex plane
(i.e., $\mu = C\lambda^k$ for some constant $C$), \eqref{eq:angle} is also
zero, and we arrive at the same solution $x_0$ as when $\omega = 1$. More
generally, we need to account for the case that $\mu$ and $\lambda^k$ are
not the same direction in the complex plane. \emph{Mathematica} provides
the root as
\begin{align}\nonumber
x_0 & \coloneqq 2\arctan\left(\frac{a C_\mu - c C_\lambda \pm \sqrt{a^2 C_\mu^2 - 4 a^2 d^2 + 8 a b c d - 2 a c C_\lambda C_\mu - 4 b^2 c^2 + b^2 C_\mu^2 - 2 bd C_\lambda C_\mu + c^2 C_\lambda^2 + d^2C_\lambda^2 }}{-2(a d - b c) + b C_\mu - dC_\lambda}\right) \\
& = 2\arctan\left(\frac{a C_\mu - c C_\lambda \pm \sqrt{(aC_\mu - cC_\lambda)^2 + 
	(bC_\mu - dC_\lambda)^2 - 4(ad-bc)^2}}{-2(a d - b c) + b C_\mu - dC_\lambda}\right).\label{eq:max}
\end{align}

Now we want to evaluate \eqref{eq:frac} at our maximum, $x_0$. Note
that the maximum in \eqref{eq:max} takes the form $x_0 = 2\arctan(r)$ for
a certain $r$, and recall the identities
\begin{align*}
\cos(2\arctan(r)) = \frac{1-r^2}{1+r^2}, \hspace{5ex}
\sin(2\arctan(r)) = \frac{2r}{1+r^2}.
\end{align*}
Then from \eqref{eq:frac},
\begin{align}\nonumber
\frac{C_\lambda - 2a\cos(2\arctan(r)) + 2b\sin(2\arctan(r))}
	{C_\mu - 2c\cos(2\arctan(r)) + 2d\sin(2\arctan(r))}
& = \frac{C_\lambda - \frac{2a(1-r^2)}{1+r^2} + \frac{4br}{1+r^2}}
	{C_\mu - \frac{2c(1-r^2)}{1+r^2} + \frac{4dr}{1+r^2}} \nonumber\\
& = \frac{C_\lambda(1+r^2) - 2a(1-r^2) + 4br}
	{C_\mu(1+r^2) - 2c(1-r^2) + 4dr} \nonumber\\
& = \frac{(C_\lambda + 2a)r^2 + 4br + C_\lambda - 2a}
	{(C_\mu + 2c)r^2 + 4dr + C_\mu - 2c}.\label{eq:bound1}
\end{align}
Thus to compute the bound in \eqref{eq:bound0}, we first evaluate 
$r$ from \eqref{eq:max},
\begin{align}\label{eq:r}
r \coloneqq \frac{a C_\mu - c C_\lambda \pm \sqrt{(aC_\mu - cC_\lambda)^2 + 
	(bC_\mu - dC_\lambda)^2 - 4(ad-bc)^2}}{-2(a d - b c) + b C_\mu - dC_\lambda},
\end{align}
where 
\begin{align*}
a & = \omega(\omega-1)\Rea(\lambda^k), \\
b & = \omega(\omega-1)\Ima(\lambda^k), \\
c & = \Rea(\mu), \\
d & = \Ima(\mu), \\
C_\mu & = 1 + |\mu|^2 = 1 + c^2+d^2, \\
C_\lambda & = (\omega-1)^2 + \omega^2|\lambda^k|^2 = (\omega-1)^2 +
	\omega^2(\Rea(\lambda^k)^2 + \Ima(\lambda^k)^2).
\end{align*}
We then plug $r$ into \eqref{eq:bound1} and take the square root to
map from \eqref{eq:frac} to \eqref{eq:bound0}.

\section{One-Dimensional Model Problem Results} \label{app2}
This section thoroughly examines weighted-relaxation and MGRIT for three model problems, the 1D heat equation, the 1D advection equation with purely imaginary spatial eigenvalues, and the 1D advection equation with complex spatial eigenvalues.
For full multilevel experiments, V-cycles are used and we coarsen down to a grid of size 4 or less in time.  During searches in the weight-space for experimentally optimal weights, we use a step size of 0.1.  Other testing parameters are discussed below on a case-by-case basis.

Regarding notation, we introduce a level subscript to allow for level-dependent
weights, i.e., $\omega_{C,\ell=k}$ is the weight used on level $k$.  If the
level subscript is omitted, then the weight is uniform across all levels.
For example, $\omega_{C,\ell=0}$ represents the relaxation weight for the first
application of C-relaxation on the finest level 0, and $\omega_{CC,\ell=1}$
represents the relaxation weight for the second application of C-relaxation
(degree two weighted-Jacobi) on the first coarse level 1.

\subsection{One-dimensional heat equation}
\label{sec:results_heat}
We consider the one-dimensional heat equation subject to an initial condition and homogeneous Dirichlet boundary conditions,
\begin{align}
\begin{split}
\frac{\partial u}{\partial t} - \alpha \frac{\partial^2 u}{\partial x^2} &= f(x,t), \hspace{10pt} \alpha > 0, \hspace{10pt} x \in \Omega = [0, L], \hspace{10pt} t \in [0, T], \\
u(x, 0) &= u_0(x), \hspace{10pt} x \in \Omega, \\
u(x, t) &= 0, \hspace{10pt} x \in \partial \Omega, \hspace{10pt} t \in [0, T].
\end{split}
\end{align}
We transform the model problem to a system of ODEs of the form (\ref{ODE}) by using second-order central differencing for discretizing the spatial derivative and then a standard one-step method (backward Euler) of the form (\ref{time step eq}) for discretizing the time derivative.  We call this the \textit{Backward Time, Central Space} or BTCS scheme, which yields
\begin{equation}
\mathbf{u}_j = (I - \delta_t G)^{-1} \mathbf{u}_{j-1} + (I - \delta_t G)^{-1} \delta_t \mathbf{f}_j, \hspace{10pt} j=1,2,...,N_t,
\end{equation}
where the linear operator G in (\ref{ODE}) is the three-point stencil $\frac{\alpha}{h_x^2} [1, -2, 1]$. In 
the form of (\ref{time step eq}), $\Phi = (I - \delta_t G)^{-1}$ and $\mathbf{g}_j = (I - \delta_t G)^{-1} 
\delta_t \mathbf{f}_j $.
The eigenvalues of $\Phi$ and $\Phi^m$ are computed using the eigenvalues of G, i.e.,
$$\kappa_{\gamma} = -\frac{4}{h_x^2} \sin^2 \Big(\frac{\gamma \pi}{2(N_x + 1)}\Big),$$
for $\gamma = 1, 2, ..., N_x$, 
which in turn allows for the computation of the theoretical convergence estimate (\ref{sharper bd}).
For more details on our computation of $\kappa_{\gamma}$, see the work \cite{Do2016}.

The following functions with the given domains are used for numerical experiments,
\begin{subequations}
	\begin{align*}
	u(x, t) &= \sin(\pi x)\cos(t), \\
	f(x,t) &= \sin(\pi x) [\sin(t) - \pi^2 \cos(t)], \\
	\alpha = 1&, \hspace{10pt} x \in [0, 1], \hspace{10pt} t \in [0, 0.625].
	\end{align*}
\end{subequations}
The residual norm halting tolerance for MGRIT is set to $ 10^{-10} / \sqrt{h_x \delta_t}$. 
Reported convergence rates are taken as an average over the last 5 MGRIT
iterations, where $\| r_k\|_2 / \| r_{k-1} \|_2$ is the convergence rate at iteration $k$ and $r_k$ is the residual from equation \eqref{time step eq matrix} at iteration $k$.
The combination of grid points in space $N_x$ and time $N_t$ are chosen so that 
a $\frac{\delta_t}{h_x^2}  = 12.8$.  This value was chosen to be of moderate magnitude and 
consistent with other MGRIT literature, namely the work \cite{Do2016}.

\subsubsection{Weighted FCF- and FCFCF-relaxation}
\label{sec:relax_heat}
We start by considering the two-level method for weighted FCF- and FCFCF-relaxation, i.e., degree-one and 
degree-two relaxation, respectively. Here, the search for the experimentally optimal pair of weights for FCFCF-relaxation and $m=2$ is 
depicted in Figure \ref{fig:Heat FCFCF Two Level}, where $(\omega_{C}, \omega_{CC}) = (1.7, 0.9)$ is the 
point corresponding to the minimal experimental 
convergence rate.  The search space of possible weights is  
$0 \leq \omega_{C}, \omega_{CC} \leq 2.0$, and is based on a more expansive preliminary search.  
A similar study was done in the thesis \cite{Su2019_v2} for 
FCF-relaxation and found that $\omega_{C} = 1.3$ is the point where the minimal 
convergence rate is reached.

\begin{figure}[h!]
    \centering
    \begin{subfigure}[b]{0.4\textwidth}
    \includegraphics[width=\textwidth]{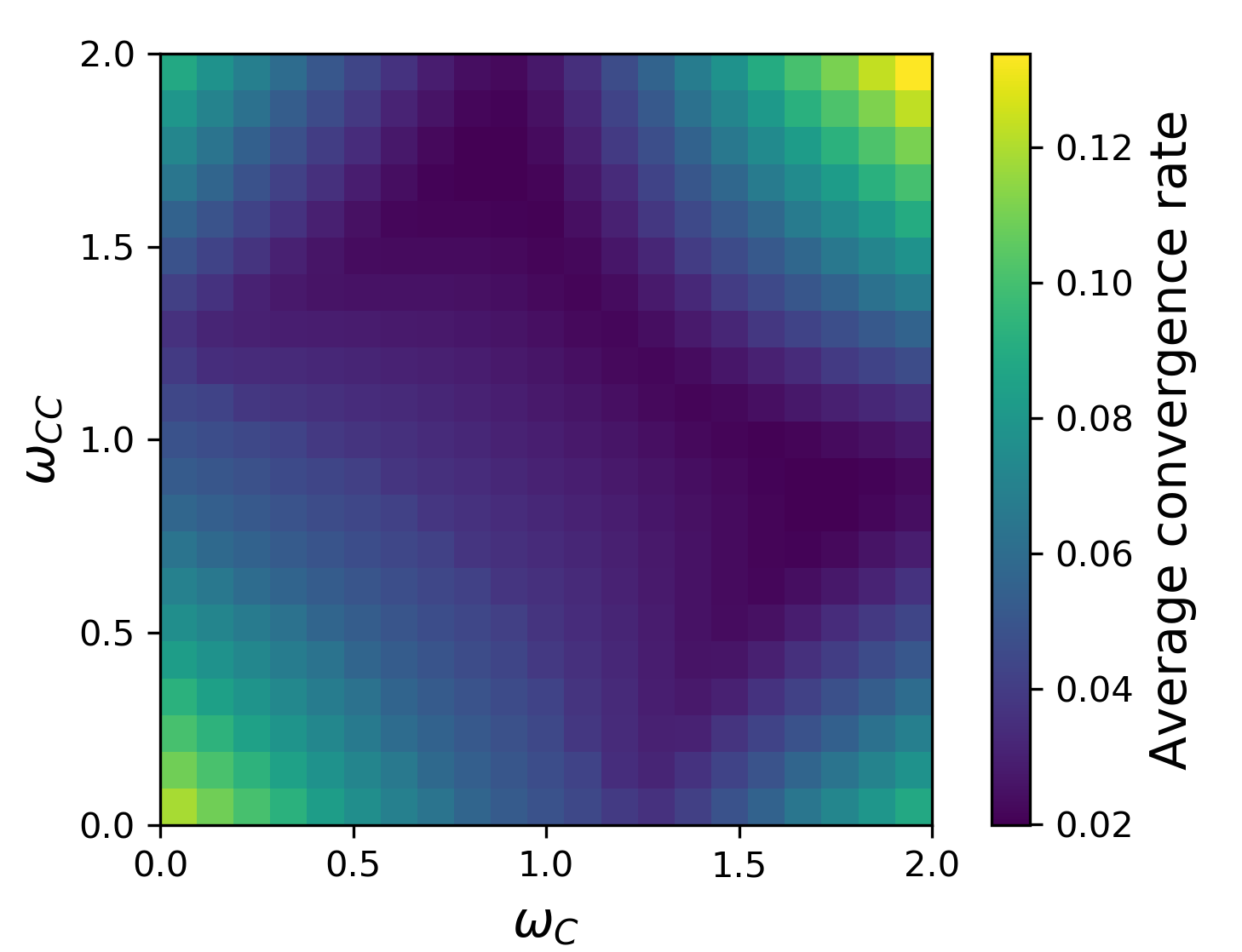}
    \caption{\normalsize Convergence Rate}
    \label{fig:Heat FCFCF Two Level Conv}
    \end{subfigure}
     \begin{subfigure}[b]{0.38\textwidth}
    \includegraphics[width=\textwidth]{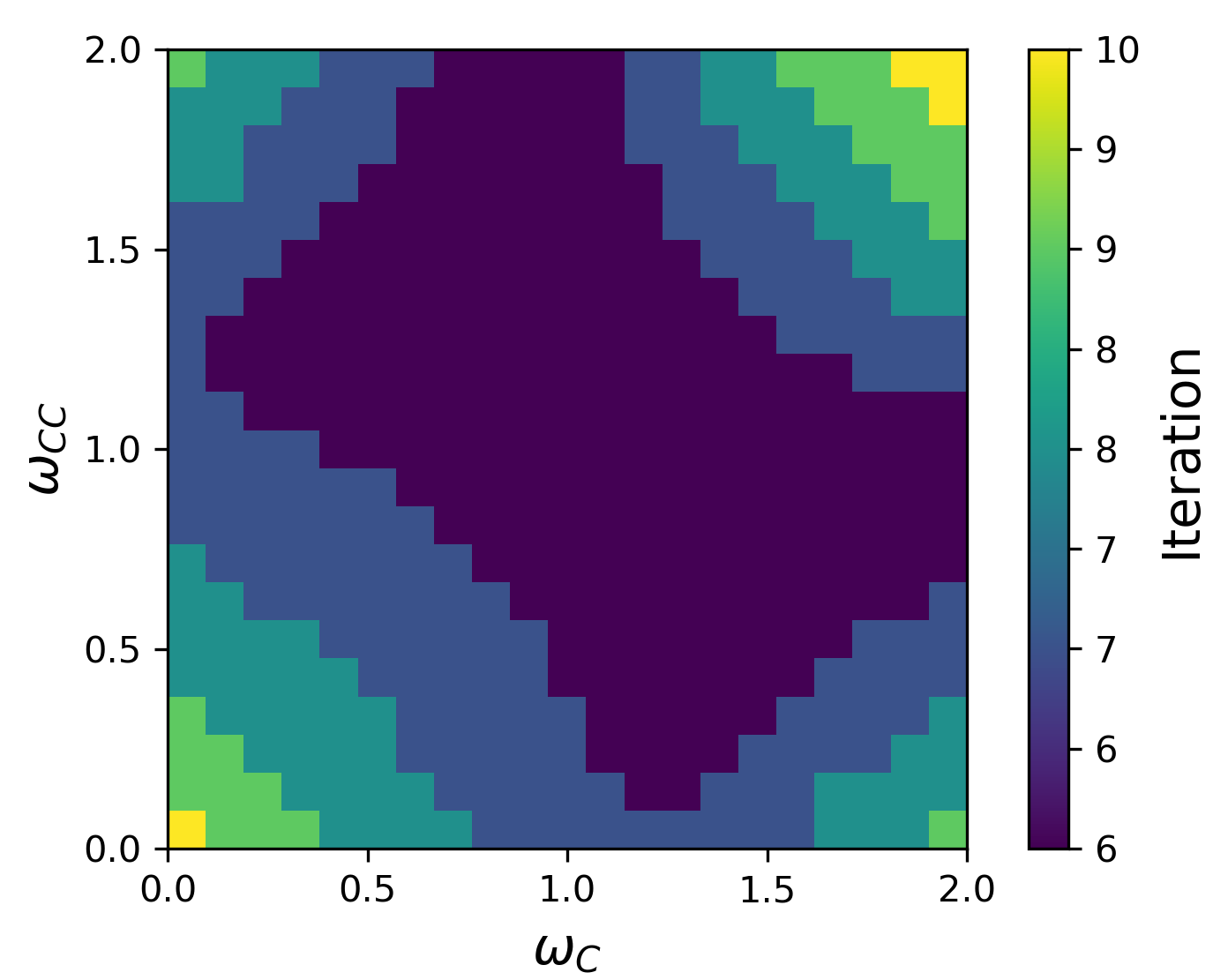}
    \caption{\normalsize Iterations}
    \label{fig:Heat FCFCF Two Level Iter}
    \end{subfigure}
    \caption{Two-level MGRIT experimental convergence rates (left) and iteration counts (right) using FCFCF-relaxation and various 
    relaxation weights $\omega_C$ and $\omega_{CC}$ for the 1D heat equation, coarsening factor $m = 2$, and grid sizes $(N_x, N_t) = (291, 4097)$.}
    \label{fig:Heat FCFCF Two Level}
\end{figure}

Table \ref{tab:Heat Conv and Iter Two level} depicts the convergence rate and iterations
for the two-level case.  Each table entry is formatted as \textit{convergence rate (iterations)}.
The experimentally optimal weights for FCFCF-relaxation $(\omega_{C}, \omega_{CC}) = (1.7, 0.9)$, found using 
$(N_x, N_t) = (291, 4097)$ and $m=2$ above, is highlighted in bold.  This weight choice leads to a saving 
of  1 MGRIT iteration, or 16\%, over unitary weights and FCFCF-relaxation on the largest problem.  
The best weight choice for FCF-relaxation of $\omega_C = 1.3$ yields a saving of 1 iteration, or 14\%, over a unitary weight choice (i.e., $\omega_C = 1.0$) on the largest problem.
At the bottom of the table, we examine whether the experimentally optimal weights for FCF- and FCFCF-relaxation carry over to another 
coarsening factor choice, $m = 16$, and find that this is largely the case.

Table \ref{tab:Heat Conv and Iter Multi level} repeats these experiments for a full multilevel method.
We see that the best two-level choice for FCFCF-relaxation of $(1.7, 0.9)$ still
performs well, but no longer yields the fastest convergence.  Another search of the weight-space for the 
multilevel case yielded
the experimentally optimal pair of weights $(\omega_{C}, \omega_{CC}) = (2.0, 0.9)$ when $m=2$, which allows for 
saving 1 iteration.  The uniform weight choice of $\omega_C = 1.3$ for FCF-relaxation 
continues to save 1 iteration.

Regarding cost, we can say that the cost of relaxation is the dominant cost of each V-cycle \cite{Fa2014}, thus a V-cycle with $m=2$ and 
FCFCF-relaxation has a cost of about $1.66\times$ when compared to a V-cycle using FCF-relaxation.  Furthermore, we can then say that the use of weighted relaxation with FCF-relaxation is the most efficient solver depicted, as the number of iterations (8) for the largest problem size in Table \ref{tab:Heat Conv and Iter Multi level} and weighted FCF-relaxation is noticeably less than 1.66 times the number of iterations for weighted FCFCF-relaxation ($1.66*6 \approx 10$).

\begin{table}[h!]
\centering
\begin{tabular}{c r|c|c|c|c}
    
     & $N_x \times N_t$ & $291 \times 4097$ & $411 \times 8193$ & $581 \times 16385$ & $821 \times 32769$ \\ \toprule
     \multirow{3}{*}{$m=2$} & $\omega_C=1.0$                      & 0.049 (7) & 0.048 (7) & 0.039 (7) & 0.039 (7) \\ 
     & $1.3$                                                      & 0.036 (7) & 0.036 (7) & 0.034 (6) & 0.034 (6) \\ 
     & $1.5$                                                      & 0.048 (7) & 0.049 (7) & 0.049 (7) & 0.049 (7) \\ \midrule
     \multirow{4}{*}{$m=2$} & $(\omega_C,\omega_{CC})=(1.0, 1.0)$ & 0.029 (6) & 0.029 (6) & 0.029 (6) & 0.028 (6) \\ 
     & $(1.3, 1.0)$                                               & 0.025 (6) & 0.024 (6) & 0.024 (6) & 0.023 (6) \\ 
     & $\textbf{(1.7, 0.9)}$                                      & 0.020 (6) & 0.020 (6) & 0.019 (6) & 0.016 (5) \\ 
     & $(2.0, 0.9)$                                               & 0.023 (6) & 0.023 (6) & 0.023 (6) & 0.023 (6) \\  \midrule
     \multirow{2}{*}{$m=16$} & $\omega_C=1.0$                     & 0.101 (9) & 0.099 (8) & 0.099 (8) & 0.099 (8) \\ 
     & $1.3$                                                      & 0.074 (8) & 0.075 (8) & 0.075 (8) & 0.074 (8) \\  \midrule
     \multirow{4}{*}{$m=16$} & $(\omega_C,\omega_{CC})=(1.0, 1.0)$& 0.056 (7) & 0.060 (7) & 0.060 (7) & 0.060 (7) \\ 
     & $(1.3, 1.0)$                                               & 0.049 (7) & 0.053 (7) & 0.053 (7) & 0.053 (7) \\ 
     & $(1.7, 0.9)$                                               & 0.041 (6) & 0.042 (6) & 0.041 (6) & 0.040 (6) \\ 
     & $(2.0, 0.9)$                                               & 0.042 (6) & 0.042 (6) & 0.042 (6) & 0.042 (6) \\ \bottomrule
     \end{tabular}
\caption{Two-level MGRIT convergence rates (iterations) for the 1D heat equation and weighted FCF- and FCFCF-relaxation.}
\label{tab:Heat Conv and Iter Two level}
\end{table}

\begin{table}[h!]
\centering
\begin{tabular}{c r|c|c|c|c}
    
     & $N_x \times N_t$ & $291 \times 4097$ & $411 \times 8193$ & $581 \times 16385$ & $821 \times 32769$ \\   \toprule
     \multirow{2}{*}{$m=2$} & $\omega_C=1.0$                      & 0.118 (9) & 0.121 (9) & 0.123 (9) & 0.125 (9) \\ 
     & $1.3$                                                      & 0.092 (8) & 0.095 (8) & 0.096 (8) & 0.096 (8) \\  \midrule
     \multirow{4}{*}{$m=2$} & $(\omega_C,\omega_{CC})=(1.0, 1.0)$ & 0.065 (7) & 0.066 (7) & 0.067 (7) & 0.068 (7) \\ 
     & $(1.3, 1.0)$                                               & 0.057 (7) & 0.058 (7) & 0.059 (7) & 0.059 (7) \\ 
     & $(1.7, 0.9)$                                               & 0.048 (7) & 0.049 (7) & 0.049 (7) & 0.049 (7) \\ 
     & $\textbf{(2.0, 0.9)}$                                      & 0.032 (6) & 0.032 (6) & 0.032 (6) & 0.032 (6) \\  \midrule
     \multirow{2}{*}{$m=16$} & $\omega_C=1.0$                     & 0.101 (9) & 0.099 (8) & 0.098 (8) & 0.098 (8) \\ 
     & $1.3$                                                      & 0.071 (8) & 0.068 (7) & 0.067 (7) & 0.067 (7) \\  \midrule
     \multirow{4}{*}{$m=16$} & $(\omega_C,\omega_{CC})=(1.0, 1.0)$& 0.056 (7) & 0.060 (7) & 0.060 (7) & 0.060 (7) \\ 
     & $(1.3, 1.0)$                                               & 0.048 (7) & 0.053 (7) & 0.052 (7) & 0.052 (7) \\ 
     & $(1.7, 0.9)$                                               & 0.037 (6) & 0.040 (6) & 0.039 (6) & 0.038 (6) \\ 
     & $(2.0, 0.9)$                                               & 0.041 (6) & 0.041 (6) & 0.041 (6) & 0.041 (6) \\  \bottomrule
\end{tabular}
\caption{Multilevel MGRIT convergence rates (iterations) for the 1D heat equation and weighted FCF- and FCFCF-relaxation.}
\label{tab:Heat Conv and Iter Multi level}
\end{table}

\subsubsection{Multilevel weights for C-relaxation}
\label{sec:Multilevel weights for C-relaxation}

We now consider the effect of level-dependent FCF-relaxation weights on MGRIT.  Weighted FCFCF-relaxation is not 
considered because it is not as efficient as FCF, as discussed in Section \ref{sec:relax_heat}, and the search 
space quickly becomes prohibitive.  Thus, the search for the experimentally optimal pair of weights for 
three-level MGRIT with FCF-relaxation and $m=2$ is depicted in Figure 
\ref{fig: Heat1D Multilevel Weight Three-level}, where $(\omega_{C,\ell=0}, \omega_{C,\ell=1})=(1.0, 2.0)$ 
is the point corresponding to the minimal convergence rate.

\begin{figure}[h!]
    \centering
    \begin{subfigure}[b]{0.4\textwidth}
    \includegraphics[width=\textwidth]{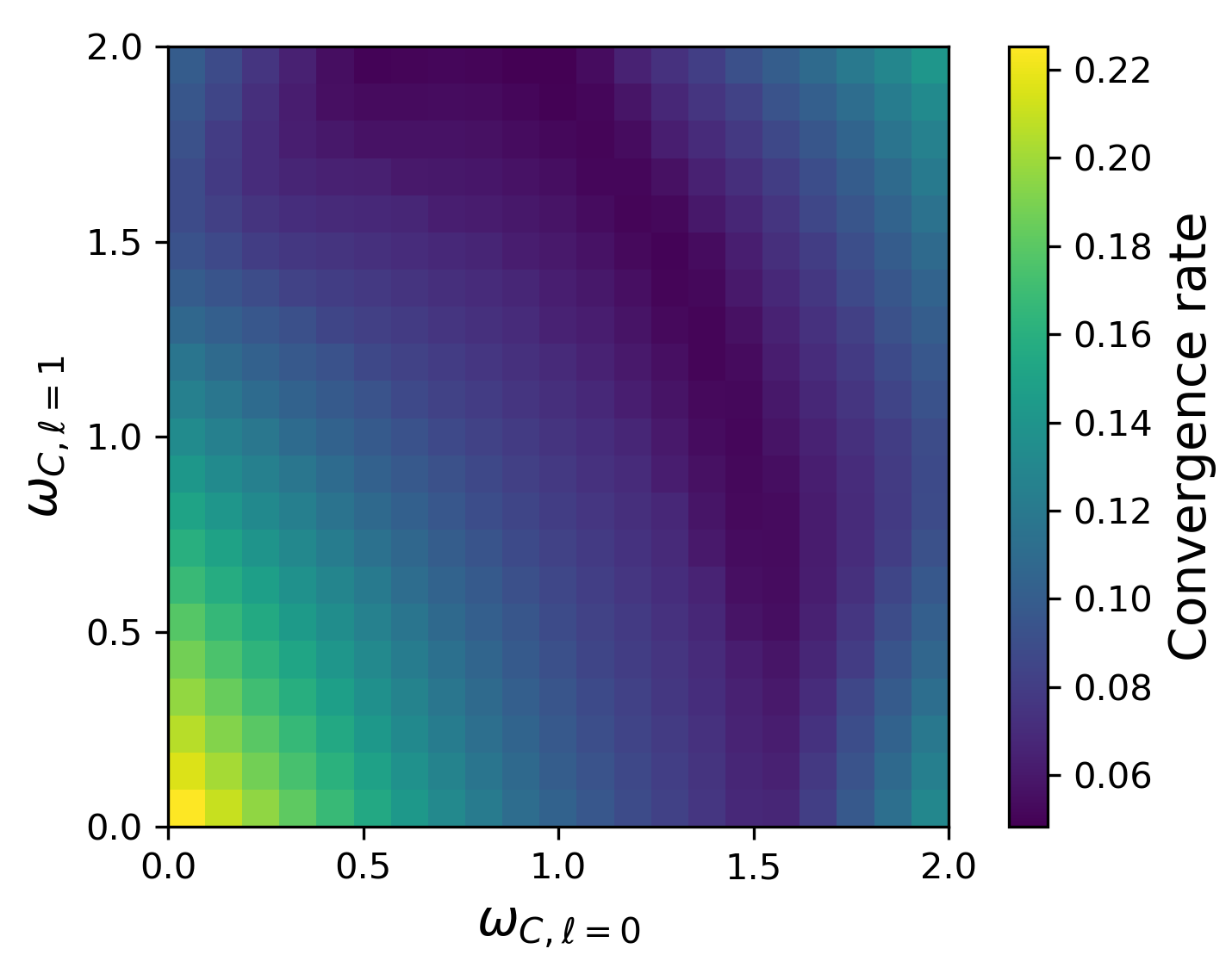}
    \caption{\normalsize Convergence Rate}
    \end{subfigure}
     \begin{subfigure}[b]{0.4\textwidth}
    \includegraphics[width=\textwidth]{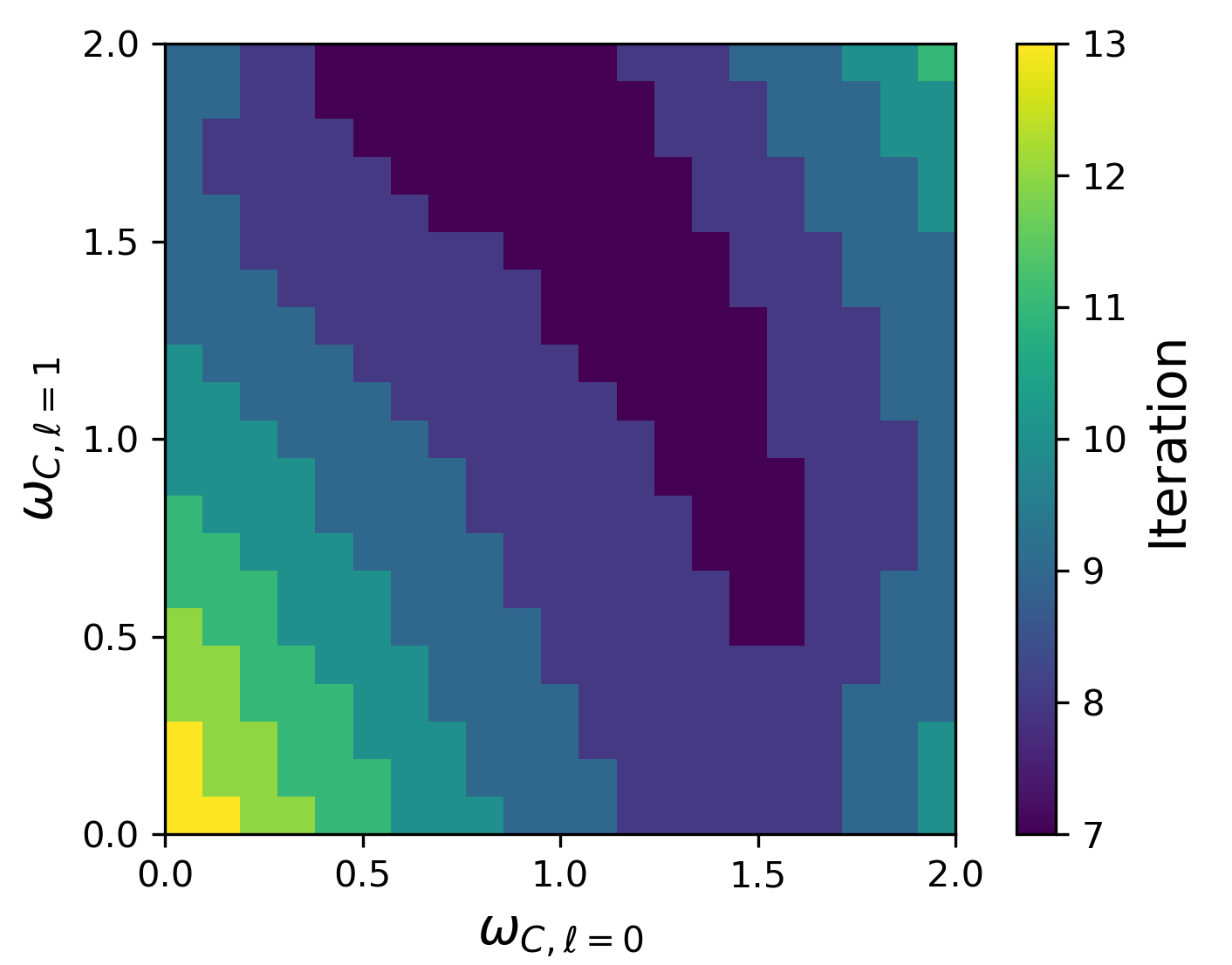}
    \caption{\normalsize Iterations}
    \end{subfigure}
    \caption{Three-level MGRIT experimental convergence rates (left) and iteration counts (right) using level-dependent FCF-relaxation weights $\omega_{C,\ell=0}$ and $\omega_{C,\ell=1}$ for the 1D heat equation, coarsening factor $m=2$, and grid size $(N_x, N_t) = (291, 4097)$.}
    \label{fig: Heat1D Multilevel Weight Three-level}
\end{figure}

Next, we move to a four-level method while keeping fixed the experimentally optimal weights found in Figure 
\ref{fig: Heat1D Multilevel Weight Three-level} and search only for the weight on level three (the second
coarse grid), $\omega_{C,\ell=2}$.  The search for $\omega_{C,\ell=2}$ is depicted in Figure 
\ref{fig: Heat1D Multilevel Weight Four-level}, and the trio of experimentally optimal weights is 
found to be $(\omega_{C,\ell=0}, \omega_{C,\ell=1}, \omega_{C,\ell=2})=(1.0, 2.0, 1.7)$ when $m=2$.

\begin{figure}[h!]
    \centering
    \begin{subfigure}[b]{0.4\textwidth}
    \includegraphics[width=\textwidth]{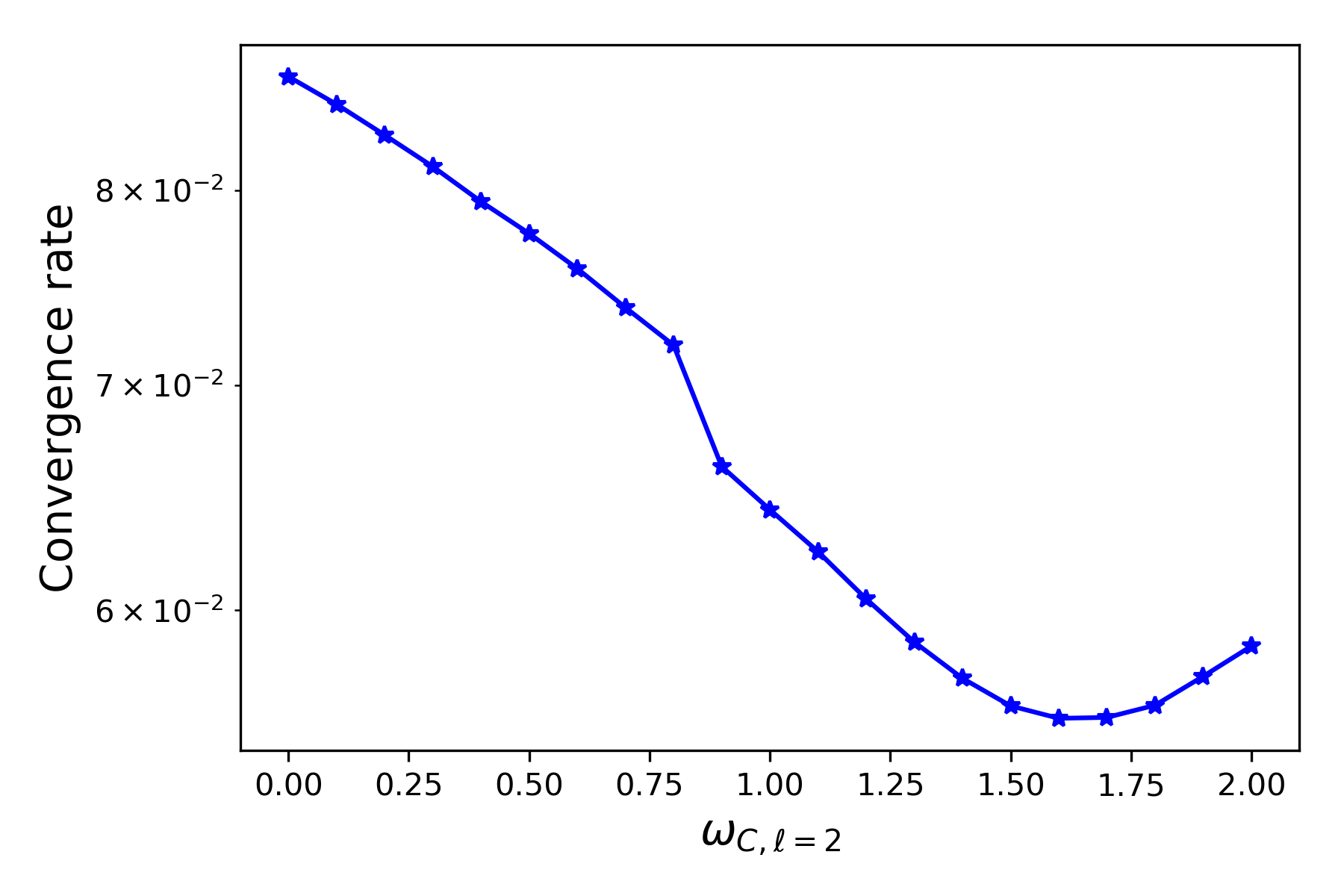}
    \caption{\normalsize Convergence Rate}
    \end{subfigure}
     \begin{subfigure}[b]{0.4\textwidth}
    \includegraphics[width=\textwidth]{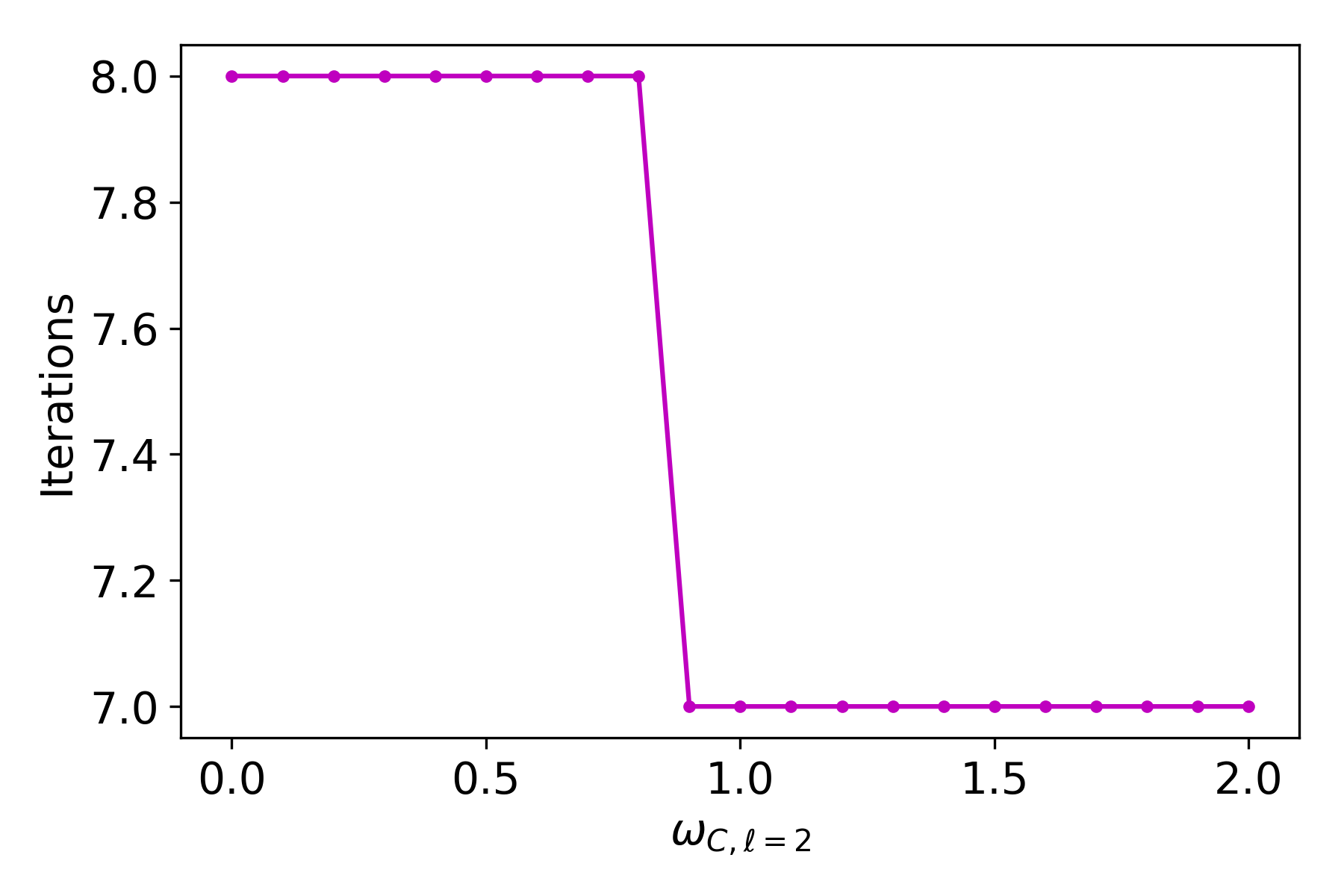}
    \caption{\normalsize Iterations}
    \end{subfigure}
    \caption{Four-level MGRIT convergence rates (left) and iteration counts (right) using FCF-relaxation, as we search 
    for the best level-three relaxation weight $\omega_{C,\ell=2}$, with the fixed values of
    $(\omega_{C,\ell=0}, \omega_{C,\ell=1})=(1.0, 2.0)$ on the first two levels.
    The problem is the 1D heat equation, coarsening factor $m=2$, and grid size $(N_x, N_t) = (291, 4097)$.} 
    \label{fig: Heat1D Multilevel Weight Four-level}
\end{figure}

Table \ref{tab:vary_weight_heat} depicts the convergence rate and iterations for level 
dependent weights, comparing the experimentally ``best" choice of 
$(\omega_{C,\ell=0}, \omega_{C,\ell=1}, \omega_{C,\ell=2})=(1.0, 2.0, 1.7)$ against unitary weights and the best uniform
weight choice of $\omega_C=1.3$.  Level dependent weights provide only a very modest improvement in convergence rate 
with $m=2$ and no benefit in iteration count over the best uniform weight choice of $\omega_C=1.3$.  Additionally, the selected level dependent weights do not 
translate to improved performance for $m=16$, as shown at the bottom of the table.  Thus, we conclude that 
level independent weights for problems similar to the heat equation are likely sufficient.

\begin{table}[h!]
\centering
\begin{tabular}{c r|c|c|c|c}
     & $N_x \times N_t$ & $291 \times 4097$ & $411 \times 8193$ & $581 \times 16385$ & $821 \times 32769$ \\ \toprule
     \multirow{3}{*}{$m=2$} & $(\omega_{C,\ell=0}, \omega_{C,\ell=1}, \omega_{C,\ell=2}) = (1.0, 1.0, 1.0)$  & 0.090 (8) & 0.090 (8) & 0.090 (8) & 0.090 (8) \\ 
     &$\textbf{(1.0, 2.0, 1.7)}$                                                                             & 0.056 (7) & 0.056 (7) & 0.056 (7) & 0.056 (7) \\ 
     &$(1.3, 1.3, 1.3)$                                                                                      & 0.069 (8) & 0.069 (8) & 0.063 (7) & 0.062 (7) \\ \midrule
     \multirow{3}{*}{$m=16$} & $(\omega_{C,\ell=0}, \omega_{C,\ell=1}, \omega_{C,\ell=2}) = (1.0, 1.0, 1.0)$ & 0.101 (9) & 0.099 (8) & 0.098 (8) & 0.098 (8) \\ 
     &$(1.0, 2.0, 1.7)$                                                                                      & 0.087 (8) & 0.086 (8) & 0.087 (8) & 0.087 (8) \\ 
     &$(1.3, 1.3, 1.3)$                                                                                      & 0.071 (8) & 0.068 (7) & 0.067 (7) & 0.067 (7) \\ \bottomrule
     \end{tabular}
    \caption{Four-level MGRIT convergence rates (iterations) for the 1D heat equation with level-dependent weights.}
\label{tab:vary_weight_heat}
\end{table}



\subsubsection{Varying $\delta_t$ experiment}
\label{sec:deltat_diffusion}
Lastly, for the one-dimensional heat equation, we explore the question of why weighted relaxation offers a 
significantly larger convergence benefit for multilevel MGRIT than for two-level MGRIT (compare Tables
\ref{tab:Heat Conv and Iter Two level} and \ref{tab:Heat Conv and Iter Multi level}).  In particular, 
we are interested if the progressively larger $\delta_t$ on coarse grids drives the improved 
performance for weighted relaxation in a multilevel setting.  Thus, Table \ref{tab:heat_vary_dt} depicts the
two-level MGRIT convergence rate for various fine-grid $\delta_t$ values that mimic the $\delta_t$ values 
encountered with $m=2$ on coarse MGRIT levels, when a final time of $0.625$ is used and $N_t=16385$ (i.e., the largest problem from Tables \ref{tab:Heat Conv and Iter Two level} and \ref{tab:Heat Conv and Iter Multi level}).  To further mimic the coarse levels in 
MGRIT, $N_t$ adapts with $\delta_t$, so that the final time is unchanged, e.g., when $\delta t$ has been 
multiplied by 16 in Table \ref{tab:heat_vary_dt}, $N_t$ decreases by a factor of 16 from 4096 to 256.  
However, as evidenced in the table, no MGRIT dependence on $\delta_t$ for 
weighted-relaxation is found, so we conclude that a more complication multilevel interaction is driving
the improved benefit of weighted-relaxation in the multilevel case.

\begin{table}[h!]
\centering
\begin{tabular}{r | c|c|c|c|c}
     $\delta_t$         & $3.81e^{-5}$ & $2 \cdot 3.81e^{-5}$  & $4 \cdot 3.81e^{-5}$  & $8 \cdot 3.81e^{-5}$  & $16 \cdot 3.81e^{-5}$ \\  \toprule        
     Iterations         & 6     & 7     & 7     & 7     & 7     \\
     Convergence Rate   & 0.034 & 0.036 & 0.036 & 0.036 & 0.036 
\end{tabular}
\caption{Two-level MGRIT with $\omega_{C}=1.3$ and $m=2$ for various fine-grid $\delta_t$ values for the 1D heat equation.}
\label{tab:heat_vary_dt}
\end{table}

\subsection{One-dimensional linear advection equation with purely imaginary spatial eigenvalues}
\label{sec:results_adv_imag} 
We now consider the one-dimensional linear advection equation subject to an initial condition and periodic boundary conditions,
\begin{align}
\begin{split}
\frac{\partial u}{\partial t} - \alpha \frac{\partial u}{\partial x} &= 0, 
\hspace{10pt} \alpha > 0, \hspace{10pt} x \in \Omega = [0, L], \hspace{10pt} t \in [0, T], \\
u(x, 0) &= u_0(x), \hspace{10pt} x \in \Omega, \\
u(0, t) &= u(L, t), \hspace{10pt} t \in [0, T].
\end{split}
\end{align}
If we apply the BTCS scheme, we obtain
\begin{equation*}
\mathbf{u}_j = (I - \delta_t G)^{-1} \mathbf{u}_{j-1}, \hspace{10pt} j=1,2,...,N_t,
\end{equation*}
where the linear operator G from (\ref{ODE}) is the two-point stencil $\frac{\alpha}{2h_x} [-1, 0, 1]$. Here, $\Phi = (I - \delta_t G)^{-1}$ and $\mathbf{g}_j = 0$. Similar to the heat equation, the eigenvalues of $\Phi$ and $\Phi^m$ are computed from the eigenvalues of $G$, i.e., 
$$\kappa_{\gamma} = \frac{i}{h_x} \sin \left(\frac{2 \pi \gamma}{N_x}\right),$$
for $\gamma = 1, 2, ..., N_x$, 
which in turn allows for the computation of the theoretical convergence estimate (\ref{sharper bd}).

The following function with the given domain is used for numerical experiments, 
\begin{subequations}
\begin{align}
u(x, t) = e^{-25((x - t) - 0.5)^2}, \label{eq:LA_eqn1} \\
\alpha = 1, \hspace{10pt} x \in [0, 1], \hspace{10pt} t \in [0, 1]  \label{eq:LA_eqn2}.
\end{align}
\end{subequations}
The function is chosen as a standard test problem that satisfies the spatially periodic boundary conditions.
The MGRIT residual norm halting tolerance is set to $ 10^{-8} / \sqrt{h_x \delta_t}$ and the maximum allowed
iterations is set to $70$, because some cases will fail to quickly converge. 
Reported convergence rates are taken as $( \| r_k \|_2 / \| r_0 \|_2 )^{1/k}$ at the final iteration $k$, where $r_i$ is the residual from equation \eqref{time step eq matrix} at iteration $i$.  The combination of grid points in space $N_x$ and time $N_t$ are chosen so that $\frac{\delta_t}{h_x} = 0.5$.

\subsubsection{Weighted FCF- and FCFCF-relaxation}
We again start by considering the two-level method for weighted FCF- and FCFCF-relaxation. The search for 
the experimentally optimal pair of weights for FCFCF-relaxation and $m=2$ is depicted in Figure 
\ref{fig: AdvcC FCFCF-relaxation}, where $(\omega_C, \omega_{CC}) = (1.0, 2.3)$ is the point corresponding to the minimal
convergence rate.  The search space of weights is widened to $0 \le \omega_C, \omega_{CC} \le 3$, because 
a more expansive preliminary search indicated this was a reasonable range. A similar 
study was done in the thesis \cite{Su2019_v2} for FCF-relaxation and found that 
$\omega_{C} = 1.8$ is the point where the minimal convergence rate is reached.

Table \ref{tab:LA Conv and Iter for Twolevel} depicts the convergence rate and iterations for 
the two-level case.  The experimentally optimal pair of weights for FCFCF-relaxation $(\omega_C, \omega_{CC}) = (1.0, 2.3)$, found
in Figure \ref{fig: AdvcC FCFCF-relaxation}, is highlighted in bold, and this choices leads to saving 1 iteration, or 
7\% over unitary weights and FCFCF-relaxation on the largest problem.  The best weight choice for FCF-relaxation of $\omega_C = 1.8$
yields a saving of 1 iteration, or 7\%, over a unitary weight choice on the largest problem. 
At the bottom of the table, we examine whether the experimentally optimal weights carry over 
to another coarsening factor, $m=4$, and find that this is not the case, in contrast to the heat equation.  
MGRIT for advection problems is typically sensitive to changes in $m$ (as opposed to the heat equation) \cite{Do2016, HoDeFaMaSc2019}, hence we do not consider $m=16$ or other large coarsening factors.

Table \ref{tab:LA Conv and Iter for Multi level} repeats these experiments for a full multilevel method.
We see that the best two-level choice for FCFCF-relaxation of 
$(\omega_C, \omega_{CC}) = (1.0, 2.3)$ fails to provide a benefit for larger problems 
in the multilevel setting.  Thus, we carry out 
another search in the weight-space and find that $(\omega_C, \omega_{CC}) = (2.3, 0.6)$ (in bold) yields 
the fastest convergence when $m=2$, saving 25\% of the iterations over unitary weights 
$(\omega_C, \omega_{CC}) = (1.0, 1.0)$ on the largest problem.  A search in the weight-space
for FCF-relaxation yielded the best convergence rate when $\omega_C = 1.5$, 
saving 22\% of the iterations on the second largest problem.  At the bottom of the table, we 
show that the best weight choices for $m=2$ do not carry over to $m=4$.  The choice of $\omega_C =
1.4$ for FCF-relaxation is depicted to illustrate the performance for the best weight choice found in 
that case.

Overall, we note that linear advection is traditionally difficult for MGRIT \cite{Do2016, HoDeFaMaSc2019}, so 
while these iteration counts with experimentally optimal weights are not scalable, we view any significant 
improvement in convergence as an important step.

\begin{figure}[h!]
    \centering
    \begin{subfigure}[b]{0.4\textwidth}
    \includegraphics[width=\textwidth]{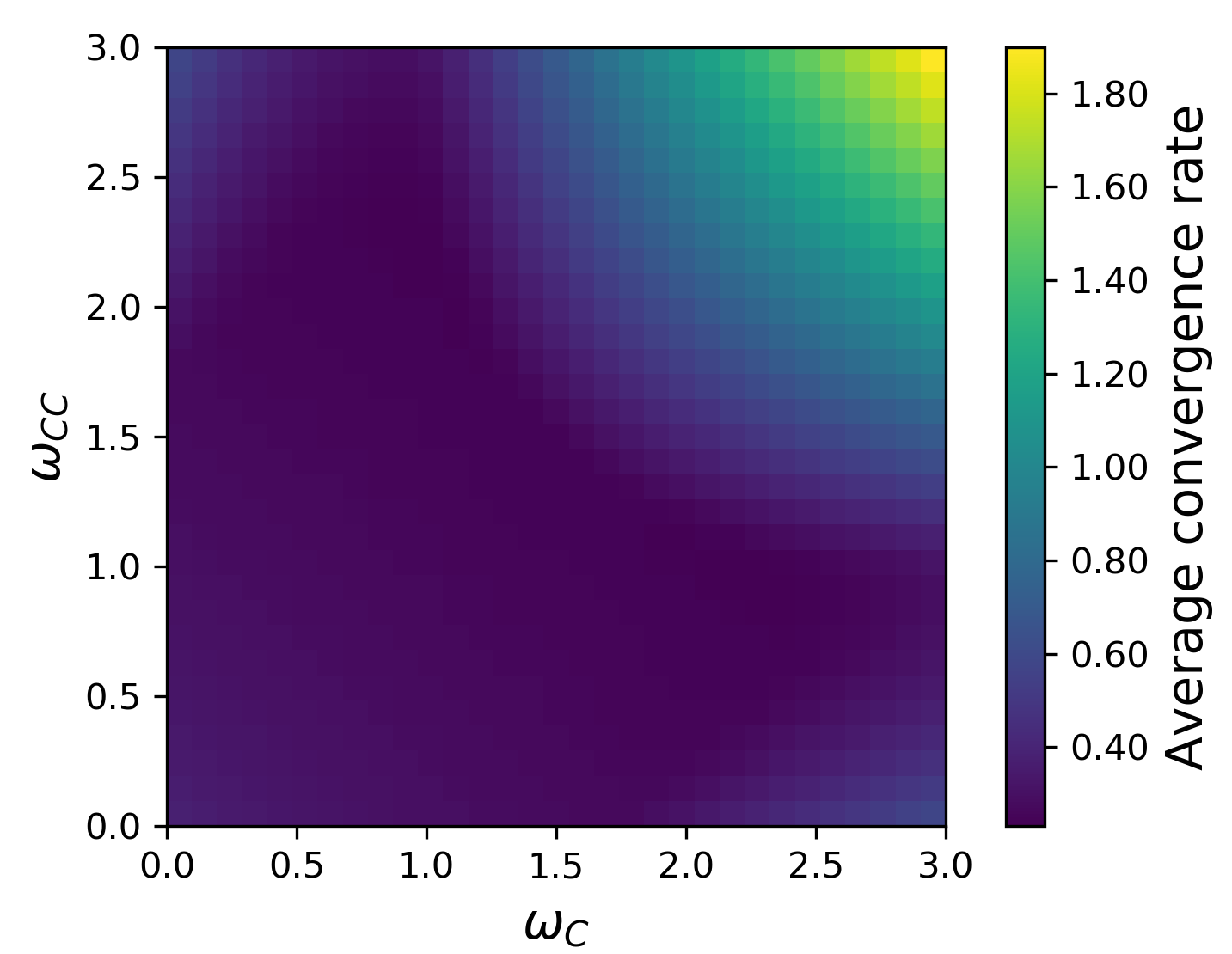}
    \caption{\normalsize Convergence Rate}
    \end{subfigure}
     \begin{subfigure}[b]{0.4\textwidth}
    \includegraphics[width=\textwidth]{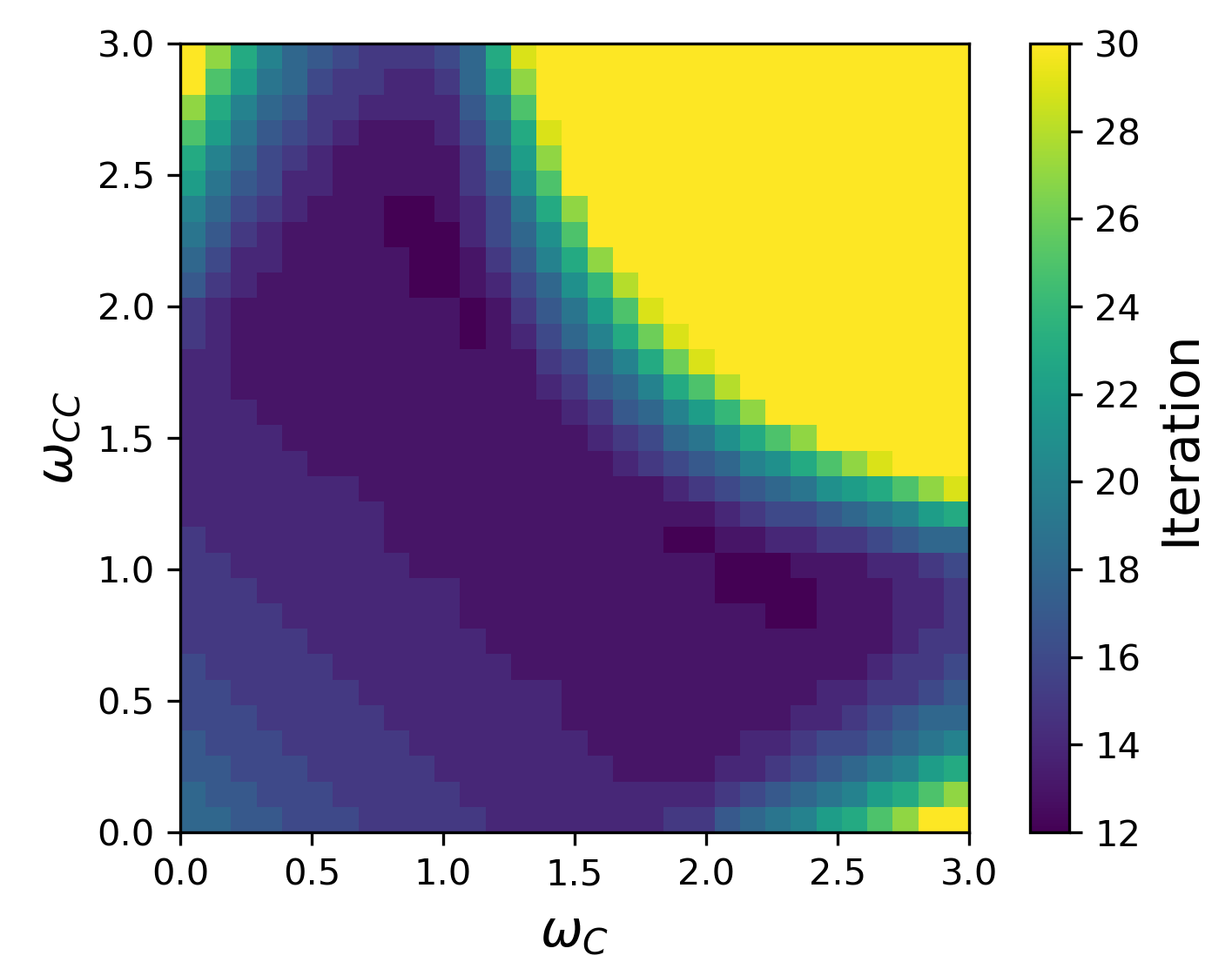}
    \caption{\normalsize Iterations}
    \end{subfigure}
    \caption{Two-level MGRIT experimental convergence rates (left) and iteration counts (right) using FCFCF-relaxation and various 
    relaxation weights $\omega_C$ and $\omega_{CC}$ for the 1D linear advection equation, coarsening factor $m=2$, and grid size $(N_x, N_t) = (1025, 1025)$.}
    \label{fig: AdvcC FCFCF-relaxation}
\end{figure}

\begin{table}[h!]
\centering
\begin{tabular}{l r|c|c|c|c}
     
     & $N_x \times N_t$ & $513 \times 513$ & $1025 \times 1025$ & $2049 \times 2049$ & $4097 \times 4097$ \\ \toprule
     \multirow{2}{*}{$m=2$} & $\omega_C = 1.0$                   & 0.304 (15) & 0.307 (15) & 0.308 (15) & 0.309 (15) \\ 
     & $1.8$                                                     & 0.280 (14) & 0.282 (14) & 0.284 (14) & 0.285 (14) \\ \midrule
     \multirow{4}{*}{$m=2$} &$(\omega_C,\omega_{CC})=(1.0, 1.0)$ & 0.263 (13) & 0.266 (13) & 0.268 (13) & 0.278 (14) \\ 
     & $(1.8, 1.0)$                                              & 0.249 (13) & 0.254 (13) & 0.257 (13) & 0.257 (13) \\ 
     & $\textbf{(1.0, 2.3)}$                                     & 0.237 (12) & 0.250 (13) & 0.251 (13) & 0.252 (13) \\ 
     & $(2.3, 0.6)$                                              & 0.238 (12) & 0.254 (13) & 0.256 (13) & 0.256 (13) \\ \midrule
     \multirow{2}{*}{$m=4$} & $\omega_C =1.0$                    & 0.564 (30) & 0.607 (34) & 0.617 (35) & 0.619 (35) \\ 
     & $1.8$                                                     & 0.763 (63) & 0.777 (67) & 0.780 (68) & 0.780 (68) \\ 
     & $1.5$                                                     & 0.568 (30) & 0.581 (31) & 0.591 (32) & 0.596 (33) \\ \midrule
     \multirow{4}{*}{$m=4$} & $(\omega_C,\omega_{CC})=(1.0, 1.0)$& 0.473 (23) & 0.537 (27) & 0.557 (29) & 0.566 (30) \\ 
     & $(1.5, 1.0)$                                              & 0.448 (21) & 0.511 (25) & 0.537 (27) & 0.546 (28) \\ 
     & $(1.0, 2.3)$                                              & 0.655 (40) & 0.675 (43) & 0.679 (44) & 0.680 (44) \\ 
     & $(2.3, 0.6)$                                              & 0.643 (38) & 0.660 (41) & 0.663 (41) & 0.664 (41) \\ \bottomrule
\end{tabular}
\caption{Two-level MGRIT convergence rates (iterations) for the 1D linear advection equation and weighted FCF- and FCFCF-relaxation.}
\label{tab:LA Conv and Iter for Twolevel}
\end{table}

\begin{table}[h!]
\centering
\begin{tabular}{c r|c|c|c|c}
    &$N_x \times N_t$ & $513 \times 513$ & $1025 \times 1025$ & $2049 \times 2049$ & $4097 \times 4097$ \\ \toprule 
    \multirow{2}{*}{$m=2$} &   $\omega_C = 1.0$                 & 0.560 (30) & 0.675 (44) & 0.771 (67) & 0.854 (> 100) \\ 
    & $1.5$                                                     & 0.495 (24) & 0.606 (35) & 0.718 (52) & 0.810 (82) \\ \midrule
    \multirow{4}{*}{$m=2$}& $(\omega_C,\omega_{CC})=(1.0, 1.0)$ & 0.464 (23) & 0.576 (32) & 0.678 (45) & 0.765 (64) \\ 
    & $(1.5, 1.0)$                                              & 0.423 (20) & 0.542 (29) & 0.646 (40) & 0.738 (57) \\ 
    & $(1.0, 2.3)$                                              & 0.452 (22) & 0.605 (35) & 0.744 (59) & 0.858 (>100) \\ 
    & $\textbf{(2.3, 0.6)}$                                     & 0.390 (19) & 0.492 (25) & 0.603 (34) & 0.696 (48) \\ \midrule
    \multirow{2}{*}{$m=4$} & $\omega_C = 1.0$                   & 0.581 (32) & 0.666 (42) & 0.757 (61) & 0.838 (95) \\ 
    & $1.4$                                                     & 0.535 (27) & 0.611 (34) & 0.712 (50) & 0.802 (77) \\ \midrule
    \multirow{4}{*}{$m=4$}& $(\omega_C,\omega_{CC})=(1.0, 1.0)$ & 0.476 (23) & 0.577 (31) & 0.677 (43) & 0.774 (66) \\ 
                                                                & $(1.4, 1.0)$ & 0.448 (22) & 0.544 (28) & 0.643 (39) & 0.752 (60) \\ 
                                                                & $(1.0, 2.3)$ & 0.658 (41) & 0.683 (44) & 0.761 (63) & 0.884 (>100) \\ 
                                                                & $(2.3, 0.6)$ & 0.607 (34) & 0.640 (38) & 0.758 (62) & 0.860 (>100) \\ \bottomrule
\end{tabular}
\caption{Multilevel MGRIT convergence rates (iterations) for the 1D linear advection equation and weighted FCF- and FCFCF-relaxation.}
\label{tab:LA Conv and Iter for Multi level}
\end{table}

\subsubsection{Multilevel weights for C-relaxation}
\label{sec:ml_w_LA}
We again consider the effect of level-dependent FCF-relaxation weights on MGRIT, similar to the heat equation.  
Weighted FCFCF-relaxation is again not considered due to its cost and size of search space. Thus, the search
for the experimentally optimal pair of weights for three-level MGRIT with FCF-relaxation and $m=2$ is depicted
in Figure \ref{fig: AdvcC Multilevel Weight Three-level}, where $(\omega_{C,\ell=0}, \omega_{C,\ell=1})=(1.3, 2.0)$
is the point corresponding to the minimal convergence rate.

\begin{figure}[h!]
    \centering
    \begin{subfigure}[b]{0.4\textwidth}
    \includegraphics[width=\textwidth]{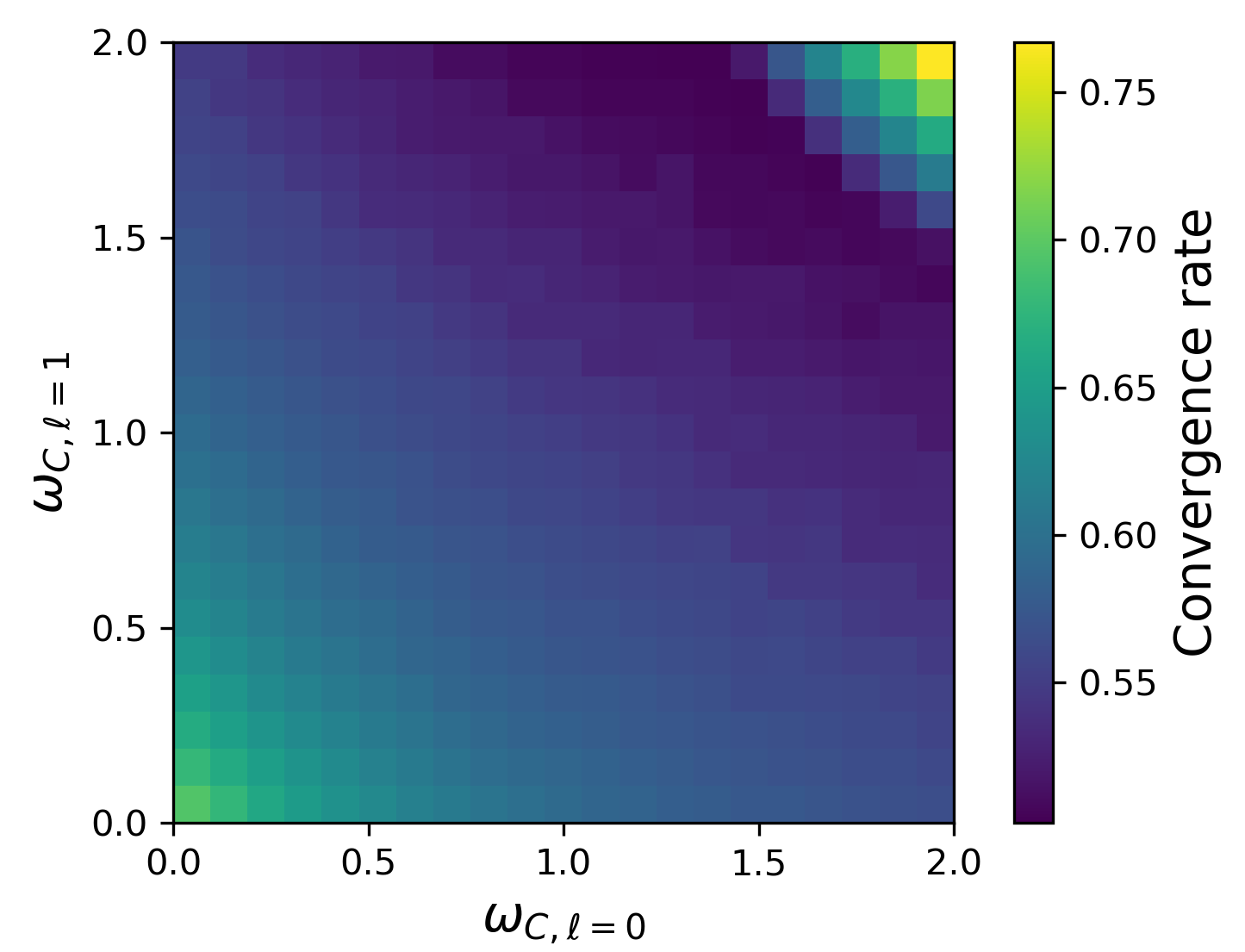}
    \caption{\normalsize Convergence Rate}
    \end{subfigure}
     \begin{subfigure}[b]{0.38\textwidth}
    \includegraphics[width=\textwidth]{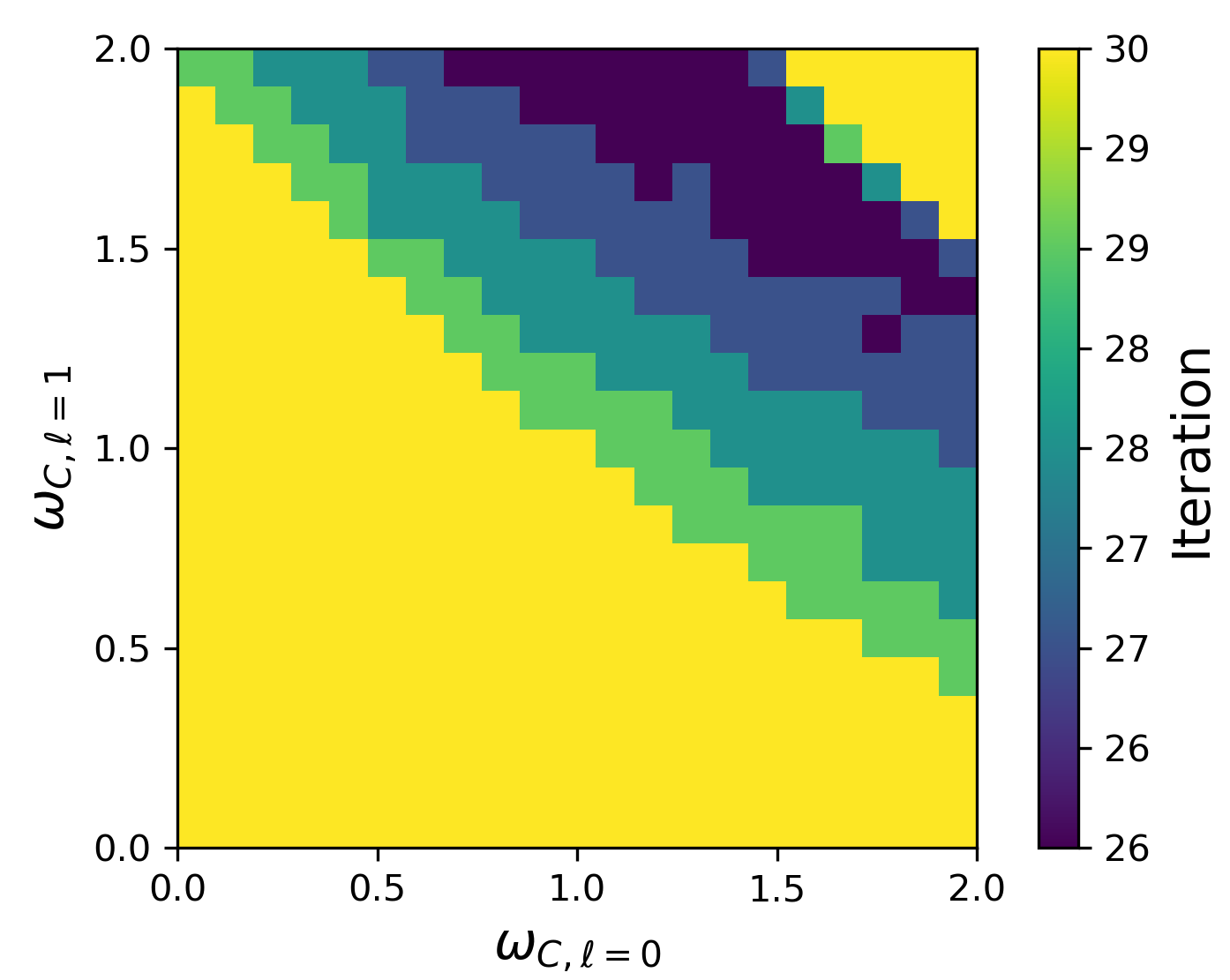}
    \caption{\normalsize Iterations}
    \end{subfigure}
    \caption{Three-level MGRIT experimental convergence rates (left) and iteration counts (right) using various level-dependent FCF-relaxation weights $\omega_{C,\ell=0}$ and $\omega_{C,\ell=1}$ for the 1D linear advection equation, coarsening factor $m=2$, and grid size $(N_x, N_t) = (1025, 1025)$.} 
    \label{fig: AdvcC Multilevel Weight Three-level}
\end{figure}

Next, we move to a four-level method while keeping fixed the experimentally optimal weights found in 
Figure \ref{fig: AdvcC Multilevel Weight Three-level} and search only for the weight on level three
(the second coarse grid), $\omega_{C,\ell=2}$.  This search is depicted in Figure 
\ref{fig: Advc1D_C Multilevel Weight Four-level} and the trio of experimentally optimal
weights is found to be $(\omega_{C,\ell=0}, \omega_{C,\ell=1}, \omega_{C,\ell=2})=(1.3, 2.0, 1.7)$ when $m=2$.

\begin{figure}[h!]
    \centering
    \begin{subfigure}[b]{0.4\textwidth}
    \includegraphics[width=\textwidth]{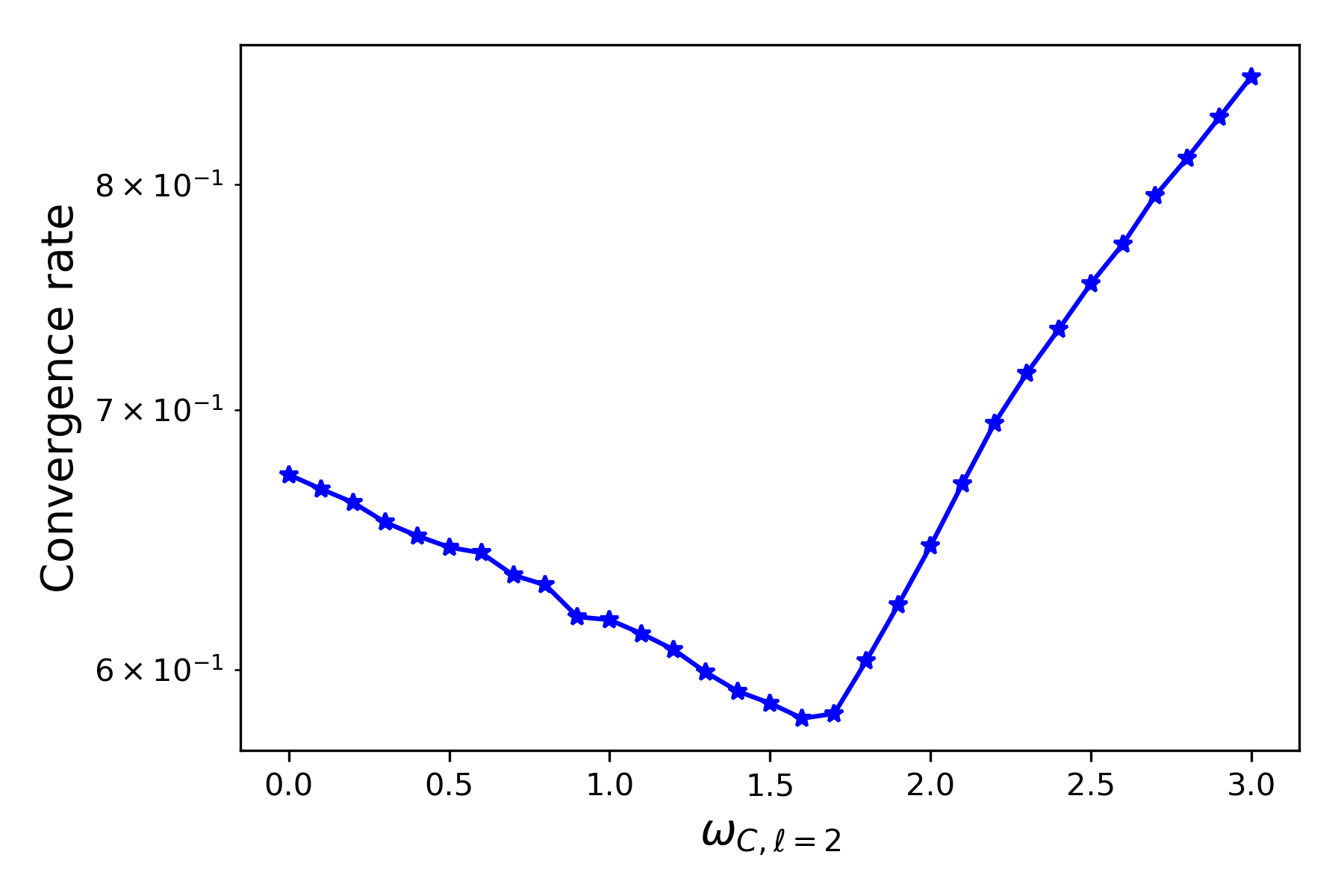}
    \caption{\normalsize Convergence Rate}
    \end{subfigure}
     \begin{subfigure}[b]{0.4\textwidth}
    \includegraphics[width=\textwidth]{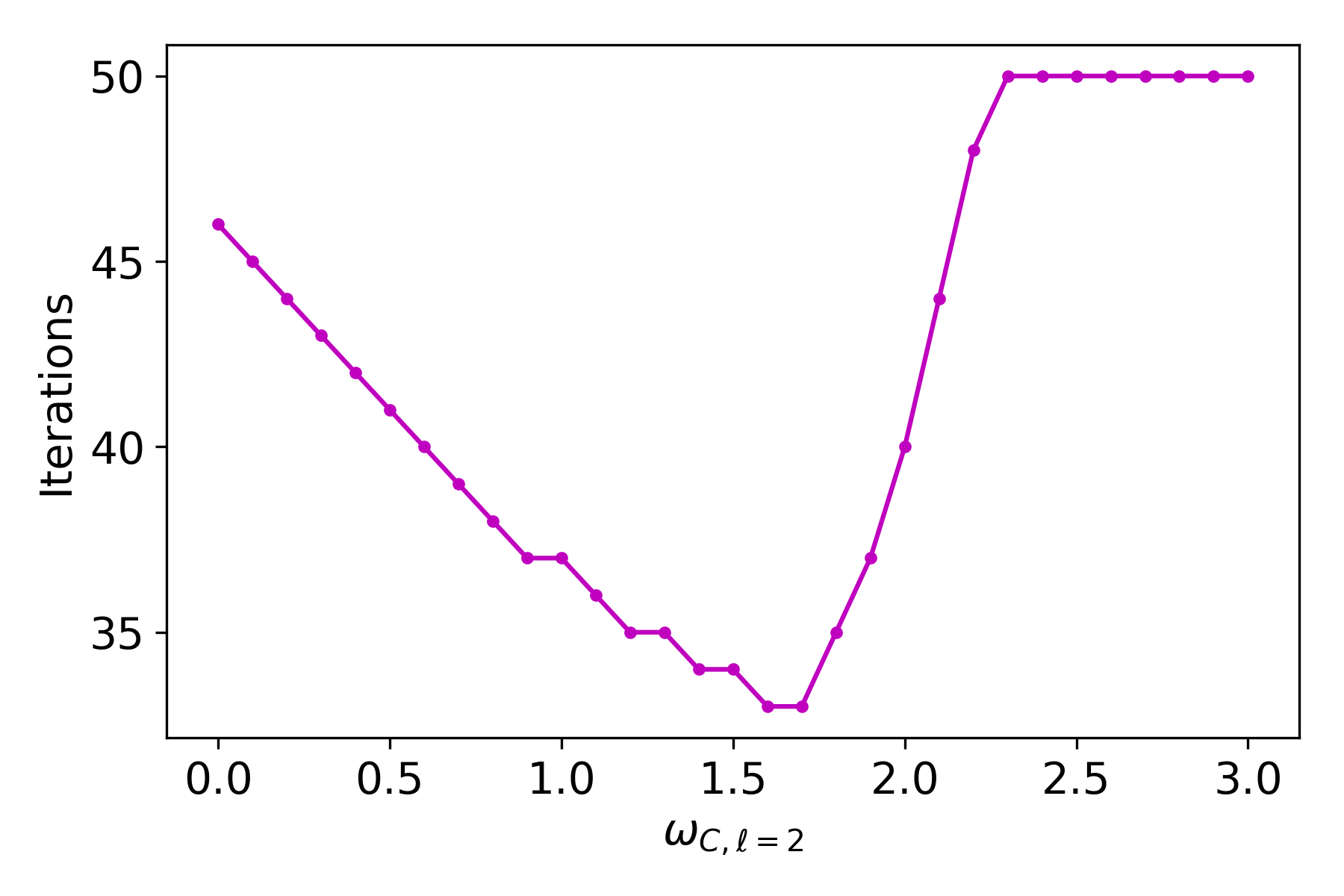}
    \caption{\normalsize Iterations}
    \end{subfigure}
    \caption{Four-level MGRIT convergence rates (left) and iteration counts (right) using FCF-relaxation, as we search 
    for the best level-three relaxation weight $\omega_{C,\ell=2}$, with the fixed values of $(\omega_{C,\ell=0}, \omega_{C,\ell=1})=(1.3, 2.0)$ on the first two levels. The problem is the 1D linear advection equation, coarsening factor $m=2$, and grid size $(N_x, N_t) = (1025, 1025)$. The maximum allowed iterations is 50.}
    \label{fig: Advc1D_C Multilevel Weight Four-level}
\end{figure}

Table \ref{tab:varyweight_LA} depicts the convergence rate and iterations for level dependent weights, 
comparing the experimentally ``best" choice of $(\omega_{C,\ell=0}, \omega_{C,\ell=1}, \omega_{C,\ell=2})=(1.3, 2.0, 1.7)$
against unitary weights and the best uniform weight choice of $\omega_C=1.5$.  Level dependent weights provide only a 
modest improvement in convergence, but it is a larger improvement than observed for the heat equation, where no iterations were 
saved.  Here, only 3 iterations (4.7\%) are saved for $m=2$, when compared to the best uniform weight choice of $\omega_C = 1.5$.
At the bottom of the table, we show how this expensive weight optimization procedure does not carry over to 
another coarsening factor of $m=4$, and instead show that a uniform weight choice of $\omega_C = 1.4$ still provides a 
substantial improvement in convergence.  We conclude that for this problem, level-dependent weights do not 
offer much improvement for convergence and come at the high cost of finding weights.

\begin{table}[h!]
\centering
\begin{tabular}{c r|c|c|c|c}
    & $N_x \times N_t$ & $513 \times 513$ & $1025 \times 1025$ & $2049 \times 2049$ & $4097 \times 4097$ \\ \toprule 
   \multirow{3}{*}{$m=2$}  &   $(\omega_{C,\ell=0}, \omega_{C,\ell=1}, \omega_{C,\ell=2})=(1.0, 1.0, 1.0)$ & 0.562 (31) & 0.670 (43) & 0.749 (60) & 0.788 (72)  \\ 
    & $\textbf{(1.3, 2.0, 1.7)}$ & 0.584 (32) & 0.591 (33) & 0.695 (47) & 0.754 (61) \\ 
    & $(1.5, 1.5, 1.5)$ & 0.485 (24) & 0.609 (35) & 0.710 (51) & 0.764 (64) \\ \midrule 
   \multirow{3}{*}{$m=4$} & $(\omega_{C,\ell=0}, \omega_{C,\ell=1}, \omega_{C,\ell=2})=(1.0, 1.0, 1.0)$ & 0.579 (31) & 0.670 (42) & 0.755 (61) & 0.838 (96) \\ 
    & $(1.3, 2.0, 1.7)$ & 0.545 (28) & 0.673 (44) & 0.794 (76) & 0.983 (>100) \\ 
    & $(1.4, 1.4, 1.4)$ & 0.535 (27) & 0.613 (35) & 0.711 (50) & 0.803 (77) \\ \bottomrule
     \end{tabular}
     \caption{Four-level MGRIT convergence rates (iterations) for the 1D linear advection equation with level-dependent weights.}
\label{tab:varyweight_LA}
\end{table}

\subsubsection{Varying $\delta_t$ experiment}
\label{sec:vary_dt_LA}
Lastly, similar to the heat equation, we explore the question of why weighted relaxation offers a significantly
larger benefit for multilevel MGRIT than for two-level MGRIT (compare Tables 
\ref{tab:LA Conv and Iter for Multi level} and \ref{tab:LA Conv and Iter for Twolevel}).  Thus, we explore 
whether increasing the $\delta_t$ value has a discernible impact on MGRIT convergence.  Table 
\ref{tab:LA_vary_dt} depicts the two-level MGRIT convergence rate for various fine-grid $\delta_t$ values
that mimic the $\delta_t$ values encountered with $m=2$ on coarse MGRIT levels, when a final time of $1.0$ is used and $N_t = 4097$ (i.e., the largest problem in Tables \ref{tab:LA Conv and Iter for Twolevel} and \ref{tab:LA Conv and Iter for Multi level}).  The value
$N_t$ also adapts with $\delta_t$ so that the final time remains unchanged, similar to coarse MGRIT levels, e.g., when $\delta_t$ is 
multiplied by 16 in Table \ref{tab:LA_vary_dt}, $N_t$ decreases by a factor or 16 from 4097 to 257.  The 
table shows that only a weak potential dependence exists between $\delta_t$ and MGRIT convergence, with a slight
improvement in convergence rate as $\delta_t$ increases, but no decrease in iterations.  This leads us to 
believe that a more complicated multilevel interaction is driving the improved benefit of weighted-relaxation 
in the multilevel case.

\begin{table}[h!]
\centering
\begin{tabular}{r | c|c|c|c|c}
     $\delta_t$         & $2.44e^{-4}$ & $2 \cdot 2.44e^{-4}$ & $4 \cdot 2.44e^{-4}$ & $8 \cdot 2.44e^{-4}$ & $16 \cdot 2.44e^{-4}$ \\  \toprule        
     Iterations         & 14    & 14    & 14    & 14    &  14    \\
     Convergence Rate   & 0.285 & 0.284 & 0.282 & 0.280 & 0.274
\end{tabular}
\caption{One-dimensional linear advection equation and two-level MGRIT with $\omega_{C}=1.8$ and $m=2$ for various fine-grid $\delta_t$ values.}
\label{tab:LA_vary_dt}
\end{table}

\subsection{One-dimensional advection equation with grid-dependent dissipation}
\label{sec:1dadv_complex}
The final one-dimensional model problem considered is the one-dimensional advection equation with grid-dependent 
dissipation, which yields complex spatial eigenvalues.  For initial condition $u_0(x)$ and periodic spatial
boundary condition, we have
\begin{align}
\begin{split}
&\frac{\partial u}{\partial t} - \alpha \frac{\partial u}{\partial x} - \epsilon h_x \frac{\partial^2 u}{\partial x^2} = 0, \\
&\alpha > 0, \hspace{10pt} \epsilon > 0, \hspace{10pt} x \in \Omega = [0, L], \hspace{10pt} t \in [0, T], \\
&u(x, 0) = u_0(x), \hspace{10pt} x \in \Omega, \\
&u(0, t) = u(L, t), \hspace{10pt} t \in [0, T] .
\end{split}
\end{align}
By applying standard central differencing for discretizing the spatial derivatives, we obtain the classic first-order upwind difference scheme with $\epsilon = 0.5$. Next, using backward Euler for discretizing the temporal derivative results in
\begin{equation}
\mathbf{u}_j = (I - \delta_t G)^{-1} \mathbf{u}_{j-1}, \hspace{10pt} j=1,2,...,N_t,
\end{equation}
where the linear operator G from (\ref{ODE}) is the two-point upwinding stencil $\frac{\alpha}{h_x} [-1, 1, 0]$. The eigenvalues of G are then computed from the combination of the previously described eigenvalues for the heat equation and linear advection equations (see Sections \ref{sec:results_heat} and \ref{sec:results_adv_imag}, respectively),
 yielding 
 $$\kappa_{\gamma} = \frac{i}{h_x} \sin \left(\frac{2 \pi \gamma}{N_x}\right) - \frac{4 \epsilon}{h_x} \sin^2 \left( \frac{\gamma \pi}{2(N_x + 1)} \right),$$
 for $\gamma = 1,2,...,N_x$.
These values for $\kappa_{\gamma}$ allow for the computation of the theoretical convergence estimate 
(\ref{sharper bd}).  

The same function, domains, and boundary conditions are used as in equations \eqref{eq:LA_eqn1} and 
\eqref{eq:LA_eqn2}.  Likewise, the same MGRIT residual norm tolerance, convergence rate measurements, and maximum iterations are used as in 
Section \ref{sec:results_adv_imag}. The combination of grid points in space $N_x$ and time $N_t$ are chosen 
so that $\frac{\delta_t}{h_x} = 1.0$.

\subsubsection{Weighted FCF- and FCFCF-relaxation}
We again start by considering the two-level method for weighted FCF- and FCFCF-relaxation.  The search
for the experimentally optimal pair of weights for FCFCF-relaxation and $m=2$ is depicted in Figure 
\ref{fig:LA Iter WDiss}, where $(\omega_C, \omega_{CC}) = (2.4, 1.0)$ is the point corresponding to the minimal
convergence rate.  The search space of weights is the same as that for Section 
\ref{sec:results_adv_imag}, $0 \le \omega_C, \omega_{CC} \le 3$,  because 
a more expansive preliminary search indicated this was a reasonable range.

A similar study was done in the thesis \cite{Su2019_v2} for FCF-relaxation 
and found that $\omega_C = 1.9$ is the point where the minimal convergence rate is reached.

Table \ref{tab: AdvcU FCFCF Two Level} depicts the convergence rate and iterations for the two-level case.  The 
experimentally optimal pair of weights found in Figure \ref{fig:LA Iter WDiss} for FCFCF-relaxation $(\omega_C, \omega_{CC}) = (2.4, 1.0)$
is in bold, and this choice leads
to saving 1 iteration, or 11\%, over unitary weights and FCFCF-relaxation on the largest problem.  
The best weight choice for FCF-relaxation 
of $\omega_C = 1.9$ yields only a marginal improvement in convergence and no reduction in iterations when compared to a unitary 
weight and FCF-relaxation on the largest problem.  At the bottom of the table, we examine whether the experimentally optimal weights carry
over to $m=4$ and find that they do not, e.g., $(\omega_C, \omega_{CC}) = (2.4, 1.0)$ is slightly out-performed by
$(\omega_C, \omega_{CC}) = (2.2, 0.5)$.  Additionally, the experimentally best weight for FCF-relaxation and $m=4$ was found to
be 1.7 (not 1.9).   

Table \ref{tab: AdvcU FCFCF Multi Level} repeats these experiments for a full multilevel method.  We see that the 
best two-level choice for FCFCF-relaxation of $(\omega_C, \omega_{CC}) = (2.4, 1.0)$ fails to provide a benefit for larger problems.  Thus,
we carry out another search for FCFCF-relaxation and find that the weights $(\omega_C, \omega_{CC}) = (2.2, 0.5)$ yield the fastest multilevel
convergence when $m=2$, saving 9 iterations, or 22\%, when compared to unitary weights and FCFCF-relaxation on the 
largest problem.  A search in the weight-space for FCF-relaxation yielded the best convergence rate when 
$\omega_C = 1.6$, saving 14 iterations or 21\%, over a unitary weight choice on the largest problem.  At the bottom of the table, we show that the best weight choices for $m=2$ do not carry over
to $m=4$.  We depict the results for an experimentally best weight of 
1.4 for FCF-relaxation in order to show that, curiously, MGRIT with FCF-relaxation performs better for $m=4$ than for $m=2$.

We again note that linear advection is traditionally difficult for MGRIT, so we view this improved convergence 
when using experimentally optimal weights to be an important step. 

\begin{figure}[h!]
    \centering
    \begin{subfigure}[b]{0.4\textwidth}
    \includegraphics[width=\textwidth]{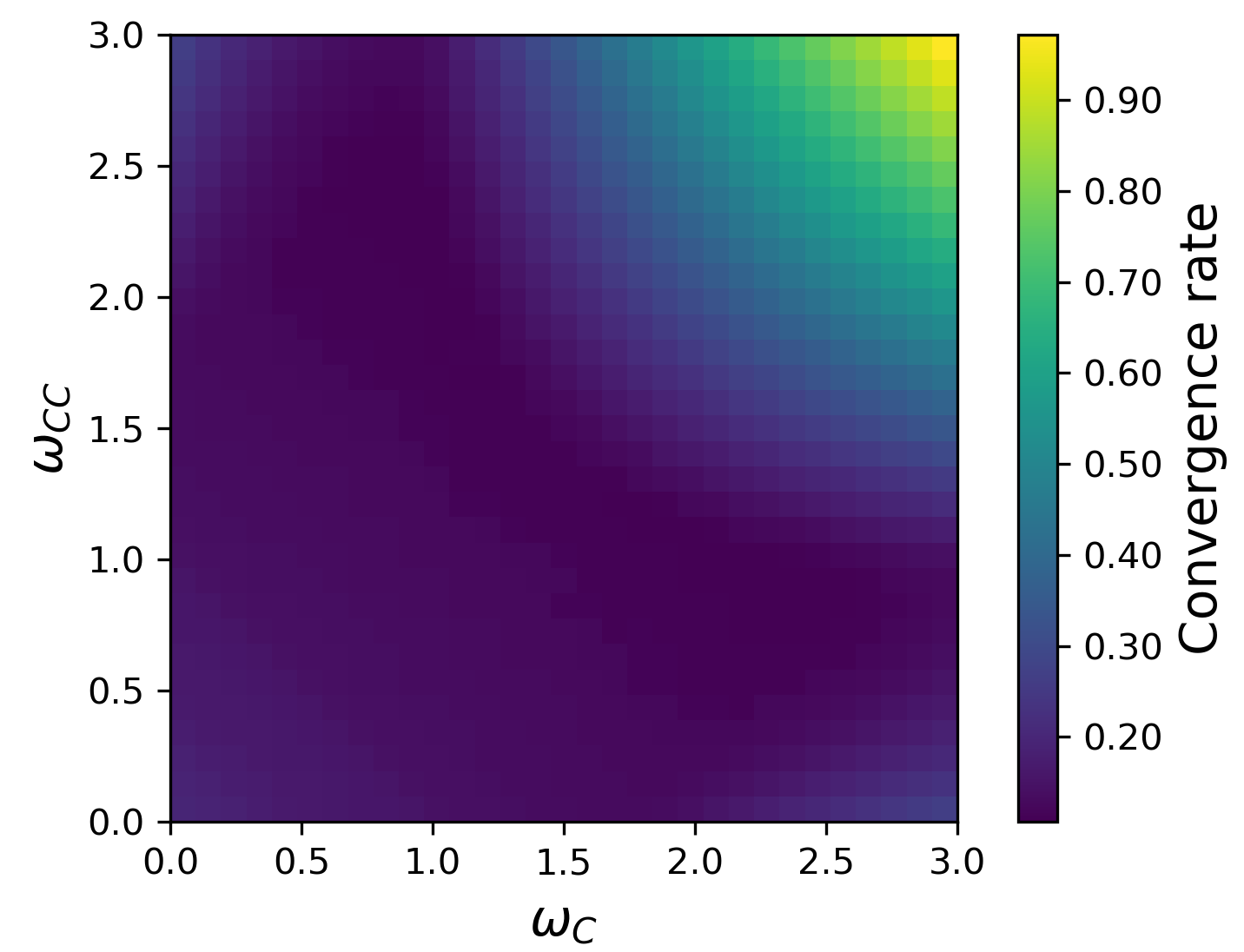}
    \caption{\normalsize Convergence Rate}
    \label{fig:LA Conv WDiss}
    \end{subfigure}
     \begin{subfigure}[b]{0.4\textwidth}
    \includegraphics[width=\textwidth]{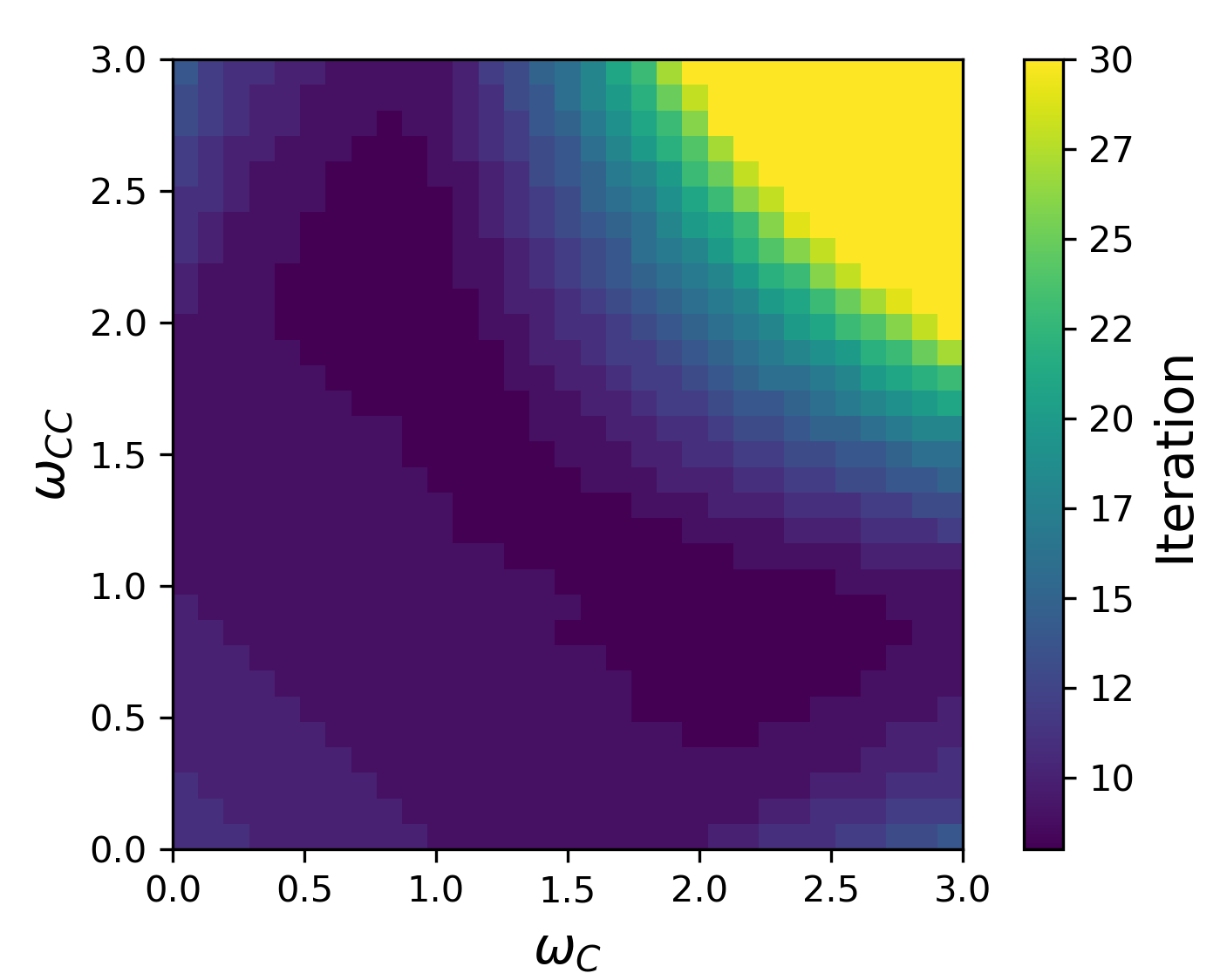}
    \caption{\normalsize Iterations}
    \label{fig:LA Iter WDiss}
    \end{subfigure}
    \caption{Two-level MGRIT experimental convergence rates (left) and iteration counts (right) using FCFCF-relaxation and various 
    relaxation weights $\omega_C$ and $\omega_{CC}$ for the 1D linear advection equation with dissipation, coarsening factor $m=2$, and grid size $(N_x, N_t) = (1025, 1025)$. }
\end{figure}

\begin{table}[h!]
\centering
\begin{tabular}{c r|c|c|c|c}
     
     & $N_x \times N_t$                                           & $513 \times 513$ & $1025 \times 1025$ & $2049 \times 2049$ & $4097 \times 4097$ \\ \toprule
     \multirow{2}{*}{$m=2$} & $\omega_C =1.0$                     & 0.147 (9) & 0.150 (9) & 0.151 (9) & 0.151 (9) \\ 
     & $1.9$                                                      & 0.140 (9) & 0.141 (9) & 0.142 (9) & 0.142 (9) \\ \midrule
     \multirow{3}{*}{$m=2$} & $(\omega_C,\omega_{CC})=(1.0, 1.0)$ & 0.133 (9) & 0.134 (9) & 0.135 (9) & 0.136 (9) \\ 
     & $(2.2, 0.5)$                                               & 0.115 (8) & 0.117 (8) & 0.117 (8) & 0.118 (8) \\ 
     & $\textbf{(2.4, 1.0)}$                                      & 0.114 (8) & 0.115 (8) & 0.116 (8) & 0.116 (8) \\  \midrule
     \multirow{2}{*}{$m=4$} & $\omega_C =1.0$                     & 0.366 (17) & 0.339 (18) & 0.332 (18) & 0.394 (18) \\ 
     & $1.7$                                                      & 0.343 (16) & 0.352 (16) & 0.363 (17) & 0.366 (17) \\  \midrule
     \multirow{4}{*}{$m=4$} & $(\omega_C,\omega_{CC})=(1.0, 1.0)$ & 0.304 (14) & 0.329 (15) & 0.346 (16) & 0.349 (16) \\ 
     & $(1.7, 1.0)$                                               & 0.273 (13) & 0.304 (14) & 0.323 (15) & 0.326 (15) \\ 
     & $(2.2, 0.5)$                                               & 0.314 (15) & 0.323 (15) & 0.330 (15) & 0.338 (16) \\ 
     & $(2.4, 1.0)$                                               & 0.328 (15) & 0.334 (16) & 0.337 (16) & 0.338 (16) \\ \bottomrule
\end{tabular}
\caption{Two-level MGRIT convergence rates (iterations) for the 1D advection equation with dissipation and weighted FCF- and FCFCF-relaxation.}
\label{tab: AdvcU FCFCF Two Level}
\end{table}

\begin{table}[h!]
\centering
\begin{tabular}{c r|c|c|c|c}
    
     & $N_x \times N_t$ & $513 \times 513$ & $1025 \times 1025$ & $2049 \times 2049$ & $4097 \times 4097$ \\   \toprule
     \multirow{2}{*}{$m=2$} & $\omega_C =1.0$                     & 0.438 (21) & 0.560 (30) & 0.667 (43) & 0.772 (66) \\ 
     & $1.6$                                                      & 0.388 (18) & 0.488 (23) & 0.613 (35) & 0.719 (52) \\ \midrule
     \multirow{4}{*}{$m=2$} & $(\omega_C,\omega_{CC})=(1.0, 1.0)$ & 0.344 (16) & 0.432 (21) & 0.559 (29) & 0.660 (41) \\ 
     & $(1.6, 1.0)$                                               & 0.293 (14) & 0.412 (20) & 0.520 (26) & 0.638 (38) \\ 
     & $\textbf{(2.2, 0.5)}$                                      & 0.295 (14) & 0.363 (17) & 0.482 (24) & 0.585 (32) \\ 
     & $(2.4, 1.0)$                                               & 0.388 (19) & 0.564 (32) & 0.725 (53) & 0.834 (94) \\  \midrule
      \multirow{2}{*}{$m=4$} & $\omega_C =1.0$                    & 0.428 (20) & 0.549 (28) & 0.657 (40) & 0.746 (57) \\ 
     & $1.4$                                                      & 0.375 (18) & 0.496 (24) & 0.607 (34) & 0.694 (46) \\  \midrule
     \multirow{4}{*}{$m=4$} & $(\omega_C,\omega_{CC})=(1.0, 1.0)$ & 0.336 (16) & 0.449 (21) & 0.562 (29) & 0.677 (43) \\ 
     & $(1.4, 1.0)$                                               & 0.301 (14) & 0.416 (20) & 0.542 (28) & 0.653 (39) \\ 
     & $(2.2, 0.5)$                                               & 0.454 (22) & 0.582 (31) & 0.682 (44) & 0.712 (49) \\ 
     & $(2.4, 1.0)$                                               & 0.404 (19) & 0.559 (30) & 0.672 (42) & 0.689 (45) \\  \bottomrule
\end{tabular}
\caption{Multilevel MGRIT convergence rates (iterations) for the 1D linear advection equation with dissipation and weighted FCF- and FCFCF-relaxation.}
\label{tab: AdvcU FCFCF Multi Level}
\end{table}

\begin{remark}
To avoid repetition, we omit our experiments for level-dependent weights and for varying $\delta_t$, 
because the results are similar to that seen in Sections \ref{sec:ml_w_LA} and \ref{sec:vary_dt_LA} for the linear
advection equation with purely imaginary spatial eigenvalues.  That is, optimized level-dependent weights saved 2 iterations, or 7\%,
in the four-level setting and FCF-relaxation, and little MGRIT dependence on the size of $\delta_t$ was found.
\end{remark}

\end{document}